\documentclass[11pt]{article}

\usepackage{latexsym}
\usepackage{amssymb}
\usepackage{amsthm}
\usepackage{amscd}
\usepackage{amsmath}
\usepackage{mathrsfs}
\usepackage{graphicx}
\usepackage{shuffle}
\usepackage{hyperref}
\usepackage{mathtools}
\usepackage[bbgreekl]{mathbbol}
\usepackage{mathdots}
\usepackage{ytableau}
\usepackage{tikz-cd}

\usepackage{stmaryrd}

\usepackage[colorinlistoftodos]{todonotes}
\usetikzlibrary{chains,scopes}
\usetikzlibrary{shapes.geometric,positioning}

\DeclareSymbolFontAlphabet{\mathbb}{AMSb}
\DeclareSymbolFontAlphabet{\mathbbl}{bbold}

\newcommand{\RPP}{\ensuremath\mathrm{PP}}
\newcommand{\MRPP}{\ensuremath\mathsf{ShPP}}

\numberwithin{equation}{section}
\allowdisplaybreaks[1]

\theoremstyle{definition}
\newtheorem* {theorem*}{Theorem}
\newtheorem* {conjecture*}{Conjecture}
\newtheorem{theorem}{Theorem}[section]

\theoremstyle{definition}

\newtheorem* {example*}{Example}

\newtheorem{lemma}[theorem]{Lemma}
\theoremstyle{definition}
\newtheorem{definition}[theorem]{Definition}
\theoremstyle{definition}

\newtheorem{proposition}[theorem]{Proposition}
\newtheorem{corollary}[theorem]{Corollary}

\newtheorem*{remark*}{Remark}
\theoremstyle{definition}

\theoremstyle{definition}
\newtheorem {example}[theorem]{Example}
\theoremstyle{definition}

\theoremstyle{definition}

\theoremstyle{definition}

\def\({\left(}
\def\){\right)}

\newcommand{\QQ}{\mathbb{Q}}
\newcommand{\cP}{\mathcal{P}}
\newcommand{\cQ}{\mathcal{Q}}

\newcommand{\cO}{\mathcal{O}}
\newcommand{\cR}{\mathcal{R}}

\def\ZZ{\mathbb{Z}}

\def\spanning{\textnormal{-span}}

\def\barr{\begin{array}}
\def\earr{\end{array}}
\def\ba{\begin{aligned}}
\def\ea{\end{aligned}}
\def\be{\begin{equation}}
\def\ee{\end{equation}}

\def\qquand{\qquad\text{and}\qquad}
\def\quand{\quad\text{and}\quad}
\def\qquord{\qquad\text{or}\qquad}
\def\quord{\quad\text{or}\quad}

\def\cH{\mathcal H}

\def\hs{\hspace{0.5mm}}

\def\ben{\begin{enumerate}}
\def\een{\end{enumerate}}

\def\hs{\hspace{0.5mm}}

\def\D{\hat D}

\newcommand{\cA}{\mathcal{A}}
\newcommand{\cB}{\mathcal{B}}

\def\arcstart{\ \xy<0cm,-.15cm>\xymatrix@R=.1cm@C=.3cm }
\newcommand{\arcstartc}[1]{\ \xy<0cm,-.15cm>\xymatrix@R=.1cm@C=#1cm}

\def\cF{\mathcal{F}}
\def\cG{\mathcal{G}}

\newcommand{\weight}{\operatorname{wt}}

\def\cW{\mathcal{W}}

\def\D{\textsf{D}}
\def\SD{\textsf{SD}}

\def\GQ{G\hspace{-0.2mm}Q}
\def\GP{G\hspace{-0.2mm}P}

\def\gq{g\hspace{-0.1mm}q}
\def\gp{g\hspace{-0.1mm}p}

\def\jq{j\hspace{-0.1mm}q}
\def\jp{j\hspace{-0.1mm}p}

\def\ss{/\hspace{-1mm}/}

\def\cV{\mathcal{V}}

\newcommand{\cC}{\mathcal{C}}

\usepackage{fullpage}

\newcommand{\ytab}[1]{
\ytableausetup{boxsize = .55cm,aligntableaux=center}
{\small\begin{ytableau}  #1  \end{ytableau}}
}

\newcommand{\ytabsmall}[1]{
\ytableausetup{boxsize = .3cm,aligntableaux=center}
{\small\begin{ytableau}  #1  \end{ytableau}}
}

\def\SetTab{\mathsf{SVT}}
\def\ShSetTab{\mathsf{ShSVT}}

\def\ss{/\hspace{-1mm}/}

\def\BT{\mathsf{BT}}
\def\ShBTQ{\mathsf{ShBT}_Q}
\def\ShBTP{\mathsf{ShBT}_P}

\def\cq{\mathsf{q\_corners}}
\def\cp{\mathsf{p\_corners}}

\def\tauShBT{\tau^{\mathrm{ShBT}}}

\def\wtjp{\jp^{\mathrm{comb}}}
\def\wtjq{\jq^{\mathrm{comb}}}
\def\wtgp{\gp^{\mathrm{comb}}}
\def\wtgq{\gq^{\mathrm{comb}}}

\def\cribbonsym{\mathsf{ShRibbons}_Q}
\newcommand{\cribbons}[2]{\cribbonsym(#1\ss#2)}
\newcommand{\bribbons}[2]{\mathsf{ShRibbons}_P(#1\ss#2)}
\def\RC{\mathsf{RC}}

\def\cp{\mathfrak{p}}
\def\cq{\mathfrak{q}}

\def\cpp{\tilde{\mathfrak{p}}}
\def\cqq{\tilde{\mathfrak{q}}}

\def\cU{\mathcal{U}}

\def\PPart{\mathsf{PP}}

\def\yyyletter{\mathfrak{b}}
\def\zzzletter{\mathfrak{c}}
\newcommand{\yyy}[3]{ \yyyletter^{#1}_{#2,(#3)}}
\newcommand{\zzz}[3]{ \zzzletter^{#1}_{#2,(#3)}}

\usepackage{young}

\definecolor{darkred}{rgb}{0.7,0,0} 
\newcommand{\defn}[1]{{\color{darkred}\emph{#1}}} 

\begin{document}
\title{Combinatorial formulas for shifted dual stable Grothendieck polynomials}
\author{
Joel Brewster Lewis \\ Department of Mathematics  \\ George Washington University \\ {\tt jblewis@gwu.edu}
\and
Eric Marberg \\ Department of Mathematics \\  HKUST \\ {\tt emarberg@ust.hk}
}

\date{}

\maketitle

\begin{abstract}
The $K$-theoretic Schur $P$- and $Q$-functions $\GP_\lambda$ and $\GQ_\lambda$ may be concretely defined 
as weight generating functions for semistandard shifted set-valued tableaux.
These symmetric functions are the shifted analogues of stable Grothendieck polynomials, and were introduced by Ikeda and Naruse 
for applications in geometry. 
Nakagawa and Naruse specified families of dual $K$-theoretic Schur $P$- and $Q$-functions 
$\gp_\lambda$ and $\gq_\lambda$ via a Cauchy identity involving  $\GP_\lambda$ and $\GQ_\lambda$.
They conjectured that the dual power series are weight generating functions for certain shifted plane partitions. 
We prove this conjecture. We also derive a related generating function formula for the images of $\gp_\lambda$ and $\gq_\lambda$
under the $\omega$ involution of the ring of symmetric functions. This confirms a conjecture of Chiu and the second author.
Using these results, we verify a conjecture of Ikeda and Naruse that the $\GQ$-functions are a basis for a ring. 
\end{abstract}

\tableofcontents

\section{Introduction}

The main results of this paper are to derive explicit combinatorial generating functions
for certain families of ``dual'' power series defined indirectly by Cauchy identities.
The formulas that we establish were originally conjectured in \cite{ChiuMarberg,NakagawaNaruse}.
This introduction gives a very quick summary of the power series involved and
the generating functions to be derived. We also explain one application of our formulas 
to resolve a conjecture of Ikeda and Naruse from \cite{IkedaNaruse}.

\subsection{Shifted set-valued generating functions}

The \defn{shifted Young diagram}
of a strict integer partition $\lambda = (\lambda_1 > \lambda_2   > \dots> \lambda_k>0)$
is the set 
\[\SD_\lambda := \{ (i,i+j-1) \in [k]\times \ZZ : 1 \leq j \leq \lambda_i\}\quad\text{where $[k] := \{1,2,\dots,k\}$.}\]
Elements of $\SD_\lambda$ are called \defn{positions} or \defn{boxes}.
A \defn{shifted set-valued tableau} of shape $\lambda$
is a filling  $T $ of $\SD_{\lambda}$ by finite, nonempty subsets of $\tfrac{1}{2}\ZZ$.
Throughout, we let
$ i' := i - \tfrac{1}{2}$  for $i \in \ZZ$
and refer to half-integers as \defn{primed numbers}. 

Let $T_{ij}$ denote the entry assigned by  $T$ to box  $(i,j) \in \SD_{\lambda}$,
and write $(i,j) \in T$ to indicate that $(i,j) $ is in the domain of $T$.
The \defn{diagonal positions} of  $T$ are the boxes $(i,j) \in T$ with $i=j$.
A shifted set-valued tableau $T$ is \defn{semistandard} if all of the following conditions hold:
\begin{itemize}
\item[(S1)] its entries $T_{ij}$ are nonempty finite subsets of $\{1'<1<2'<2<\dots\}$,
\item[(S2)] one has $\max(T_{ij}) \leq \min(T_{i+1,j})$ and $\max(T_{ij}) \leq \min(T_{i,j+1})$ for all relevant $(i,j) \in T$,
\item[(S3)] no unprimed number appears in different boxes within the same column, and
\item[(S4)] no primed number appears in different boxes within the same row.
\end{itemize}
We draw shifted tableaux in French notation; for example, both
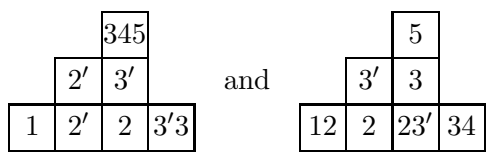
\[
\ytableausetup{boxsize=0.6cm,aligntableaux=center}
 \begin{ytableau}
\none & \none & 345 & \none\\
\none &  2'  & 3' & \none\\
1 & 2' & 2 & 3'3
\end{ytableau}
\quand
 \begin{ytableau}
\none & \none & 5 & \none\\
\none &  3'  & 3 & \none\\
12 & 2 & 23' & 34
\end{ytableau}
\]
are semistandard shifted set-valued tableaux of shape $(4, 2, 1)$.
Let $\beta, x_1,x_2,x_3,\dots$ be commuting indeterminates. 
Define $|T| := \sum_{(i,j) \in T} |T_{ij}|$ and ${\bf x}^T := \prod_i x_i^{a_i+b_i}$
 where $a_i$ and $b_i$ are the number of times that 
$i$ and $i'$ appear in the shifted set-valued tableau $T$, respectively. 
Our examples above both have $|T| = 10$ and ${\bf x}^T = x_1 x_2^3 x_3^4 x_4 x_5$. 
When $\lambda$ is a partition (or more generally, any sequence with finite sum), we write $|\lambda|$ for the sum of its entries.

\begin{definition}
The \defn{$K$-theoretic Schur $P$- and $Q$-functions} indexed by a strict partition $\lambda$ are the formal power series
$ \GP_{\lambda} := \sum_{T \in \ShSetTab_P(\lambda)} \beta^{|T|-|\lambda|} {\bf x}^{T}$
and
$\GQ_{\lambda} := \sum_{T \in \ShSetTab_Q(\lambda)} \beta^{|T|-|\lambda|} {\bf x}^{T}$
where $\ShSetTab_Q(\lambda)$ 
is the  set of all semistandard shifted set-valued tableaux of shape $\lambda$ and 
 $\ShSetTab_P(\lambda)$ is the subset of such  tableaux with
no \textbf{primed} numbers in any diagonal positions.\footnote{
The functions $\GP_\lambda$ and $\GQ_\lambda$  are sometimes defined  
by these formulas with $\beta$ set to $-1$ or $1$.
There is no loss of generality in this since we can recover  
$ \GP_{\lambda} $ from $ \sum_{T \in \ShSetTab_P(\lambda)} (-1)^{|T|-|\lambda|} {\bf x}^{T}$
by substituting $x_i\mapsto  -\beta x_i$ then dividing by $\beta^{|\lambda|}$ (and likewise for $\GQ_\lambda$).
Similar comments apply to all other power series in this paper involving $\beta$.
}
 \end{definition}

These expressions belong to ring $\ZZ[\beta]\llbracket x_1,x_2,\dots\rrbracket  $. 
If $\deg \beta=-1$  then $\GP_\lambda$ and $\GQ_\lambda$ are   homogeneous of degree $|\lambda|$,
but if $\deg \beta=0$ then the power series have unbounded degree.
Both $\GP_\lambda$ and $\GQ_\lambda$ are symmetric in the $x_i$ 
variables  \cite[Thm.~9.1]{IkedaNaruse}.
Setting $\beta=0$ turns $\GP_\lambda$ and $\GQ_\lambda$ into the classical \defn{Schur $P$- and $Q$-functions} $P_\lambda$ and $Q_\lambda$.
It follows that as
 $\lambda$ ranges over all strict partitions,
the sets  $\{ \GP_\lambda\}$ and $\{\GQ_\lambda\}$ are linearly independent.
While $Q_\lambda=2^{\ell(\lambda) }P_\lambda$,
each $\GQ_\lambda$ is a  more complicated but still finite $\ZZ[\beta]$-linear combination of $\GP_\mu$'s \cite[Thm.~1.1]{ChiuMarberg}.

Ikeda and Naruse first introduced these symmetric functions  in \cite{IkedaNaruse}
for applications in geometry.
Specializations of $\GP_\lambda$ and $\GQ_\lambda$ represent the structure sheaves of Schubert varieties in the torus equivariant $K$-theory of the maximal isotropic Grassmannians of orthogonal and symplectic types \cite[Cor.~8.1]{IkedaNaruse}; see also  \cite{NakagawaNaruse2017,NakagawaNaruse,Naruse}. 
These power series further appear as ``stable limits'' of $K$-theory classes of
 certain orbit closures in the type A flag variety \cite{MP2020,MP2021}.

\subsection{Dual functions via Cauchy identities}\label{cauchy-intro-sect}
 
Our main results concern the following dual forms of $\GP_\lambda$ and $\GQ_\lambda$.

\begin{definition}
The \defn{dual $K$-theoretic Schur $P$- and $Q$-functions} $\gp_\lambda$ and $\gq_\lambda$
are the unique elements of $\ZZ[\beta]\llbracket  x_1,x_2,\dots\rrbracket  $ indexed by strict partitions $\lambda$ 
satisfying the Cauchy identities
 \be\label{cauchy-eq} \sum_\lambda \GQ_\lambda({\bf x}) \gp_\lambda({\bf y}) = \sum_\lambda \GP_\lambda({\bf x}) \gq_\lambda({\bf y}) = \prod_{i,j \geq 1} \tfrac{ 1 - \overline{x_i} y_j}{1-x_iy_j}
 \quad\text{where }\overline{x} := \tfrac{-x}{1+\beta x}
 .\ee
  \end{definition}

The power series
$\gp_{\lambda}$
and
$\gq_{\lambda}$
are special cases of Nakagawa and Naruse's \defn{dual universal factorial Schur $P$- and $Q$-functions},
which are defined via a more general
version of \eqref{cauchy-eq} \cite[Def.~3.2]{NakagawaNaruse}.
Both $\{\gp_\lambda\}$ and $\{\gq_\lambda\}$ are  linearly independent families  of  functions that are symmetric in the $x_i$ variables and  
homogeneous if $\deg \beta = 1$ \cite[Thm.~3.1]{NakagawaNaruse}. 
 These properties let one define
the following conjugate symmetric functions, which were first considered in \cite{ChiuMarberg}:

 \begin{definition}
   Write $\omega$ for the $\ZZ[\beta]$-linear involution of the ring of symmetric functions (in the variables $x_1,x_2,\dots$ with coefficients in $\ZZ[\beta]$)
   acting on Schur functions as $\omega(s_\lambda) = s_{\lambda^\top}$.
 The \defn{conjugate dual $K$-theoretic Schur $P$- and $Q$-functions} of a strict partition $\lambda$ are the   symmetric functions
 $ \jp_\lambda:= \omega(\gp_\lambda)$ and $ \jq_\lambda:= \omega(\gq_\lambda).$
 \end{definition}

\subsection{Shifted plane partition generating functions}

Our first main result is a generating function formula for $\gp_\lambda$ and $\gq_\lambda$ 
that was predicted in \cite{NakagawaNaruse}.
Let  $ \lambda$ be a strict partition.
A \defn{shifted plane partition} of shape $\lambda$
is
a filling 
of $\SD_{\lambda}$ by elements of $ \{1'<1<2'<2<\dots\}$ with weakly increasing rows and columns. Examples include 
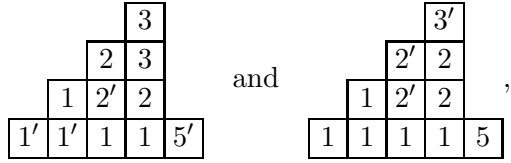
\[
\ytableausetup{boxsize=0.5cm,aligntableaux=center}
\begin{ytableau}
\none &\none & \none & 3  \\
\none &\none & 2 & 3 \\
\none & 1& 2' & 2 \\
1' &1'  & 1& 1 & 5' 
\end{ytableau}
\quand
\begin{ytableau}
\none &\none & \none & 3'  \\
\none &\none & 2' & 2 \\
\none &1 & 2' & 2 \\
1  & 1 & 1&1& 5 
\end{ytableau},
\]
which both have shape $(5,3,2,1)$.
Given such a filling $T$, let
 $c_i$ be the number of distinct columns of $T$ containing $i$
and let $r_i$ be the number of distinct rows of $T$ containing  $i'$. Then define
\[
\weight_\RPP(T) := (c_1 + r_1, c_2 + r_2, c_3 + r_3, \dots)
\quand
\textstyle{\bf x}^{\weight_\RPP(T)} :=   \prod_{i\geq 1} x_i^{c_i + r_i}.
\] 
Both examples above  have $|\weight_\RPP(T)| =9$ and ${\bf x}^{\weight_\RPP(T)} = x_1^4 x_2^3x_3x_5$.

\begin{theorem}\label{main-thm1}
Let $\MRPP_Q(\lambda)$ be the set of all
shifted plane partitions of shape $\lambda$,
and  define
$\MRPP_P(\lambda)$ 
to be the subset of  such fillings
with no \textbf{unprimed} diagonal entries. Then
\[
\ba
\gp_{\lambda} = \sum_{T\in \MRPP_P(\lambda)} (-\beta)^{|\lambda| - |\weight_\RPP(T)|} {\bf x}^{\weight_\RPP(T)} 
\quand
\gq_{\lambda} = \sum_{T\in \MRPP_Q(\lambda)} (-\beta)^{|\lambda| -  |\weight_\RPP(T)|} {\bf x}^{\weight_\RPP(T)} .\ea
\]
\end{theorem}

 This result was conjectured by Nakagawa and Naruse as \cite[Conj.~5.1]{NakagawaNaruse}.
We will actually prove a more general formula for the skew versions of $\gp_\lambda$ and $\gq_\lambda$; see Theorem~\ref{gp-gq-thm}.
These formulas make it possible to compute the terms in $\gp_\lambda$ and $\gq_\lambda$,
which is not straightforward from \eqref{cauchy-eq}.

\subsection{Shifted bar tableaux generating functions}

Our second main result is a generating function formula for $\jp_\lambda$ and $\jq_\lambda$ 
that was predicted in \cite{ChiuMarberg}.
Continue to let $\lambda$ be a strict integer partition.
Suppose $V$ is a shifted tableau\footnote{That is, a shifted set-valued tableau whose entries are all singleton sets.} of shape $\lambda$
with no unprimed entries repeated in any column and no primed entries repeated in any row.
Let $\Pi$ be a partition of the diagram $\SD_{\lambda}$
into (disjoint, nonempty) subsets of adjacent boxes containing the same entry in $V$. 
Each block of $ \Pi$ 
is a contiguous ``bar'' of positions in the same row or  column,
and we refer to the pair $T=(V,\Pi)$ as
a \defn{shifted bar tableau} of shape $\lambda$.

If $V$ is semistandard in the sense of having weakly increasing rows and columns,
then we say that $T$ is also \defn{semistandard}. 
We draw shifted bar tableaux as pictures like
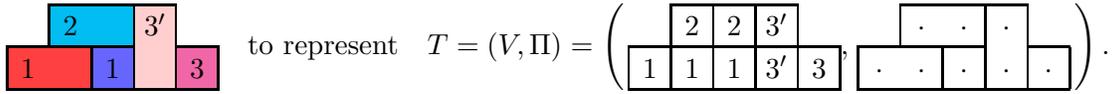
\[
\begin{young}[15.5pt][c] 
, & ]=![cyan!75]2  & =]![cyan!75]  \ynobottom & ![pink!75]3'\ynobottom  \\ 
]=![red!75]1 & =]![red!75]   & ]=]![blue!60] 1 & ]=]![pink!75] \ynotop & ![magenta!75]3 
\end{young}
\quad\text{to represent}\quad
T=(V,\Pi) =  \(\hs \begin{young}[15.5pt][c]
 , &  2 &  2 &  3'\\
1 & 1 &   1 &   3' &   3
\end{young}
,\hs \begin{young}[15.5pt][c] 
, & ]= \cdot & =] \cdot \ynobottom & \cdot\ynobottom  \\ 
]= \cdot & =]\cdot & ]=]\cdot & ]=] \cdot \ynotop &  \cdot
\end{young}\hs \).
\]
These objects are the shifted analogues of what are called \defn{valued-set tableaux} in \cite{LamPyl}.
The word ``valued-set'' is just a formal transposition of ``set-valued'';
we believe that the name ``bar tableau'' is more intuitive and descriptive.

Given a shifted bar tableau $T=(V,\Pi)$, let
$ |T| := |\Pi|$ and $ {\bf x}^T := \prod_{i\geq 1} x_i^{b_i}$
where $b_i$ is the number of blocks in $\Pi$ containing $i$ or $i'$.
Our example above has $|T| = 5$ and ${\bf x}^T = x_1^2 x_2 x_3^2$.

\begin{theorem}\label{main-thm2}
Let $\ShBTQ(\lambda)$ denote the set of all semistandard shifted bar tableaux of shape $\lambda$
and let $\ShBTP(\lambda)$
be the subset of such tableaux with no \textbf{primed} diagonal entries.
Then
\[
\ba
 \jp_{\lambda} = \sum_{T\in \ShBTP(\lambda)} (-\beta)^{|\lambda|-|T|} {\bf x}^{T} 
 \quand
\jq_{\lambda} = \sum_{T\in \ShBTQ(\lambda)} (-\beta)^{|\lambda| -|T|} {\bf x}^{T} .
\ea
\]
\end{theorem}

This result was conjectured by Chiu and the second author as \cite[Conj.~7.2]{ChiuMarberg}.
We will actually prove a more general formula for the skew versions of $\jp_\lambda$ and $\jq_\lambda$; see Theorems~\ref{jq-thm} and \ref{jp-thm}.

\subsection{Application to conjectures of Ikeda and Naruse}

Consider the modules 
consisting of all infinite $\ZZ[\beta]$-linear combinations of 
the functions $\{ \GP_\lambda\}$ and $\{\GQ_\lambda\}$, with $\lambda$ ranging over all strict partitions.
Ikeda and Naruse proved that these modules are both rings \cite[Props.~3.4 and 3.5]{IkedaNaruse}.
Concretely, this means that  $\GP_\lambda \GP_\mu$ (respectively, $\GQ_\lambda \GQ_\mu$)
always expands as a (possibly infinite) $\ZZ[\beta]$-linear combination of $\GP_\nu$'s (respectively, $\GQ_\nu$'s).

Ikeda and Naruse conjectured that these expansions are actually finite \cite[Conj.~3.1 and 3.2]{IkedaNaruse}.
For the $\GP$-functions,
this stronger ring property was established by Clifford, Thomas, and Yong in \cite{CTY}; other proofs have subsequently been given in 
\cite[\S4]{HKRWZZ}, 
\cite[\S1.2]{M2021}, and \cite[\S8]{PechenikYong}.
The problem of showing the same property for the $\GQ$-functions appears to still be open. 
Building on \cite{ChiuMarberg}, 
we are able to resolve this problem:

\begin{theorem}\label{consequence-thm}
The free $\ZZ[\beta]$-modules with bases $\{ \GP_\lambda\}$ and $\{\GQ_\lambda\}$ (where $\lambda$ ranges over all strict partitions)
are subrings of $\ZZ[\beta]\llbracket  x_1,x_2,\dots\rrbracket  $. In particular, 
if $\lambda$ and $\mu$ are strict partitions then 
\[
\GP_\lambda \GP_\mu = \sum_\nu a_{\lambda\mu}^\nu  \beta^{|\nu| - |\lambda|-|\mu|} \GP_\nu
\quand
\GQ_\lambda \GQ_\mu = \sum_\nu b_{\lambda\mu}^\nu  \beta^{|\nu| - |\lambda|-|\mu|} \GQ_\nu
\] 
for integer coefficients $a_{\lambda\mu}^\nu$ and $ b_{\lambda\mu}^\nu$ that 
are (a)  nonzero for only finitely many strict partitions $\nu$, (b) zero whenever $|\nu| < |\lambda| + |\mu|$ or $\ell(\nu) > \ell(\lambda) + \ell(\mu)$, and (c) nonnegative.
\end{theorem}

\begin{proof}
The ring property for the span of the $\GP$-functions follows from \cite{CTY}.
Meanwhile,
\cite[Cor.~7.7]{ChiuMarberg}  is exactly the assertion that the ring property for 
the span of the $\GQ$-functions
follows from Theorem~\ref{main-thm2} (or more precisely from its skew version, which is Theorem~\ref{jp-thm}).  These properties assert that the products $\GP_\lambda \GP_\mu$ and $\GQ_\lambda \GQ_\mu$ expand in the respective bases with coefficients in $\ZZ[\beta]$ satisfying (a).  Taking $\deg \beta = -1$, the homogeneity of the $\GP$- and $\GQ$-functions implies that in fact the coefficient of $\GP_\nu$ (respectively, $\GQ_\nu$) must be of the form $a_{\lambda\mu}^\nu\beta^{|\nu| - |\lambda| - |\mu|}$ (respectively,  $b_{\lambda\mu}^\nu\beta^{|\nu| - |\lambda| - |\mu|}$) for an integer $a_{\lambda\mu}^\nu$ (respectively, $b_{\lambda\mu}^\nu$), and moreover that $a_{\lambda\mu}^\nu= b_{\lambda\mu}^\nu= 0$ whenever $|\nu| < |\lambda| + |\mu|$.
Given our main results (namely, Theorems~\ref{jq-thm} and \ref{jp-thm}), \cite[Thm.~7.6]{ChiuMarberg} asserts that 
$a_{\lambda\mu}^\nu= b_{\lambda\mu}^\nu=0$ whenever $\ell(\nu) > \ell(\lambda) + \ell(\mu)$.

Given Ikeda and Naruse's interpretation of $\GP_\lambda$ and $\GQ_\lambda$ as $K$-theory representatives for Schubert varieties in the orthogonal and Lagrangian Grassmannians, the fact that  $a_{\lambda\mu}^\nu$ and $ b_{\lambda\mu}^\nu$ are nonnegative 
can be deduced from a result of Brion \cite{Brion}.
We only explain this for the $b_{\lambda\mu}^\nu$ coefficients, since 
 $a_{\lambda\mu}^\nu \geq 0$ follows directly from \cite[Thm.~1.2]{CTY}.

As summarized in \cite[\S8.1]{IkedaNaruse}, 
for each strict partition $\lambda \subseteq (n,\dots,3,2,1)$ there is an associated \defn{Schubert variety} $\Omega_\lambda$
in the \defn{Lagrangian Grassmannian} $\cG_n := LG(n)$.
The variety $\Omega_\lambda$ is closed with codimension $|\lambda|$, and its 
structure sheaf $\cO_{\Omega_\lambda}$ defines a class $[\cO_{\Omega_\lambda}]$ in the $K$-theory ring $K(\cG_n)$ 
of coherent sheaves on $\cG_n$. These classes are an additive basis for $K(\cG_n)$,
and so we have 
 $ [\cO_{\Omega_\lambda}]  [\cO_{\Omega_\mu}] =  \sum_\nu (-1)^{|\nu| - |\lambda|-|\mu|}   c_{\lambda\mu}^\nu [\cO_{\Omega_\nu}]$
 for some integers $c_{\lambda\mu}^\nu \in \ZZ$.
The main result of \cite{Brion} shows that the coefficients $c_{\lambda\mu}^\nu$ are nonnegative. (In fact, the main result of \cite{Brion}
is a more general statement that applies  to Schubert varieties in any complex flag variety.) 

If we set $\beta=-1$ and define
 $\mathit{G\Gamma}_{n,+}$ to be the $\ZZ$-span of all polynomials $\GQ_\lambda(x_1,\dots,x_n)$ with $\ell(\lambda) \leq n$,
then  there is a surjective ring homomorphism 
$\mathit{G\Gamma}_{n,+} \to K(\cG_n)$ 
sending $\GQ_\lambda(x_1,\dots,x_n) $ to $ [\cO_{\Omega_\lambda}]$ if $\lambda \subseteq (n,\dots,3,2,1)$
and to zero otherwise \cite[Cor.~8.1]{IkedaNaruse}.
If $n$ is large enough that $(n, \ldots, 3, 2, 1)$ contains both $\lambda$ and $\mu$ as well as the finite number of strict partitions $\nu$ with $b_{\lambda\mu}^\nu \neq 0$, then it follows that 
 $ [\cO_{\Omega_\lambda}]  [\cO_{\Omega_\mu}] =  \sum_\nu   (-1)^{|\nu| - |\lambda|-|\mu|} b_{\lambda\mu}^\nu [\cO_{\Omega_\nu}]$,
and so $b_{\lambda\mu}^\nu = c_{\lambda\mu}^\nu \geq 0$.
\end{proof}

\begin{remark*}
The sets $\{ \GP_\lambda(x_1,\dots,x_n) : \ell(\lambda) \leq n\}$ and 
$\{ \GQ_\lambda(x_1,\dots,x_n) : \ell(\lambda) \leq n\}$ of polynomials in $\beta,x_1,\dots,x_n$ are linearly independent over $\ZZ[\beta]$
by \cite[Thm.~3.1 and Prop.~3.2]{IkedaNaruse}, and the expansion of $\GQ_\lambda(x_1,\dots,x_n) \GQ_\mu(x_1,\dots,x_n)$
into $\GQ_\nu(x_1,\dots,x_n)$'s can be calculated by a finite linear algebra computation.  Theorem~\ref{consequence-thm} implies that if $\ell(\lambda) + \ell(\mu) \leq n$ then the 
same finite computation completely determines the expansion of $\GQ_\lambda\GQ_\mu$ into $\GQ_\nu$'s,
and likewise for the $\GP$-functions.
\end{remark*}

The proof of \cite[Cor.~7.7]{ChiuMarberg}, with minor changes, would also show that 
our generating function for $\jq_\lambda$ implies the ring property for  
$\ZZ[\beta]\spanning\{\GP_\lambda\}$.
However, this nonconstructive argument is less informative than \cite[Thm.~1.2]{CTY},
which gives an explicit Littlewood--Richardson  rule for products of $\GP$-functions.
It is an open problem to find such a rule for the $\GQ$-functions,
as well as for the $\gp$- and $\gq$-functions,
which span two other subrings of symmetric functions. 

 \subsection{Comparison with unshifted versions}

Our main results are shifted analogues of ``classical'' theorems that we summarize here for comparison.
The (unshifted) \defn{Young diagram} of a partition $\lambda = (\lambda_1 \geq \lambda_2 \geq \dots \geq \lambda_k>0)$
is the set of positions $\D_\lambda = \{ (i,j) \in [k]\times\ZZ : 1 \leq j \leq \lambda_i\}$. 
An (unshifted) \defn{semistandard set-valued tableau} of shape $\lambda$ is defined in the same way as the analogous shifted object,
except that such a tableau is a filling 
of $\D_\lambda$  
by finite nonempty subsets of $\{1<2<3<\dots\}$ rather than $\{1'<1<2'<2<\dots\}$.

\begin{definition}\label{G-def}
Write $\SetTab(\lambda)$ 
for the  set of semistandard set-valued tableaux of shape $\lambda$.
Then the \defn{stable Grothendieck polynomial} of  $\lambda$ is the formal power series
$ G_{\lambda} := \sum_{T \in \SetTab(\lambda)} \beta^{|T|-|\lambda|} {\bf x}^{T}$.
\end{definition}

\begin{definition}\label{g-def}
The \defn{dual stable Grothendieck polynomials} $g_\lambda$ are the unique 
formal power series in $\ZZ[\beta]\llbracket x_1,x_2,\dots\rrbracket  $ indexed by integer partitions $\lambda$ satisfying the Cauchy identity 
\[ \sum_\lambda G_\lambda({\bf x}) g_\lambda({\bf y})
 =
 \sum_\lambda s_\lambda({\bf x}) s_\lambda({\bf y})
  = \prod_{i,j \geq 1} \tfrac{ 1}{1-x_iy_j}.\]
  \end{definition}
  
   \begin{definition}
 The \defn{conjugate dual stable Grothendieck polynomial} of  $\lambda$ is
 $ j_\lambda:= \omega(g_{\lambda^\top})$.
 \end{definition}

All three families $\{ G_\lambda\}$, $\{g_\lambda\}$, and $\{ j_\lambda\}$ are linearly independent symmetric functions
which coincide with the usual Schur functions $\{ s_\lambda\}$ when $\beta=0$ \cite{Buch,LamPyl}. Our definition of $j_\lambda$ involves a transposition of indices
 compared to \cite[\S9.8]{LamPyl}; this convention ensures that $j_\lambda|_{\beta=0} = s_\lambda$.

We define (unshifted) \defn{plane partitions} and \defn{semistandard bar tableaux}
of shape $\lambda$ in the same way as our shifted versions, except the relevant objects are fillings of $\D_\lambda$
by positive integers (excluding primed numbers).
Let $\PPart(\lambda)$ be the set of
 plane partitions of shape $\lambda$. 
 Let $\BT(\lambda)$ be the set of semistandard bar tableaux of shape $\lambda$; these objects are called \defn{valued-set tableaux} in \cite[\S9]{LamPyl}.

\begin{theorem}[{Lam and Pylyavskyy \cite[\S9]{LamPyl}}]
\label{gj-thm}
For all partitions $\lambda$ it holds that
\[
g_{\lambda} = \sum_{T\in \PPart(\lambda)} (-\beta)^{|\lambda| - |\weight_\RPP(T)|} {\bf x}^{\weight_\RPP(T)}
\quand  j_{\lambda} = \sum_{T\in \BT(\lambda)} (-\beta)^{|\lambda|-|T|} {\bf x}^{T} .\]
\end{theorem}

The stable Grothendieck polynomials $G_\lambda$ were introduced in Fomin and Kirillov's paper \cite{FominKirillov}  
as certain limits of Lascoux and Sch\"utzenberger's \defn{Grothendieck polynomials} \cite{LS1982},
which are $K$-theory representatives for Schubert varieties.
Buch \cite[Thm.~3.1]{Buch} derived the set-valued tableaux generating function for $G_\lambda$
given in Definition~\ref{G-def}, and also 
proved that the stable Grothendieck polynomials are a $\ZZ[\beta]$-basis for a ring \cite[Cor.~5.5]{Buch}.
For another proof of this ring property, see \cite{Y2019}.
The structure constants $n_{\lambda\mu}^\nu$ in the expansion $G_\lambda G_\mu = \sum_{\nu} n_{\lambda\mu}^\nu  \beta^{|\nu| - |\lambda| - |\mu|}  G_\nu$
are nonnegative integers with $n_{\lambda\mu}^\nu=0$ whenever $|\nu| < |\lambda| + |\nu|$ or $\ell(\nu) > \ell(\lambda) + \ell(\mu)$
by \cite[Thm.~5.4]{Buch}.

 Lam and Pylyavskyy defined $g_\lambda$ and $j_\lambda$ by the formulas in Theorem~\ref{gj-thm} with $\beta$ set to $-1$.
 They then proved that $\{ g_\lambda\}$ is the basis for the ring of symmetric functions dual to $\{G_\lambda\}$
 via the Hall inner product \cite[Thm.~9.15]{LamPyl}, along with the identity $j_\lambda = \omega(g_{\lambda^\top})$ \cite[Prop.~9.25]{LamPyl}.
For proofs of the Cauchy identity in Definition~\ref{g-def} and various generalizations, see \cite{GunnZinnJustin,LascouxNaruse,Y2019}.

\subsection{Outline}

We conclude this introduction with a brief outline of the rest of this article.
Section~\ref{prelim-sect} contains some miscellaneous results and definitions that are needed in later arguments.
The most technical part of our proofs of the main results 
is showing that the desired generating functions are symmetric.
We derive this explicitly for $\jp_\lambda$ and $\jq_\lambda$
by constructing Bender--Knuth involutions for semistandard shifted bar tableaux in Section~\ref{bk-sect}.
Once symmetry is established, we are able to prove Theorems~\ref{main-thm1} and \ref{main-thm2} by an inductive algebraic argument
in Section~\ref{gen-fun-der-sect}.

\subsection*{Acknowledgments}
The first author was partially supported by Simons Foundation grant 634530.
The second author was partially supported by Hong Kong RGC grant GRF 16306120.

\section{Preliminaries}\label{prelim-sect}

This section has three parts.
Section~\ref{structure-constants-sect} explains the main algebraic consequences of the Cauchy identity \eqref{cauchy-eq}.
Section~\ref{skew-sect} gives the precise definitions of the skew forms of the various symmetric functions in the introduction.
Section~\ref{pieri-sect} derives a Pieri rule for multiplying our dual functions.

\subsection{Structure constants}\label{structure-constants-sect}

Everywhere below, $\lambda$, $\mu$, and $\nu$ are strict partitions.
It follows from \cite[Props.~3.4 and 3.5]{IkedaNaruse} that there are integers $a_{\lambda\mu}^\nu$ and $b_{\lambda\mu}^\nu$ such that 
\be\label{ab-eq}\ba
\GP_\lambda \GP_\mu = \sum_\nu a_{\lambda\mu}^\nu \beta^{|\nu|-|\lambda|-|\mu|} \GP_\nu
\quand
\GQ_\lambda \GQ_\mu = \sum_\nu b_{\lambda\mu}^\nu \beta^{|\nu|-|\lambda|-|\mu|} \GQ_\nu.
\ea
\ee
Let ${\bf y} = (y_1,y_2,\dots)$ be a second set of commuting indeterminates. 
Given a power series $f=f({\bf x})=f(x_1,x_2,\dots) \in \ZZ[\beta]\llbracket x_1,x_2,\dots\rrbracket  $, let  
 $f({\bf y}) :=f(y_1,y_2,\dots) \in \ZZ[\beta]\llbracket y_1,y_2,\dots\rrbracket $
 and define 
\[ f({\bf x}, {\bf y})  := f(x_1,y_1,x_2,y_2,x_3,y_3,\dots) \in \ZZ[\beta]\llbracket x_1,y_1,x_2,y_2,x_3,y_3,\dots\rrbracket  .\]
This operation corresponds to the coproduct on the Hopf algebra of symmetric functions; see \cite[\S2.1]{GrinbergReiner}.
It follows from \cite[Thms.~4.19 and 5.11]{LM2021}
that there are also integers $\widehat a_{\lambda\mu}^\nu$ and $\widehat b_{\lambda\mu}^\nu$ such that 
\be\label{G-coproduct}
\ba 
\GP_\nu({\bf x}, {\bf y}) &= \sum_{\lambda,\mu} \widehat b_{\lambda\mu}^\nu \beta^{|\lambda|+|\mu|-|\nu|} \GP_\lambda({\bf x})\GP_\mu({\bf y}),
\\
\GQ_\nu({\bf x}, {\bf y}) &= \sum_{\lambda,\mu} \widehat a_{\lambda\mu}^\nu \beta^{|\lambda|+|\mu|-|\nu|} \GQ_\lambda({\bf x})\GQ_\mu({\bf y}).
\ea
\ee Observe that the letters $a$ and $b$ here have switched places compared to \eqref{ab-eq}.
We introduce a third set of integer coefficients $\widehat c_{\lambda\mu}^\nu$ such that 
\be\label{G-coproduct2}
\ba 
\GQ_\nu({\bf x}, {\bf y}) &= \sum_{\lambda,\mu} \widehat c_{\lambda\mu}^\nu \beta^{|\lambda|+|\mu|-|\nu|} \GQ_\lambda({\bf x})\GP_\mu({\bf y}).
\ea
\ee

The Cauchy identity \eqref{cauchy-eq} defining $\gp_\lambda$ and $\gq_\lambda$ implies that
\be\label{hat-abc-eq}
\ba 
\gp_\lambda \gp_\mu &= \sum_\nu \widehat a_{\lambda\mu}^\nu \beta^{|\lambda|+|\mu|-|\nu|} \gp_\nu,
&&&
\gp_\nu({\bf x}, {\bf y}) &= \sum_{\lambda,\mu} b_{\lambda\mu}^\nu \beta^{|\nu|-|\lambda|-|\mu|} \gp_\lambda({\bf x})\gp_\mu({\bf y}),
\\
\gq_\lambda \gq_\mu &= \sum_\nu \widehat b_{\lambda\mu}^\nu \beta^{|\lambda|+|\mu|-|\nu|} \gq_\nu,
&&&\gq_\nu({\bf x}, {\bf y}) &= \sum_{\lambda,\mu} a_{\lambda\mu}^\nu \beta^{|\nu|-|\lambda|-|\mu|} \gq_\lambda({\bf x})\gq_\mu({\bf y}).
\\
\gp_\lambda \gq_\mu &= \sum_\nu \widehat c_{\lambda\mu}^\nu \beta^{|\lambda|+|\mu|-|\nu|} \gp_\nu,
\ea\ee
We explain how to derive the bottom left identity; the others are special cases of \cite[Prop.~3.2]{NakagawaNaruse} and follow similarly.
Introducing a third sequence of variables ${\bf z} = (z_1,z_2,\dots)$,
one can write 
\[
\ba\sum_{\nu} \GQ_\nu({\bf x},{\bf y}) \gp_\nu({\bf z}) &= 
 \prod_{i,j \geq 1} \tfrac{ 1 - \overline{x_i} z_j}{1-x_iz_j}\cdot  \prod_{i,j \geq 1} \tfrac{ 1 - \overline{y_i} z_j}{1-y_iz_j}
= \(\sum_{\lambda} \GQ_\lambda({\bf x}) \gp_\lambda({\bf z})\) \( \sum_{\mu} \GP_\mu({\bf y}) \gq_\mu({\bf z})\)
\\
&= \sum_{\lambda,\mu} \GQ_\lambda({\bf x}) \GP_\mu({\bf y}) \gp_\lambda({\bf z})\gq_\mu({\bf z}).
\ea
\]
Then by substituting \eqref{G-coproduct2} into the first expression and equating coefficients of $\GQ_\lambda({\bf x}) \GP_\mu({\bf y})$,
one obtains the identity $\gp_\lambda \gq_\mu = \sum_\nu \widehat c_{\lambda\mu}^\nu \beta^{|\lambda|+|\mu|-|\nu|} \gp_\nu$, as desired.

Since $\omega$ is a bialgebra automorphism, the formulas in \eqref{hat-abc-eq}   still hold
if  we replace every ``$\gp$'' by ``$\jp$''
and every ``$\gq$'' by ``$\jq$''.

The coefficients $\widehat a_{\lambda\mu}^\nu$ and $\widehat b_{\lambda\mu}^\nu$ 
are zero whenever $|\nu| > |\lambda| +|\mu|$ by \cite[Eq.~(5.6)]{ChiuMarberg}; since $\gq_\mu$ is a linear combination
of $\gp_\kappa$'s with $\kappa\subseteq \mu$ by \cite[Cor.~6.2]{ChiuMarberg}, the same must be true of $ \widehat c_{\lambda\mu}^\nu$.
Thus the sums on the left side of \eqref{hat-abc-eq} can be limited to strict partitions with $|\nu| \leq |\lambda| +|\mu|$.
Likewise, the coefficients $ a_{\lambda\mu}^\nu$ and $ b_{\lambda\mu}^\nu$ 
are zero whenever $|\nu| < |\lambda| +|\mu|$ by \cite[Prop.~6.5]{ChiuMarberg},
so the sums on the right side of \eqref{hat-abc-eq} can be limited to strict partitions with $|\lambda| +|\mu| \leq |\nu|$.

\subsection{Skew generalizations}\label{skew-sect}

We write $\mu\subseteq \lambda$ if $\mu$ and $\lambda$ are partitions with $\mu_i \leq \lambda_i$ for all $i$.
If $\mu \subseteq\lambda$ are strict partitions then the \defn{shifted diagram} of $\lambda/\mu$ is the set difference $\SD_{\lambda/\mu} := \SD_\lambda - \SD_\mu$. We define shifted set-valued tableaux, plane partitions, and bar tableaux of skew shape $\lambda/\mu$
in exactly the same way as in the introduction, only now the relevant objects are fillings of $\SD_{\lambda/\mu}$.
The definitions of all related weight statistics like the monomials ${\bf x}^T$ are also unchanged.
Some relevant notation:
\begin{itemize}
\item Let $\ShSetTab_Q(\lambda/\mu)$ be the set of semistandard shifted set-valued tableaux of shape $\lambda/\mu$, and
let $\ShSetTab_P(\lambda/\mu)$ be the subset of such tableaux with no primed numbers in diagonal boxes.

\item Let $\MRPP_Q(\lambda/\mu)$ be the set of 
shifted plane partitions of shape $\lambda/\mu$,
and let
$\MRPP_P(\lambda/\mu)$ 
 be the subset of  such fillings
with no unprimed diagonal entries.

\item Let $\ShBTQ(\lambda/\mu)$ be the set of  semistandard shifted bar tableaux of shape $\lambda/\mu$,
and let $\ShBTP(\lambda/\mu)$
be the subset of such tableaux with no primed diagonal entries.

\end{itemize}
We define  these sets to be empty if $\mu \not\subseteq\lambda$.
These sets will index the terms in generating functions for the skew analogues of  
$\GQ_\lambda$, $\GP_\lambda$, $\gq_\lambda$, $\gp_\lambda$, $\jq_\lambda$, and $\jp_\lambda$,
which are defined as follows:

\begin{definition}
For a strict partition $\lambda$, let 
$\gp_{\lambda/\mu},\gq_{\lambda/\mu} \in  \ZZ[\beta]\llbracket x_1,x_2,\dots\rrbracket  $ be the elements with
\be\label{gp-xy-eq}
\gp_\lambda({\bf x},{\bf y}) = \sum_{\mu }\gp_\mu({\bf x}) \gp_{\lambda/\mu}({\bf y})
\quand
\gq_\lambda({\bf x},{\bf y}) = \sum_{\mu }\gq_\mu({\bf x}) \gq_{\lambda/\mu}({\bf y})
\ee
where the sums are over all strict partitions $\mu$.
\end{definition}

Both $\gp_{\lambda/\mu}$ and $\gq_{\lambda/\mu}$ are symmetric and homogeneous if $\deg \beta = 1$,
but the families of such functions are no longer linearly independent. 
These power series are defined when $\mu\not\subseteq \lambda$,
but one can show that
 $ \gp_{\lambda/\mu}$ and $\gq_{\lambda/\mu}$ are each nonzero
if and only if $\mu \subseteq \lambda$ \cite[Props.~6.5 and 6.7]{ChiuMarberg}.

\begin{definition}
For a strict partition $\lambda$, let 
$\jp_{\lambda/\mu},\jq_{\lambda/\mu} \in  \ZZ[\beta]\llbracket x_1,x_2,\dots\rrbracket  $ be the elements with
\be\label{jp-xy-eq}
\jp_\lambda({\bf x},{\bf y}) = \sum_{\mu }\jp_\mu({\bf x}) \jp_{\lambda/\mu}({\bf y})
\quand
\jq_\lambda({\bf x},{\bf y}) = \sum_{\mu }\jq_\mu({\bf x}) \jq_{\lambda/\mu}({\bf y})
\ee
where the sums are over all strict partitions $\mu$.
\end{definition}

Equivalently, one could define these skew analogues by the following formula:

\begin{proposition}[{\cite[Eq.~(7.4)]{ChiuMarberg}}]
It holds that $\jp_{\lambda/\mu} := \omega(\gp_{\lambda/\mu})$ and $\jq_{\lambda/\mu} := \omega(\gq_{\lambda/\mu})$.
\end{proposition}

The skew versions of $\GP_\lambda$ and $\GQ_\lambda$ that are most relevant to our discussion are not the obvious generating functions
for the sets $\ShSetTab_P(\lambda/\mu)$ and $\ShSetTab_Q(\lambda/\mu)$. Instead,
define a \defn{removable corner box} of  $\SD_\lambda$
to be a position $(i,j)$
such that $\SD_\lambda - \{(i,j)\} = \SD_\mu$ for a strict partition $\mu \neq \lambda$.
Let $\RC(\lambda)$ denote the set of such boxes in $\SD_\lambda$. For strict partitions $\mu \subseteq\lambda$, define 
\[
\ShSetTab_P(\lambda\ss\mu) := \bigsqcup_{\substack{\nu \subseteq\mu \\ \SD_{\mu/\nu}\subseteq \RC(\mu)}} \ShSetTab_P(\lambda/\nu)
\quand
\ShSetTab_Q(\lambda\ss\mu) := \bigsqcup_{\substack{\nu \subseteq\mu \\ \SD_{\mu/\nu}\subseteq \RC(\mu)}} \ShSetTab_Q(\lambda/\nu),
\]
where the unions are over all strict partitions $\nu$ such that $\nu \subseteq \mu$ and $\SD_{\mu/\nu}\subseteq \RC(\mu)$.
For strict partitions $\mu \not\subseteq\lambda$  
set $\ShSetTab_P(\lambda\ss\mu)=\ShSetTab_Q(\lambda\ss\mu):= \varnothing$.

\begin{definition}\label{GJ-def}
 For strict partitions $\mu$ and $\lambda$ define $|\lambda/\mu| := |\lambda|-|\mu|$ and let 
\[ \GP_{\lambda\ss\mu} := \sum_{T \in \ShSetTab_P(\lambda\ss\mu)} \beta^{|\lambda/\mu|-|T|} {\bf x}^{T}
\quand
\GQ_{\lambda\ss\mu} := \sum_{T \in \ShSetTab_Q(\lambda\ss\mu)} \beta^{|\lambda/\mu|-|T|} {\bf x}^{T}.
\]
\end{definition}

Observe that while $\GP_{\lambda\ss \emptyset} =\GP_\lambda$ and $\GQ_{\lambda\ss\emptyset} = \GQ_\lambda$,
the functions $\GP_{\lambda\ss \lambda}$ and $\GQ_{\lambda\ss \lambda}$ are typically not equal to $1$.
Both $\GP_{\lambda\ss\mu} $ and $\GQ_{\lambda\ss\mu} $ are zero if $\mu\not\subseteq\lambda$
since the sets indexing the relevant summations are both empty. These power series are symmetric since we have
\be
\GP_\lambda({\bf x}, {\bf y}) = \sum_\mu \GP_\mu({\bf x})\GP_{\lambda\ss\mu}({\bf y}) 
\quand
\GQ_\lambda({\bf x}, {\bf y}) = \sum_\mu \GQ_\mu({\bf x})\GQ_{\lambda\ss\mu}({\bf y}) 
\ee
and the power series $\GP_\lambda$ and $\GQ_\lambda$ are symmetric.

There are several more general versions of the Cauchy identity \eqref{cauchy-eq} that relate all of these families of skew functions; see \cite[Thm.~6.9 and Cor.~7.8]{ChiuMarberg}.
Two such identities that will be needed later are given as follows.
  Let $\mu$ and $\nu$ be strict partitions. Then by \cite[Cor.~7.8]{ChiuMarberg} we have
       \be\label{cauchy-eq4}
    \sum_\lambda \GP_{\lambda\ss\mu}({\bf x}) \jq_{\lambda/\nu}({\bf y}) = \prod_{i,j \geq 1} \tfrac{1 + x_iy_j}{ 1 + \overline{x_i} y_j}
    \sum_\kappa \GP_{\nu\ss\kappa}({\bf x}) \jq_{\mu/\kappa}({\bf y})
    \ee
    and
           \be\label{cauchy-eq5}
    \sum_\lambda \GQ_{\lambda\ss\mu}({\bf x}) \jp_{\lambda/\nu}({\bf y}) = \prod_{i,j \geq 1} \tfrac{1 + x_iy_j}{ 1 + \overline{x_i} y_j}
    \sum_\kappa \GQ_{\nu\ss\kappa}({\bf x}) \jp_{\mu/\kappa}({\bf y}),
    \ee
    where $\overline{x} := \frac{-x}{1+\beta x}$ as in Section~\ref{cauchy-intro-sect} and the sums are over all strict partitions $\lambda$ and $\kappa$.

\subsection{Pieri rules}\label{pieri-sect}

Buch and Ravikumar derived
Pieri rules in \cite{BuchRavikumar} to compute $a_{\lambda\mu}^\nu$ and $b_{\lambda\mu}^\nu$ when
$\mu = (n)$ is a one-row partition.
Here, we describe analogous formulas for  $\widehat a_{\lambda\mu}^\nu$, $\widehat b_{\lambda\mu}^\nu$, and $\widehat c_{\lambda\mu}^\nu$
when $\mu=(n)$.

We define a \defn{shifted ribbon} to be a shifted skew shape $\SD_{\nu/\lambda}$ 
that does not contain two boxes $(i_1,j_1)$ and $(i_2,j_2)$ with $i_1<i_2$ and $j_1<j_2$.
We do not require this shape to be connected, so 
\[
\SD_{ (8,5,4,1)/(5,4,1)} = 
\ytabsmall{ 
\none & \none & \none & \ \\
\none & \none & \none[\cdot] & \  & \ & \ \\
\none & \none[\cdot] & \none[\cdot] & \none[\cdot] & \none[\cdot] & \ \\
\none[\cdot] & \none[\cdot] & \none[\cdot] & \none[\cdot] & \none[\cdot] & \ & \ & \
}
\quand
\SD_{ (8,4,3,1)/(5,4,1)} = 
\ytabsmall{ 
\none & \none & \none & \ \\
\none & \none & \none[\cdot] & \  & \  \\
\none & \none[\cdot] & \none[\cdot] & \none[\cdot] & \none[\cdot]  \\
\none[\cdot] & \none[\cdot] & \none[\cdot] & \none[\cdot] & \none[\cdot] & \ & \ & \
}
\]
are both shifted ribbons.   In general, if $\nu = (\nu_1 > \nu_2 > \nu_3 > \dots > \nu_k > 0)$,
 then 
$\SD_{\nu/\lambda}$ is a shifted ribbon 
if and only if the partition $ (\nu_2, \nu_3,\dots, \nu_k) $ is contained in $ \lambda$.

The \defn{row reading word order} of  $\ZZ \times \ZZ$
  has $(i_1,j_1) < (i_2,j_2)$ if $i_1 > i_2$ or if $i_1=i_2$ and $j_1<j_2$.
Suppose $T$ is a semistandard set-valued shifted tableau containing only $1'$ and $1$.
Then the shape of $T$ must be a shifted ribbon, 
and any entry which is not the first in its  connected component (in row reading order)
is uniquely determined and given by the sets $\{1'\}$ or $\{1\}$.
However, the first entries in  each connected component of $T$
may be any of $\{1'\}$, $\{1\}$, or $\{1',1\}$.

\begin{example} 
Two such tableaux of shape $(8,5,3,1)/(5,4,1)$ are given by
\[
\ytab{ 
\none & \none & \none & 1 \\
\none & \none & \none[\cdot] & 1'  & 1  \\
\none & \none[\cdot] & \none[\cdot] & \none[\cdot] & \none[\cdot] & 1'1  \\
\none[\cdot] & \none[\cdot] & \none[\cdot] & \none[\cdot] & \none[\cdot] & 1' & 1 & 1
}
\qquand
\ytab{ 
\none & \none & \none & 1'1 \\
\none & \none & \none[\cdot] & 1'  & 1  \\
\none & \none[\cdot] & \none[\cdot] & \none[\cdot] & \none[\cdot] & 1'  \\
\none[\cdot] & \none[\cdot] & \none[\cdot] & \none[\cdot] & \none[\cdot] & 1' & 1 & 1
}.
\]
There are seven other such tableaux of this shape.
\end{example}

For strict partitions $\lambda \subseteq \nu$, define 
$\cribbonsym(\nu/\lambda)$ to be 
the set of semistandard shifted set-valued tableaux of shape $\SD_{\nu/\lambda}$ containing only   $1'$ and $1$ as entries. Then let 
\[\cribbons{\nu}{\lambda} := \bigsqcup_{\substack{\mu \subseteq \lambda \\ \SD_{\lambda/\mu} \subseteq \RC(\lambda)}} \cribbonsym(\nu/\mu)\]
where the union is over strict partitions $\mu$ whose shifted diagrams are formed by deleting a possibly empty set of removable corner boxes from $\SD_\lambda$.
 The set $\cribbons{\nu}{\lambda} $ is nonempty if and only if $\SD_{\nu/\lambda}$ is a shifted ribbon,
and we have  $|\cribbonsym(\nu/\nu)|=1\leq |\cribbons{\nu}{\nu}|=4^{|\RC(\nu)|}$.

 \begin{example}
If $\lambda=(5,4,1)$ and  $ \nu =(8,5,3,1)$ then $\cribbons{\nu}{\lambda}$ consists of  
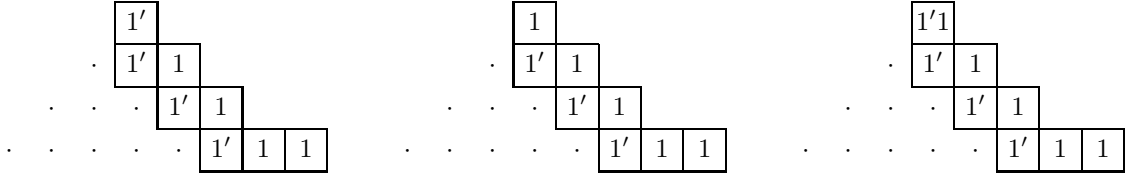
\[
\ytab{ 
\none & \none & \none & 1' \\
\none & \none & \none[\cdot] & 1'  & 1  \\
\none & \none[\cdot] & \none[\cdot] & \none[\cdot] & 1' & 1  \\
\none[\cdot] & \none[\cdot] & \none[\cdot] & \none[\cdot] & \none[\cdot] & 1' & 1 & 1
}
\qquad
\ytab{ 
\none & \none & \none & 1 \\
\none & \none & \none[\cdot] & 1'  & 1  \\
\none & \none[\cdot] & \none[\cdot] & \none[\cdot] & 1' & 1  \\
\none[\cdot] & \none[\cdot] & \none[\cdot] & \none[\cdot] & \none[\cdot] & 1' & 1 & 1
}
\qquad
\ytab{ 
\none & \none & \none & 1'1 \\
\none & \none & \none[\cdot] & 1'  & 1  \\
\none & \none[\cdot] & \none[\cdot] & \none[\cdot] & 1' & 1  \\
\none[\cdot] & \none[\cdot] & \none[\cdot] & \none[\cdot] & \none[\cdot] & 1' & 1 & 1
}
\]
in addition to the nine elements of  $\cribbonsym(\nu/\lambda)$.
\end{example}

Let $\bribbons{\nu}{\lambda}$ be the set of elements in  
$\cribbons{\nu}{\lambda}$ 
with no primed numbers appearing in any diagonal boxes.
Set $\bribbons{\nu}{\lambda} =\cribbons{\nu}{\lambda} := \varnothing$
for strict partitions $\lambda \not\subseteq\nu$. Finally, for integers $n \geq 0$
define $\bribbons{\nu}{\lambda;n}\subseteq \bribbons{\nu}{\lambda}$ and $ 
\cribbons{\nu}{\lambda;n}\subseteq \cribbons{\nu}{\lambda}$
to be the subsets of tableaux $T$ with ${\bf x}^T = x_1^n$.

\begin{proposition}\label{hat-b-prop}
Suppose $\lambda$ and $\nu$ are strict partitions and $n$ is a nonnegative integer. Then
\[ \widehat b_{\lambda,(n)}^\nu=|\bribbons{\nu}{\lambda;n}|
\quand \widehat c_{\lambda,(n)}^\nu=|\cribbons{\nu}{\lambda; n}|.\]
\end{proposition}

\begin{proof}
For simplicity assume $\beta=1$.
Then $ \widehat c_{\lambda,(n)}^\nu$ is the coefficient of 
$\GQ_{\lambda}({\bf x})\GP_{(n)}({\bf y})$ in $\GQ_\nu({\bf x}, {\bf y})$.
Since 
$\GP_\mu(y_1,0,0,\dots)$ is zero whenever $\ell(\mu) \geq 2$ and equal to $y_1^{\mu_1}$ if $\ell(\mu) \leq 1$,
the number $ \widehat c_{\lambda,(n)}^\nu$ is also the coefficient of
$\GQ_{\lambda}({\bf x}) t^n$ in the power series $\GQ_\nu({\bf x};t) $
formed by setting $y_1= t$ and $y_2=y_3=\dots=0$ in $\GQ_\nu({\bf x}, {\bf y})$.
Since $\GQ_\nu= \sum_{T \in \ShSetTab_Q(\nu)}   {\bf x}^{T}$ is symmetric,
we have
\[
 \GQ_\nu({\bf x};t)
  =
   \sum_{\lambda \subseteq \nu} \sum_{T \in \bribbons{\nu}{\lambda}} 
   \GQ_\lambda({\bf x})   t^{|T|}
   =
    \sum_{n\geq 0}  \sum_{\lambda \subseteq \nu} | \cribbons{\nu}{\lambda;n}|  
   \GQ_\lambda({\bf x})  t^{n}
   \]
so $ \widehat c_{\lambda,(n)}^\nu=|\cribbons{\nu}{\lambda; n}|$.
The argument to show that  $\widehat b_{\lambda,(n)}^\nu=|\bribbons{\nu}{\lambda; n}|$ is similar,
since $\widehat b_{\lambda,(n)}^\nu$ is equal to the coefficient of $\GP_{\lambda}({\bf x}) t^n$ in $\GP_\nu({\bf x};t) $.
\end{proof}

It is slightly less straightforward to  interpret the numbers  $\widehat a_{\lambda,(n)}^\nu$
along these lines.
However, if $n>0$ then it is easy to check that $\gq_n = 2\gp_n + [n>1] \beta \gp_{n-1}$
and so \eqref{hat-abc-eq} implies that 
\be
\widehat c_{\lambda,(n)}^\nu = 2 \widehat a_{\lambda,(n)}^\nu + [n > 1]  \widehat a_{\lambda,(n-1)}^\nu.\ee

\section{Bender--Knuth involutions for shifted bar tableaux}\label{bk-sect}

In this section, we consider several families of shifted bar tableaux.  Recall that these objects are
shifted tableaux whose boxes are divided into bars, each consisting  of a contiguous sequence of equal primed entries 
in one column or equal unprimed entries in one row. In such tableaux, no primed entry can be repeated in a row and no unprimed entry can be repeated in a column.   

For this section only, \textbf{we assume the entries of all shifted bar tableaux are limited to the set $\{1', 1, 2', 2\}$}.
In this context, the \defn{weight} of a shifted bar tableau is the pair $(a_1,a_2)$ where $a_i$ is the number of bars containing $i$ or $i'$.
Our goal is to construct a shape-preserving and weight-reversing involution of 
the set of such tableaux that are semistandard, analogous to the Bender--Knuth involutions on usual semistandard Young tableaux.

\subsection{Sorted bar tableaux}

For convenience, we will often abbreviate ``shifted bar tableau'' as ``ShBT.''
Recall that a ShBT is \defn{semistandard}
if
the entries in the tableau (ignoring bars) are weakly increasing in the order $1'<1<2'<2$ along rows and columns.  
We will need a variant of this property:

\begin{definition}
\label{def:sorted}

Let $T$ be a shifted bar tableau with all entries in $\{1', 1, 2', 2\}$. We say that $T$ is \defn{sorted}
if (in addition to our blanket assumption that no primed entry is repeated in any row and no unprimed entry is repeated in any column) any of the following holds:
\begin{itemize}
\item[(a)] 
The entries of $T$
are weakly increasing along rows and columns in the order $1' \prec 2' \prec 1\prec 2$.

\item[(b)] $T$ has exactly two diagonal boxes, $(i,i)$ and $(i+1,i+1)$,
such that   $T_{i,i+1} = 1$, $T_{i+1,i+1}=2'$,
and either 
  $T_{i,i} = 2'$ or  boxes $(i,i)$ and $(i,i+1)$ are part of the same bar with entry $1$,
  and
changing the value in box $(i+1,i+1)$ to $2$ results in a shifted bar tableau satisfying (a).
  
 For example, all of the following are sorted:
\be\label{inter-ex1}
\begin{young}[15.5pt][c] 
 , &  ![pink!99]2' \\
 2' & ]=![cyan!99] 1 & ![cyan!99]   
 \end{young},
 \qquad
 \begin{young}[15.5pt][c] 
 , &  ![pink!99]2' \\
 ]=![cyan!99] 1 & ![cyan!99]  & =]![cyan!99]   
 \end{young},
 \qquad
\begin{young}[15.5pt][c] 
 , &  ![pink!99]2' \\
 2' & ![cyan!99]1 \\
 ![lightgray!99] 1' \ynobottom & ![yellow!99] 2' \ynobottom \\
 ![lightgray!99]  \ynotop & ![yellow!99] \ynotop
  \end{young},
 \qquand
 \begin{young}[15.5pt][c] 
 , &  ![pink!99]2' \\
 ]=![cyan!99]1 & =]![cyan!99] \\
 ]=]![lightgray!99] 1' \ynobottom & ]=]![yellow!99] 2' \ynobottom \\
 ]=]![lightgray!99]  \ynotop & ]=]![yellow!99] \ynotop
  \end{young} .
 \ee
  
\item[(c)] $T$ has exactly two diagonal boxes $(i,i)$ and $(i+1,i+1)$
such that   $T_{i,i} = 1$, $T_{i,i+1}=2'$,
and either 
  $T_{i+1,i+1} = 1$ or  boxes $(i+1,i+1)$ and $(i,i+1)$ are part of the same bar with entry $2'$,
  and
changing the value in box $(i,i)$ to $1'$ results in a shifted bar tableau satisfying (a).

For example, all of the following are sorted:
\be\label{inter-ex2}
 \begin{young}[15.5pt][c] 
  , &  1 \\
 ![cyan!99]1 & ![pink!99]2' \ynobottom \\
 , & ![pink!99] \ynotop  
 \end{young},
\qquad
 \begin{young}[15.5pt][c] 
  , &  ![pink!99]2'  \ynobottom \\
 ![cyan!99]1 & ![pink!99] \ynotop \ynobottom \\
 , & ![pink!99] \ynotop  
 \end{young},
 \qquad
 \begin{young}[15.5pt][c] 
   , &  1  &  ]=![yellow!99] 2&  =]![yellow!99]  \\
 ![cyan!99]1 & ]=]![pink!99]2'  &  ]=![lightgray!99]1 &   =]![lightgray!99] 
  \end{young},
\qquand
 \begin{young}[15.5pt][c] 
   , &  ]=]![pink!99]2' \ynobottom &  ]=![yellow!99] 2&  =]![yellow!99]  \\
 ![cyan!99]1 & ]=]![pink!99] \ynotop &  ]=![lightgray!99]1 &   =]![lightgray!99] 
  \end{young}.
 \ee
 \end{itemize}
 \end{definition}

%
%

\subsection{Ascending swaps}

\def\swap{\mathsf{swap}_{1\leftrightarrow 2'}}
\def\unswap{\mathsf{unswap}_{1\leftrightarrow 2'}}
\def\align{\mathsf{align}}
\def\wtrev{\mathsf{reverse\_weight}}

 The \defn{first} (respectively, \defn{last}) box in a bar within a ShBT
is the box $(i,j)$ with $i$ and $j$ both minimal (respectively, maximal). 
Everywhere in this subsection, $T$ denotes an arbitrary ShBT with all entries in $\{1' , 2' , 1 , 2\}$, which is not necessarily semistandard or sorted.

\begin{definition}\label{def:swap}
Suppose $\cH$ and $\cV$ are bars in $T$ with unique entries $1$ and $2'$, respectively.  
Below, we define two new bars $\tilde \cH$ and $\tilde \cV$ whose union  $\tilde \cH\sqcup \tilde \cV$ occupies the same boxes as $\cH\sqcup \cV$.
We then set
$ \swap(T,\cV, \cH) := (\tilde T, \tilde \cV)$
where $\tilde T$ is formed from $T$ by replacing $\cH$   with $\tilde \cH$ and   $\cV$ with $\tilde \cV$.
The rules defining $\tilde \cH$ and $\tilde \cV$ are as follows:
\begin{itemize}

\item[(a)] Suppose the first boxes of $\cH$ and $\cV$ are $(i,j)$ and $(i+1,j)$, respectively,
and that if $(i,j-1) \in T$ then $T_{i,j-1}\neq 2'$. This means boxes $(i,j-1)$, $(i,j)$, $(i+1,j)$ in $T$ cannot have the form
\be\label{exclude1}
 \begin{young}[15.5pt][c] 
  , &  ![pink!99]2' \\
2' & ![cyan!99]1  
\end{young}
\qquord
 \begin{young}[15.5pt][c] 
  , &  ![pink!99]2' \\
]=![cyan!99]1 & =]![cyan!99]  
\end{young},
\ee
since in first case $T_{i,j-1} = 2'$ and in the second case the first boxes of $\cH$ and $\cV$ are not in the same column.
If $\cH$ has more than one box, then we form $\tilde \cH$ 
from $\cH$  by removing its first box and we form $\tilde \cV$ by adding this box to $\cV$.
If $\cH$ has only one box, then we form $\tilde \cV$ by moving $\cV$ down one row
and we form $\tilde \cH$ by moving $\cH$ to occupy the last box of $\cV$.
In pictures, we have
\[
 \begin{young}[15.5pt][c]   
   ![pink!99] 2' \ynobottom \\
    ![pink!99] \ynotop \\ 
    ]=![cyan!99] 1  & ![cyan!99] & ![cyan!99]  \end{young} \mapsto 
     \begin{young}[15.5pt][c]  
        ![pink!99]2' \ynobottom\\ 
        ![pink!99]\ynobottom \ynotop \\ 
        ![pink!99] \ynotop & ]=![cyan!99]1 & =]![cyan!99]\end{young}
\qquord
 \begin{young}[15.5pt][c]    
  ![pink!99] 2' \ynobottom \\ 
  ![pink!99] \ynotop \\ 
  ![cyan!99] 1    \end{young} \mapsto  
  \begin{young}[15.5pt][c]     
  ![cyan!99]1 \\ 
  ![pink!99]2'  \ynobottom\\ 
  ![pink!99] \ynotop \end{young}
\]
where the  blue boxes are $\cH$ and $\tilde \cH$
and the  red boxes are $\cV$ and $\tilde \cV$.

\item[(b)] Suppose the last boxes of $\cH$ and $\cV$ are $(i,j)$ and $(i,j+1)$, respectively,
and that if $(i + 1 ,j+1) \in T$ then $T_{i+1,j+1}\neq 1$. This mean boxes  $(i,j)$, $(i,j+1)$, $(i+1,j+1)$ in $T$ cannot have the form
\be\label{exclude2}
  \begin{young}[15.5pt][c] 
, &  1 \\
![cyan!99]1 & ![pink!99]2'  \end{young}
\qquord
  \begin{young}[15.5pt][c] 
 , &  ![pink!99]2' \ynobottom \\
![cyan!99]1 & ![pink!99] \ynotop  \end{young}.
\ee
If $\cV$ has more than one box, then we form $\tilde \cV$ 
from $\cV$  by removing its last box and we form $\tilde \cH$ by adding this box to $\cH$.
If $\cV$ has only one box, then we form $\tilde \cH$ by moving $\cH$ right one column
and we form $\tilde \cV$ by moving $\cV$ to occupy the first box of $\cH$.
In pictures, we have
\[
  \begin{young}[15.5pt][c] 
]=![cyan!99] 1  & =]![cyan!99] & ![pink!99] 2' \ynobottom\\ 
, & , & ![pink!99] \ynotop \ynobottom \\ 
, & , & ![pink!99] \ynotop \end{young} \mapsto  
  \begin{young}[15.5pt][c]
]=![cyan!99] 1 & ![cyan!99] & =]![cyan!99]  \\ 
, & , & ![pink!99]2' \ynobottom \\ 
, & , & ![pink!99] \ynotop\end{young}
\qquord
  \begin{young}[15.5pt][c] 
]=![cyan!99] 1  & =]![cyan!99] & ![pink!99] 2'  \end{young}
 \mapsto  
  \begin{young}[15.5pt][c] ![pink!99] 2' & ]=![cyan!99] 1& =]![cyan!99] \end{young}.
\]

\item[(c)] In all other cases 
we set $\tilde \cH := \cH$ and $\tilde \cV :=\cV$.
\end{itemize}
\end{definition}

Above we defined how to ``swap'' one $2'$-bar in $T$ with a $1$-bar.
We will extend this to a procedure successively swapping a given $2'$-bar with all $1$-bars in a certain order. Then we will further extend that operation to one successively swapping all $2'$-bars in $T$ with all $1$-bars. This will give an operation transforming $T$ to another ShBT in which the relative order of   $1'$ and $2$ is reversed.

\begin{definition}\label{swap-def}
Suppose $\cH_1,\cH_2,\dots,\cH_k$ are the distinct bars of $T$ containing $1$ ordered from right to left,  so that if $i < j$ then $\cH_i$ occurs in larger-indexed columns than $\cH_j$. 
For any single bar $ \cV$ of $T$ with unique entry $2'$, inductively define 
\[(\tilde T_0, \tilde \cV_0) := (T,\cV)
\quand  (\tilde T_{i},\tilde\cV_{i}) := 
\swap(\tilde T_{i-1},\tilde\cV_{i-1}, \cH_i)\text{ for $i=1,2,\dots,k$,}\]
and then set $\swap(T,\cV)  := \tilde T_k$.
Next, suppose $\cV_1,\cV_2,\dots,\cV_l$ are the distinct bars of $T$ containing $2'$ ordered from bottom to top,  so that if $i < j$ then $\cV_i$ occurs in smaller-indexed rows than $\cV_j$.
Finally, define
$ \swap(T) := \swap(\cdots \swap(\swap(T,\cV_1),\cV_2) \dots, \cV_l).$
\end{definition}

\begin{example}
In the image below, the bars $\cH_1$, $\cH_2$, $\cH_3$, $\cH_4$, and $\cH_5$ are respectively gray, green, yellow, blue, and white; the bars $\cV_1$, $\cV_2$, and $\cV_3$ are respectively orange, pink, and red.  We illustrate each individual operation of the form $\swap( -, \cV, \cH)$ that results in a new tableau (i.e., that does not fall in case (c) of Definition~\ref{def:swap}):
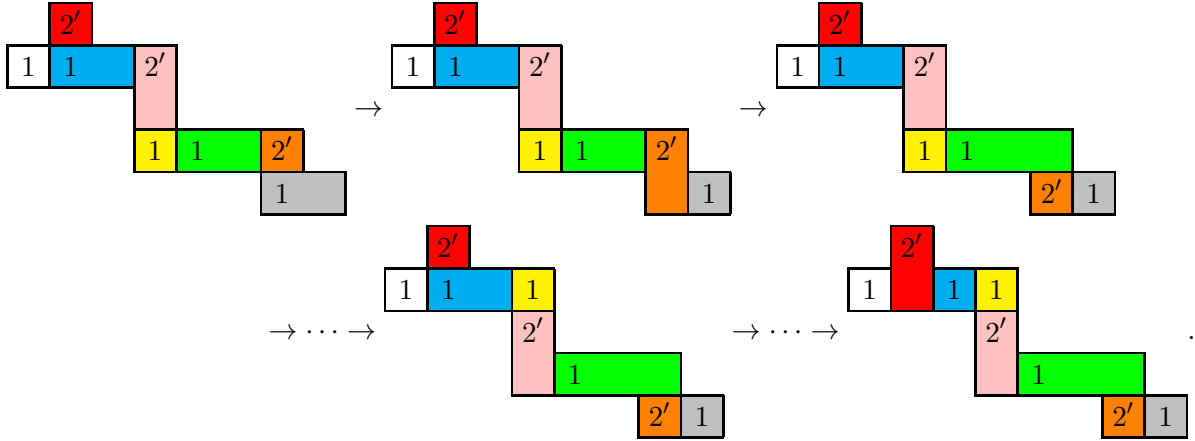
\begin{multline*}
\text{$\begin{young}[15.5pt][c] , & ![red!99] 2' \\ 1 & ]=![cyan!99] 1 & =]![cyan!99]  & ![pink!99] 2' \ynobottom \\ , & , & , & ![pink!99]  \ynotop \\ , & , & , & ![yellow!99] 1 & ]=![green!99] 1 & =]![green!99]  & ![orange!99] 2' \\ , & , & , & , & , & , & ]=![lightgray!99] 1 & =]![lightgray!99] \end{young}$}
\to
\text{$\begin{young}[15.5pt][c] , & ![red!99] 2' \\ 1 & ]=![cyan!99] 1 & =]![cyan!99]  & ![pink!99] 2' \ynobottom \\ , & , & , & ![pink!99] \ynotop \\ , & , & , & ![yellow!99] 1 & ]=![green!99] 1 & =]![green!99]  & ![orange!99] 2' \ynobottom \\ , & , & , & , & , & , & ![orange!99] \ynotop & ]=]![lightgray!99] 1\end{young}$}
\to
\text{$\begin{young}[15.5pt][c] , & ![red!99] 2' \\ 1 & ]=![cyan!99] 1 & =]![cyan!99]  & ![pink!99] 2' \ynobottom \\ , & , & , & ![pink!99] \ynotop \\ , & , & , & ![yellow!99] 1 & ]=![green!99] 1 & ==![green!99]  & =]![green!99]  \\ , & , & , & , & , & , & ![orange!99] 2' & ]=]![lightgray!99] 1\end{young}$} \\
\to \cdots \to 
\text{$\begin{young}[15.5pt][c] , & ![red!99] 2' \\ 1 & ]=![cyan!99] 1 & =]![cyan!99]  & ![yellow!99] 1 \\ , & , & , & ![pink!99] 2' \ynobottom \\ , & , & , & ![pink!99] \ynotop & ]=![green!99] 1 & ==![green!99]  & =]![green!99]  \\ , & , & , & , & , & , & ![orange!99] 2' & ]=]![lightgray!99] 1\end{young}$}
\to \cdots \to 
\text{$\begin{young}[15.5pt][c] , & ![red!99] 2' \ynobottom \\ 1 & ![red!99] \ynotop & ![cyan!99] 1 & ![yellow!99] 1 \\ , & , & , & ![pink!99] 2' \ynobottom \\ , & , & , & ![pink!99] \ynotop & ]=![green!99] 1 & ==![green!99]  & =]![green!99]  \\ , & , & , & , & , & , & ![orange!99] 2' & ]=]![lightgray!99] 1\end{young}$}.
\end{multline*}
\end{example}

\begin{proposition}
\label{prop:semi to inter}
If $T$ is a semistandard shifted bar tableau then $\swap(T)$ is a sorted shifted bar tableau of the same shape and weight.
\end{proposition}

\begin{proof}
Let $T$ be a semistandard ShBT. Then $\swap(T)$ has the same shape and weight as $T$ by construction.
Since the operation $\swap$ affects only the boxes in $T$ filled with $1$ or $2'$, the result is valid for all semistandard $T$ filled with $\{1', 1, 2', 2\}$ if and only if it is valid for those semistandard $T$ filled with only $\{1, 2'\}$.  So without loss of generality assume that the only entries of $T$ are $1$ and $2'$.  Such a tableau $T$ contains no $2 \times 2$ square of entries, as there is no way to complete the diagram $\ytab{? & 2' \\ 1 & }$ to a semistandard ShBT.  In particular, the only way that $T$ can contain one box strictly northeast of another is if it contains two consecutive boxes on the diagonal:
\[
\ytab{ \none & 2' \\ 1 & 1} \qquord \ytab{ \none & 2' \\ 1 & 2'} \qquad \text{(ignoring the division into bars).}
\]
Thus, the partial order $\preceq$ on the bars in $T$ where one bar is smaller than (``before'') another if the first lies entirely (weakly) southeast of the latter is either a total order (if $T$ does not contain two diagonal boxes) or has a unique incomparable pair that come after all other bars of $T$.  This order is compatible with the orders $\cH_1, \ldots, \cH_k$ and $\cV_1, \ldots, \cV_l$ in Definition~\ref{swap-def}.  Moreover, this order extends to every $\{1, 2'\}$-filled ShBT (semistandard, sorted, or otherwise) with the same shape as $T$, and to every tableau appearing as an intermediate stage in the computation of $\swap(T)$.  

Suppose the bars of $1$s in $T$ are $\cH_1, \ldots, \cH_k$ and the bars of $2'$s in $T$ are $\cV_1, \ldots, \cV_l$, as in Definition~\ref{swap-def}.  
Applying any number of $\swap$ operations to $T$ results in a tableau with $k$ bars of $1$s and $l$ bars of $2'$s.  For $i = 1, \ldots, l$, let 
$
T^{(i)} := \swap( \cdots \swap(T, \cV_1) \cdots, \cV_i),
$
and let the bars of $1$s and $2'$s in $T^{(i)}$ be respectively $\cH^{(i)}_1, \ldots,  \cH^{(i)}_k$ and $\cV^{(i)}_1, \ldots,  \cV^{(i)}_l$.

Given a tableau, we will say that a pair of bars $A$, $B$ is \defn{bad} if $A$ is filled with $1$, $B$ is filled with $2'$, and $A$ contains a box either directly west or directly south of a box in $B$ as in these pictures:
\[
\ytab{1 &  \cdots & 2'} \qquord \ytab{2' \\  \vdots \\ 1}.
\]
Each move in (a) and (b) of Definition~\ref{def:swap} replaces an adjacent pair of bad bars with a non-bad pair.  We claim that in the tableau $T^{(i)}$, there are no bad pairs involving any of the bars $\cV_1^{(i)}, \ldots, \cV_i^{(i)}$, except possibly in the case that $T$ has two diagonal boxes and $i = l$.  

For $i = 0$, this is vacuously true; we proceed by induction.    By hypothesis, there are no bad pairs involving $\cV_1^{(i - 1)}, \ldots, \cV_i^{(i - 1)}$ in $T^{(i)}$.  Suppose that $\cV_i = \cV_i^{(i - 1)}$ lies in column $c$, and choose $j$ such that $\cH_1^{(i - 1)}, \ldots, \cH_j^{(i - 1)}$ lie in columns whose indices are greater than or equal to $c$ (i.e., they lie weakly southeast of $\cV_i$ in $T^{(i - 1)}$), while $\cH_{j + 1} = \cH_{j + 1}^{(i - 1)}$, \ldots, $\cH_{k} = \cH_{k}^{(i - 1)}$ lie in columns whose indices are strictly smaller than $c$.  In this setup, it is possible that a single bar of $1$s lies in columns with indices both smaller than and greater than or equal to $c$, but only if $i = l$ and $\cV_i$ consists of a single diagonal box, as in 
$ \begin{young}[15.5pt][c] 
  , &  ![pink!99]2' \\
]=![cyan!99]1 & ==![cyan!99] & =]![cyan!99]  
\end{young}$.

The first $j - 1$ swaps in the computation of $\swap(T^{(i - 1)}, \cV_i)$ do nothing.  The $j$th swap either does nothing (if $\cH_j^{(i - 1)}$ does not contain any boxes in column $c$) or performs a swap of the type described in Definition~\ref{def:swap}(a).  Following this swap, the resulting two bars $\tilde{\cV}_i, \tilde{\cH}_j$ no longer form a bad pair.  Moreover, although the bar $\tilde{\cV}_i$ may extend one row below $\cV_i$, this cannot introduce any new bad pairs involving the bar $\tilde{\cV}_i$ unless there is a $1$ to the left of the first box in $\cH_j^{(i - 1)}$, which only happens if $i = l$ and $\cV_i$ consists of a single diagonal box, as in these examples:
\[
 \begin{young}[15.5pt][c] 
  , &  ![pink!99]2' \\
]=]![lightgray!99]1 & ]=![cyan!99]  1& ==![cyan!99] & =]![cyan!99]  
\end{young}
\to
 \begin{young}[15.5pt][c] 
  , &  ![pink!99]2'  \ynobottom \\
]=]![lightgray!99]1 & ]=]![pink!99]  \ynotop & ]=![cyan!99]  1 & =]![cyan!99]  
\end{young}
\quord
 \begin{young}[15.5pt][c] 
  , &  ![pink!99]2' \\
![lightgray!99]1 & ![cyan!99]  1
\end{young}
\to
 \begin{young}[15.5pt][c] 
  , &  ![cyan!99]1   \\
![lightgray!99]1 & ![pink!99] 2'   
\end{young}.
\]

If, at this stage, there are any further bad pairs involving $\tilde{\cV}_i$, they must come from $1$s in boxes directly west of a box in $\tilde{\cV}_i$.  This can happen if $i = l$ and $\tilde{\cV}_i$ contains a diagonal box $\begin{young} [15.5pt][c] 
  , &  ![pink!99]2' \ynobottom  \\
![lightgray!99]1 & ![pink!99]  \ynotop   
\end{young}$,
but otherwise the bar $\cH_{k + 1}$ of $1$s in the bad pair must be directly west of and adjacent to the last box in $\tilde{\cV}_i$.  If $\tilde{\cV}_i$ has more than one box, then the immediate next swap (of the type in Definition~\ref{def:swap}(b)) removes the box from $\tilde{\cV}_i$ that is involved in the bad pair, and the remaining swaps (with $\cH_{j + 2}$, \ldots) do nothing, leaving a tableau in which the new bar $\cV^{(i)}_i$ is not involved in any bad pairs.  If $\tilde{\cV}_i$ only has a single box, then possibly several moves of type Definition~\ref{def:swap}(b) are required; after these moves, the box occupied by $\cV^{(i)}_i$ was previously occupied by a box labeled $1$, so there cannot be any other $1$s in the same column, and so also in this case $\cV^{(i)}_i$ is not involved in any bad pairs.  

Finally, we observe that none of these moves changes any bar of $1$s that is east of $\cV^{(i - 1)}_{i - 1}$ (because they are all strictly southeast of $\cV_i$) and none of these moves causes any $1$s to appear in a lower row than they previously appeared, so no bad pairs involving $\cV^{(i - 1)}_1 = \cV^{(i)}_1$, \ldots, $\cV^{(i - 1)}_{i - 1} = \cV^{(i)}_{i - 1}$ are created.  This completes the proof of the claim.

It follows from the preceding argument that either $\swap(T)$ has no bad pairs at all or its shape has two diagonal boxes and there is a bad pair involving the bar of $2'$s in the upper diagonal box.  By comparing with the definition of sorted tableaux (see Definition~\ref{def:sorted}), we find that this is precisely the desired conclusion.
%
%
\end{proof}

\subsection{Descending swaps}

We now define an inverse to the $\swap$ operation.
In this subsection, $\tilde T$ continues to denote an arbitrary ShBT with all entries in $\{1' , 2' , 1 , 2\}$, which is not necessarily semistandard or sorted.

\begin{definition}
\label{def:unswap}
Suppose $\tilde\cH$ and $\tilde\cV$ are bars in $\tilde T$ with unique entries $1$ and $2'$, respectively.
Below, we define two new bars $  \cH$ and $  \cV$ whose union  $  \cH\sqcup   \cV$ occupies the same boxes as $\tilde\cH\sqcup \tilde\cV$.
We then set
$ \unswap(\tilde T,\tilde\cV, \tilde\cH) := (  T,   \cV)$
where $  T$ is formed from $\tilde T$ by replacing $\tilde \cH$ with $  \cH$ and $\tilde\cV$ with $  \cV$.
The rules defining $  \cH$ and $  \cV$ are as follows:
\begin{itemize}

\item[(a)] Suppose the first boxes of $\tilde\cV$ and $\tilde\cH$ are $(i,j)$ and $(i,j+1)$, respectively.
 If $\tilde\cV$ has more than one box, then we form $ \cV$ from $\tilde\cV$ by removing its first box and we form $ \cH$ by adding this box to $\tilde\cH$.
If $\tilde\cV$ has only one box, then we form $ \cH$ by moving $\tilde\cH$  to the left one column
and we form $ \cV$ by moving $\tilde\cV$ to occupy the last box of $\tilde\cH$.
In pictures, we have 
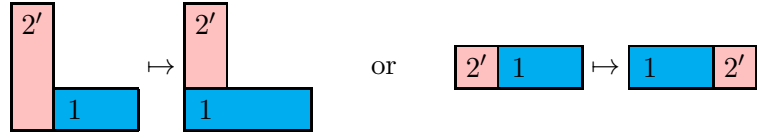
\[
     \begin{young}[15.5pt][c]  
        ![pink!99]2' \ynobottom\\ 
        ![pink!99]\ynobottom \ynotop \\ 
        ![pink!99] \ynotop & ]=![cyan!99]1 & =]![cyan!99]\end{young}
        \mapsto
 \begin{young}[15.5pt][c]   
   ![pink!99] 2' \ynobottom \\
    ![pink!99] \ynotop \\ 
    ]=![cyan!99] 1  & ![cyan!99] & ![cyan!99]  \end{young} 
\qquord
  \begin{young}[15.5pt][c] ![pink!99] 2' & ]=![cyan!99] 1& =]![cyan!99] \end{young}
  \mapsto
    \begin{young}[15.5pt][c]  ]=![cyan!99] 1  & =]![cyan!99] & ![pink!99] 2'  \end{young}
    \]
where the blue boxes are $\tilde \cH$ and $\cH$ while the red boxes are $\tilde \cV$ and $\cV$.

\item[(b)] Suppose the last boxes of $\tilde\cV$ and $\tilde\cH$ are $(i,j)$ and $(i+1,j)$, respectively.
 If $\tilde\cH$ has more than one box, then we form $ \cH$ from $\tilde\cH$ by removing its last box and we form $ \cV$ by adding this box to $\tilde\cV$.
If $\tilde\cH$ has only one box, then we form $ \cV$ by moving $\tilde\cV$ up one row
and we form $ \cH$ by moving $\tilde\cH$ to occupy the first box of $\tilde\cV$.
In pictures, we have 
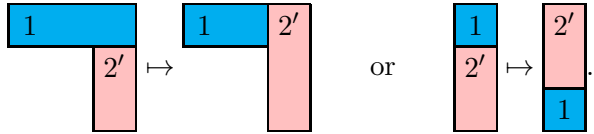
\[
  \begin{young}[15.5pt][c]
]=![cyan!99] 1 & ![cyan!99] & =]![cyan!99]  \\ 
, & , & ![pink!99]2' \ynobottom \\ 
, & , & ![pink!99] \ynotop
\end{young}
\mapsto
  \begin{young}[15.5pt][c] 
]=![cyan!99] 1  & =]![cyan!99] & ![pink!99] 2' \ynobottom\\ 
, & , & ![pink!99] \ynotop \ynobottom \\ 
, & , & ![pink!99] \ynotop \end{young}   
\qquord 
  \begin{young}[15.5pt][c]     
  ![cyan!99]1 \\ 
  ![pink!99]2'  \ynobottom\\ 
  ![pink!99] \ynotop \end{young}
  \mapsto \begin{young}[15.5pt][c]    
  ![pink!99] 2' \ynobottom \\ 
  ![pink!99] \ynotop \\ 
  ![cyan!99] 1    \end{young} .
\]

\item[(c)] In all other cases we set $ \cH := \tilde\cH$ and $ \cV :=\tilde\cV$.
\end{itemize}
\end{definition}

\begin{definition}
\label{unswap-def}
Suppose $\tilde\cH_1,\tilde\cH_2,\dots,\tilde\cH_k$ 
and $\tilde\cV_1,\tilde\cV_2,\dots,\tilde\cV_l$
are the distinct bars of $\tilde T$ containing $1$ and $2'$, respectively,
listed in the same orders as in Definition~\ref{swap-def} (i.e., the vertical bars filled with $2'$ are ordered from bottom to top and the horizontal bars filled with $1$ are ordered from right to left).  For any single bar $\tilde \cV \in \{\tilde\cV_1,\tilde\cV_2,\dots,\tilde\cV_l\}$, 
inductively define  
\[ ( T_k,  \cV_k) := (\tilde T,\tilde \cV)
\quand  ( T_{i-1},\cV_{i-1}) := 
\unswap( T_{i},\cV_{i}, \cH_i)\text{ for $i=k,\dots,3,2,1$,}
\]
and then set $\unswap(\tilde T,\tilde \cV)  :=  T_0$.
Finally let
\[ \unswap(\tilde T) := \unswap(\cdots \unswap(\unswap(\tilde T,\tilde \cV_l),\tilde\cV_{l-1}) \dots, \tilde\cV_1).\]
\end{definition}

\begin{example}
In the image below, the bars $\tilde\cH_1$, $\tilde\cH_2$, $\tilde\cH_3$, $\tilde\cH_4$, and $\tilde\cH_5$ are respectively gray, green, yellow, blue, and white; the bars $\tilde\cV_1$, $\tilde\cV_2$, and $\tilde\cV_3$ are respectively orange, pink, and red.  We illustrate each individual operation of the form $\unswap( -, \cV, \cH)$ that results in a new tableau (i.e., that does not fall in case (c) of Definition~\ref{def:unswap}):
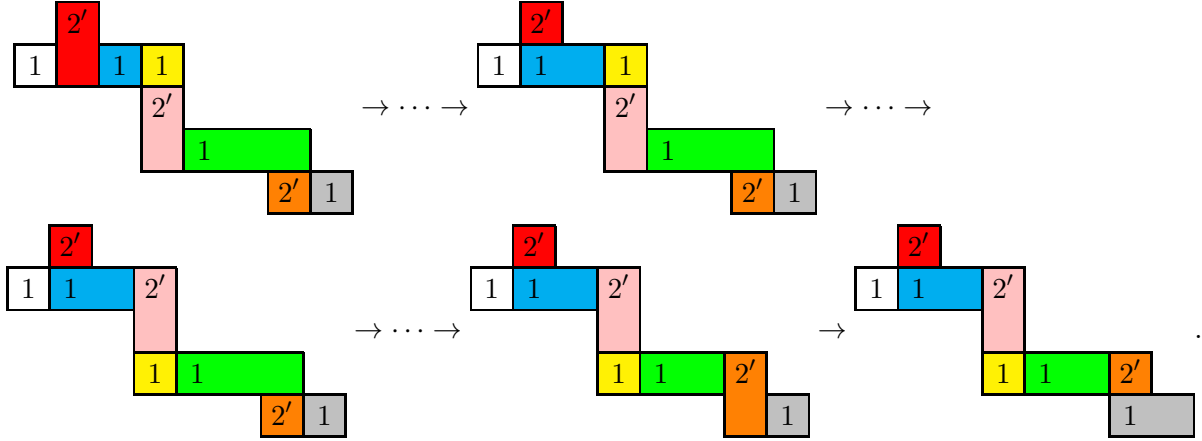
\begin{multline*}
\text{$\begin{young}[15.5pt][c] , & ![red!99] 2' \ynobottom \\ 1 & ![red!99] \ynotop & ![cyan!99] 1 & ![yellow!99] 1 \\ , & , & , & ![pink!99] 2' \ynobottom \\ , & , & , & ![pink!99] \ynotop & ]=![green!99] 1 & ==![green!99]  & =]![green!99]  \\ , & , & , & , & , & , & ![orange!99] 2' & ]=]![lightgray!99] 1\end{young}$}
\to \cdots \to 
\text{$\begin{young}[15.5pt][c] , & ![red!99] 2' \\ 1 & ]=![cyan!99] 1 & =]![cyan!99]  & ![yellow!99] 1 \\ , & , & , & ![pink!99] 2' \ynobottom \\ , & , & , & ![pink!99] \ynotop & ]=![green!99] 1 & ==![green!99]  & =]![green!99]  \\ , & , & , & , & , & , & ![orange!99] 2' & ]=]![lightgray!99] 1\end{young}$}
\to \cdots \to 
\\
\text{$\begin{young}[15.5pt][c] , & ![red!99] 2' \\ 1 & ]=![cyan!99] 1 & =]![cyan!99]  & ![pink!99] 2' \ynobottom \\ , & , & , & ![pink!99] \ynotop \\ , & , & , & ![yellow!99] 1 & ]=![green!99] 1 & ==![green!99]  & =]![green!99]  \\ , & , & , & , & , & , & ![orange!99] 2' & ]=]![lightgray!99] 1\end{young}$}
\to \cdots \to
\text{$\begin{young}[15.5pt][c] , & ![red!99] 2' \\ 1 & ]=![cyan!99] 1 & =]![cyan!99]  & ![pink!99] 2' \ynobottom \\ , & , & , & ![pink!99] \ynotop \\ , & , & , & ![yellow!99] 1 & ]=![green!99] 1 & =]![green!99]  & ![orange!99] 2' \ynobottom \\ , & , & , & , & , & , & ![orange!99] \ynotop & ]=]![lightgray!99] 1\end{young}$}
\to
\text{$\begin{young}[15.5pt][c] , & ![red!99] 2' \\ 1 & ]=![cyan!99] 1 & =]![cyan!99]  & ![pink!99] 2' \ynobottom \\ , & , & , & ![pink!99]  \ynotop \\ , & , & , & ![yellow!99] 1 & ]=![green!99] 1 & =]![green!99]  & ![orange!99] 2' \\ , & , & , & , & , & , & ]=![lightgray!99] 1 & =]![lightgray!99] \end{young}$}.
\end{multline*}
\end{example}
 
\begin{proposition}
\label{prop:inter to semi}
If $\tilde T$ is a sorted shifted bar tableau then $\unswap(\tilde T)$ is a semistandard shifted bar tableau of the same shape and weight. 
\end{proposition}

\begin{proof}
Let $\tilde T$ be a sorted shifted bar tableau. By construction $\tilde T$ and $\unswap(\tilde T)$ have the same shape and same weight.
First suppose that $\tilde T$ does not contain two diagonal boxes.  In this case, by comparing the definitions of sorted and semistandard ShBT, we see that rotating $\tilde T$ by $180^\circ$ produces a semistandard ShBT $\tilde T^r$.  Moreover, by comparing Definitions~\ref{def:swap}, \ref{swap-def} with Definitions~\ref{def:unswap}, \ref{unswap-def}, we see that $\swap(\tilde T^r) = \unswap(\tilde T)^r$, and consequently that $\unswap(\tilde T) = \swap(\tilde T^r)^r$ is a semistandard ShBT.  

Otherwise, suppose $\tilde T$ has two diagonal boxes.  There are four possible arrangements of the entries the first two diagonals of $\tilde T$,
which we illustrate as follows; in all cases, $\tilde{\cV}_l$ is colored in pink and $\tilde{\cH}_k$ is colored in cyan:
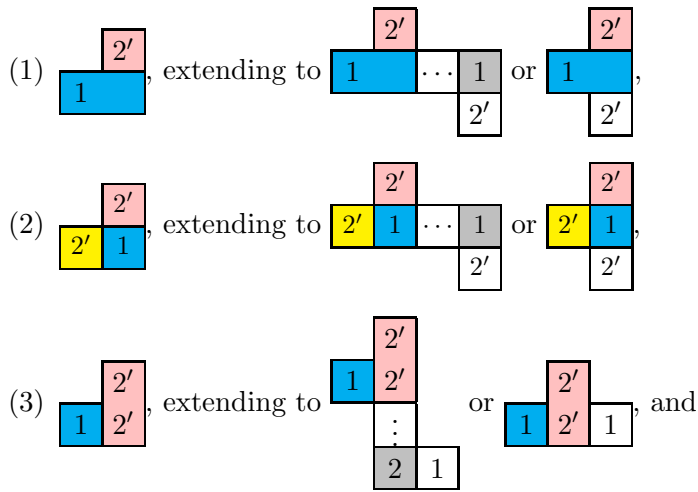
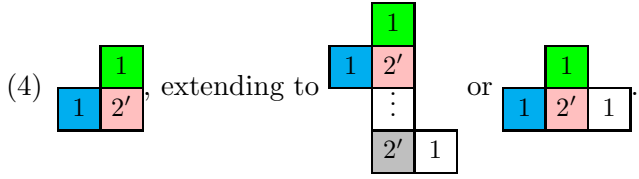
\ben
\item[(1)]
$\begin{young}[15.5pt][c] 
  , &  ![pink!99]2' \\
]=![cyan!99]1 & =]![cyan!99]  
\end{young}$,
 extending to $\begin{young}[15.5pt][c] 
  , &  ![pink!99]2' \\
]=![cyan!99]1 & =]![cyan!99] & \cdots & ]=]![lightgray!99]1 \\
, & , & , & 2'  
\end{young}$ or  $\begin{young}[15.5pt][c] 
  , &  ![pink!99]2' \\
]=![cyan!99]1 & =]![cyan!99]  \\
, & 2'  
\end{young}$, 

\item[(2)]
$
\ytab{ \none & *(pink) 2' \\ *(yellow) 2' & *(cyan) 1}$,   extending to  $\ytab{ \none & *(pink) 2' \\ *(yellow) 2' & *(cyan) 1 & \cdots & *(lightgray) 1 \\ \none & \none & \none & 2'}$ or $\ytab{ \none & *(pink) 2' \\ *(yellow) 2' & *(cyan) 1 \\  \none & 2'}$,

\item[(3)]
$\begin{young}[15.5pt][c]  , & ![pink!99]2' \ynobottom \\
![cyan!99]1 & ![pink!99]2' \ynotop\end{young}$, 
  extending to $\begin{young}[15.5pt][c]  , & ![pink!99]2' \ynobottom \\
![cyan!99]1 & ![pink!99]2' \ynotop \\
, & \vdots \\
, & ![lightgray!99]2 & 1 \end{young}$ 
 or $\begin{young}[15.5pt][c]  , & ![pink!99]2' \ynobottom \\
![cyan!99]1 & ![pink!99]2' \ynotop & 1\end{young}$, and

\item[(4)]
$\ytab{ \none & *(green) 1 \\ *(cyan) 1 & *(pink) 2'}$,  extending to $\ytab{ \none & *(green) 1 \\ *(cyan) 1 & *(pink) 2' \\ \none & \vdots \\ \none & *(lightgray) 2' & 1} $ or $\ytab{ \none & *(green) 1 \\ *(cyan) 1 & *(pink) 2' & 1}.$
\een
In the first two cases, one has $\unswap(\tilde T, \tilde{\cV}_l) = \tilde T$.  In these cases, 
deleting the unique box of $\tilde{\cV}_l$ from $\tilde T$ produces a sorted tableau
$\tilde T - \tilde{\cV}_l$ with only one diagonal box.  Therefore, by the argument in the previous paragraph, $\unswap(\tilde T - \tilde{\cV}_l)$ is a semistandard tableau. This tableau becomes $\unswap(\tilde T) = \unswap(\tilde T - \tilde{\cV}_l) \cup \{\tilde{\cV}_l\}$ when we add back the unique box of $\tilde{\cV}_l$ at the end.  Since the added box always contains $2'$, the result is semistandard, as needed.

 In the latter two cases, one has  $\unswap(\tilde T,  \tilde{\cV}_l, \tilde{\cH}_k) = (\tilde T, \tilde{\cV}_l)$, i.e., the very first swap belongs to case (c) of Definition~\ref{def:unswap} and so does not change anything.  In these cases, deleting the unique box of $\tilde{\cH}_k$ produces a sorted tableau $\tilde T - \tilde{\cH}_k$ with only one diagonal box.  Again by the argument in the first paragraph of this proof, we deduce that $\unswap(\tilde T - \tilde{\cH}_k)$ is a semistandard tableau. This tableau becomes $\unswap(\tilde T) = \unswap(\tilde T - \tilde{\cH}_k) \cup \{\tilde{\cH}_k\}$ when we add back the unique box of $\tilde{\cH}_k$ at the end, which preserves semistandard-ness.  This completes the proof.
\end{proof}

\begin{theorem}
\label{thm:inverse bijections}
The operations $\swap$ and $\unswap$
are inverse shape- and weight-preserving bijections between
semistandard and sorted shifted bar tableaux with all entries in $\{1', 1, 2', 2\}$.
\end{theorem}

\begin{proof}
By Propositions~\ref{prop:semi to inter} and~\ref{prop:inter to semi}, the two maps are shape- and weight-preserving and they have the correct domain and codomain.  Since each swap preserves the relative orders of the $\cV_i$ and of the $\cH_j$, it is immediate from the definitions that $\unswap(\swap(T)) = T$ for any semistandard ShBT $T$ and $\swap(\unswap(\tilde T)) = \tilde T$ for any sorted ShBT $\tilde T$.  
\end{proof}

\subsection{Weight reversal}

In this subsection, $T$ denotes a fixed choice of sorted ShBT as specified in Definition~\ref{def:sorted}, again with all entries in $\{1' , 2' , 1 , 2\}$.  
We partition the bars of this tableau into four kinds of ``groups'' and then define a weight-reversing involution locally group-by-group.  

It will be convenient in what follows to refer to the following \defn{conjugation} operation on sorted ShBT; all subsequent constructions will be symmetric under conjugation.
For this definition, we fix a large positive integer $N$ such that all boxes of $T$ are in $[N-1]\times [N-1]$.

\begin{definition}
\label{def:conjugation}
Let $T^\ast$ be the ShBT formed by applying the transformation
$(i,j) \mapsto (N-j,N-i)$ to the boxes in $T$ and the transformation $v \mapsto 3' - v$ to the entries in $T$. \end{definition}

For example,  $T\mapsto T^\ast$ would exchange
 
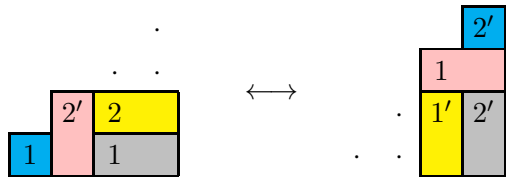
\[
 \begin{young}[15.5pt][c]
  , &   , & ,  &  ,\cdot \\
    , &   , & ,\cdot &  ,\cdot \\
  , &   ![pink!99]2' \ynobottom& ]=![yellow!99] 2 &  =]![yellow!99]  \\
 ]=]![cyan!99]1 & ![pink!99] \ynotop &  ]=![lightgray!99]1 &   =]![lightgray!99]  \end{young}
 \qquad\longleftrightarrow \quad
 \begin{young}[15.5pt][c] ,  &,  &, &  ![cyan!99]2' \\
,  &,    & ]=![pink!99]1 & =]![pink!99]  \\
,  & ,\cdot    & ]=]![yellow!99] 1' \ynobottom & ]=]![lightgray!99] 2' \ynobottom  \\
,\cdot   & ,\cdot &  ]=]![yellow!99] \ynotop & ]=]![lightgray!99] \ynotop  \end{young}
 \]
 for an appropriate choice of $N$. 
Since the examples in \eqref{inter-ex1} are conjugate to those in \eqref{inter-ex2}, this operation does indeed give another sorted ShBT.

Let $\cA$ and $\cB$ be bars in $T$.
We write $\cA \approx \cB$ if the two bars have unprimed entries in the same column or have primed entries in the same row. 
We also write $\cA \approx \cB$ if $T$ has two diagonal boxes $(i,i)$ and $(i+1,i+1)$, 
the bar $\cA$ consists of one of the two diagonal boxes (and no off-diagonal boxes), and the bar $\cB$ contains the box $(i,i+1)$, or vice-versa.  
An equivalence class  
 for the transitive closure of $\approx$ 
is a \defn{two-row-group} 
(respectively, a \defn{two-column-group})
if the class contains at least two bars in $T$ and the bars in the class fit entirely in two rows (respectively, columns), with the following exception.
It is possible to have classes that consist of exactly three boxes (in two or three bars), two of which are diagonal,
and which therefore fit entirely into two rows and also into two columns; for example:
\begin{equation}
\label{2r2c}
\begin{young}[15.5pt][c] 
 , &  ![pink!99] 2  \\
]=![cyan!99] 1 & =]![cyan!99]
\end{young},
\qquad
\begin{young}[15.5pt][c] 
 , &  ![pink!99] 1  \\
![cyan!99] 1' & ![lightgray!99] 2'
\end{young},
\qquord
\begin{young}[15.5pt][c] 
 , &  ![pink!99] 2'  \\
![cyan!99] 1' & ![lightgray!99] 2'
\end{young}.
\end{equation}
We refer to these classes as \defn{two-row-groups} (respectively, \defn{two-column-groups}) if and only if their non-diagonal entry is unprimed (respectively, primed).  Thus, of the three tableaux in \eqref{2r2c}, the first is a two-row-group and the others are two-column-groups.

Here is our first nontrivial observation about these equivalence classes.

\begin{proposition}\label{2row-prop}
Every non-singleton equivalence class for the transitive closure of  
$\approx$ is a two-row-group or a two-column-group.  Every two-row-group is of one of the forms
 \[
 \ytab{ 2 & \cdots &2 & \cdots & 2 \\
 \none & \none &1 & \cdots & 1 & \cdots & 1 }
 \quord
  \ytab{ \none & y & 2 &  \cdots &2  \\
 x & 1 & 1 & \cdots &1  & \cdots & 1 } \quad\text{(with some division into bars),}
 \]
 that is, the boxes in this union are either of the form
$(\{i\} \times [j_2,j_4])  \sqcup (\{i+1\} \times [j_1,j_3])$
 where $i \leq j_1\leq j_2\leq j_3 \leq j_4$ or 
of the form $(\{i\} \times [i, j_4])  \sqcup (\{i+1\} \times [i + 1, j_3])$
where $i+1 \leq j_3 \leq j_4$; and every two-column-group is conjugate to one of these.  Also, if a two-row-group has two diagonal boxes (marked $x$, $y$ in the illustration), then these are filled in one of the following six ways: 
\[
\begin{young}[15.5pt][c]
, & ![pink!99] 2 \\
![cyan!99] 1' & ![lightgray!99] 1
\end{young},
\qquad
\begin{young}[15.5pt][c]
, & ![pink!99] 2 \\
![cyan!99] 2' & ![lightgray!99] 1
\end{young},
\qquad
\begin{young}[15.5pt][c]
, & ![pink!99] 2 \\
![cyan!99] 1 & ![lightgray!99] 1
\end{young},
\qquad
\begin{young}[15.5pt][c]
, & ![pink!99] 2 \\
]=![cyan!99] 1 & =]![cyan!99]
\end{young},
\qquad
\begin{young}[15.5pt][c]
, & ![pink!99] 2' \\
![cyan!99] 2' & ![lightgray!99] 1
\end{young},
\quord
\begin{young}[15.5pt][c]
, & ![pink!99] 2' \\
]=![cyan!99] 1 & =]![cyan!99]
\end{young};
\]
and in a two-column-group, the options are conjugate to these.
\end{proposition}
\begin{proof}
First, consider the equivalence classes that consist of all unprimed bars.  In this case, if $\cA \approx \cB$ 
then the relations satisfied by the entries of a sorted ShBT imply that 
when the bar $\cA$ is filled with $1$ and lies in row $i$, the bar $\cB$ must be filled with $2$s and lie in row $i + 1$, and conversely. Hence the entire equivalence class is restricted to two consecutive rows $i$ and $i + 1$, with row $i$ containing $1$s and row $i + 1$ containing $2$s.  Furthermore, row $i + 1$ cannot extend to the right of row $i$ because the only way to complete the diagram
\[
\begin{young}[15.5pt][c]
\cdots & 2 & 2 \\
\cdots & 1 & {?}
\end{young}
\]
to a valid sorted ShBT is if the box $\begin{young}{?}\end{young}$ is filled with a $1$.  Similarly, row $i$ cannot extend to the left of row $i + 1$ except in the case
\[
\begin{young}[15.5pt][c]
,\cdot & 2 & \cdots \\
 1 & 1 & \cdots
\end{young}
\]
where it does so at the diagonal.  The case of equivalence classes containing all primed bars is the conjugate of this case.

We now consider the case of equivalence classes that contain both primed and unprimed bars.  In this case, the exceptional rule defining $\approx$ must come into play: the tableau $T$ must have two diagonal boxes $(i, i)$ and $(i + 1, i + 1)$ for some $i$, at least one of which must be a bar unto itself.  Any skew shifted shape that contains $(i, i)$ and $(i + 1, i + 1)$ also contains $(i, i + 1)$; up to conjugacy, we may assume that the box $(i, i + 1)$ is filled with an unprimed entry.  In fact, this unprimed entry must be $1$: no tableau of the form $\begin{young}[15.5pt][c]
, & y \\
 x & 2 & \cdots
\end{young}$ with entries in $\{1', 2', 1, 2\}$ is a sorted ShBT.  Let $\cA$ be the bar containing the box $(i, i + 1)$.  By hypothesis, one of the two boxes $(i, i)$ and $(i + 1, i + 1)$ must belong to a one-box bar that is filled with either $1'$ or $2'$.  

First, suppose that the box $(i, i)$ is a singleton bar $\cB$.  In this case, its only relation under $\approx$ is $\cA \approx \cB$, and the box $(i + 1, i + 1)$ could either be filled with $2'$ (if $\cB$ is also filled with $2'$) or with $2$.  In the case that $(i + 1, i + 1)$ is filled with $2'$, its unique relation under $\approx$ is with $\cA$, and so in either case the logic of the previous paragraph applies to conclude that all boxes in the equivalence class belong to rows $i$ and $i + 1$.  

Second, suppose that box $(i + 1, i + 1)$ is filled with a primed entry.  By Definition~\ref{def:sorted}, it follows that this entry is $2'$, and that $(i, i)$ is either filled with $1$ or with $2'$.  In the former case, the equivalence class again belongs to two rows, for the same reasons.  In the latter case, it turns out that the box $(i, i)$ must also be a bar unto itself because it is not possible to complete the diagram
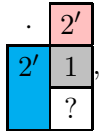
\[
\begin{young}[15.5pt][c] 
 ,\cdot &  ]=]![pink!99] 2'  \\
 ![cyan!99]2' \ynobottom & ]=]![lightgray!99] 1 \\
    ![cyan!99] \ynotop & ]=] ?  
 \end{young},
\]
to a sorted ShBT.  This establishes all of our claims about two-row-groups.  The case of two-column-groups is equivalent after taking conjugates everywhere.
\end{proof}

Now suppose $\cA$ and $\cB$ are bars within $T$ that do not belong to any two-row-
or two-column-group.
We write $\cA \sim \cB$ if $\cA$ and $\cB$ 
are adjacent unprimed bars (in the same row) or
adjacent primed bars (in the same column).
We also write $\cA \sim \cB$ 
if the bars are adjacent and one bar consists of a single box on the diagonal.
An equivalence class for the transitive closure of $\sim$ is a \defn{one-row-group} (respectively, \defn{one-column-group}) if the bars in the class fit entirely in one row (respectively, one column), with the exception that an equivalence class containing a single one-box bar is a one-row-group if its entry is unprimed and a one-column-group if its entry is primed.

\begin{proposition}\label{1row-prop}
Every equivalence class for the transitive closure of $\sim$ is either a one-row- or one-column-group.  Every one-row-group is of one of forms
(with some division into bars)
\[
\ytab{1 & \cdots & 1 & 2 & \cdots & 2}, 
\quad
\ytab{1' & 1 & \cdots & 1 & 2 & \cdots & 2}, 
\quord
\ytab{2' & 1 & \cdots & 1 & 2 & \cdots & 2}, 
\]
where in all cases the number of $1$s or $2$s might be $0$, and the latter two cases occur only when the leftmost box is on the diagonal; and every one-column-group is conjugate to one of these.
 \end{proposition}
\begin{proof}
The case of equivalence classes that do not contain any diagonal box is straightforward from the definition.  If a tableau $T$ contains two diagonal boxes $(i, i)$ and $(i + 1, i + 1)$, then at least one of them belongs to a two-row- or two-column-group with $(i, i + 1)$ and so does not belong to any one-row- or one-column-group.  Thus, we may restrict our attention to the case of tableaux with a single diagonal box, and the equivalence class that contains that box.  In particular, we must show that the diagonal box cannot be in the same group with its neighbors to the east and south simultaneously.  It is a simple finite computation to check that in all $30$ sorted ShBT of shape $\begin{young}
, &   &  \\
,\cdot &  &  
\end{young}$, at least one of the neighbors of the diagonal box belongs to a two-row- or two-column-group.  Then the result follows immediately.
%
%
 \end{proof}

We now describe a collection of operations that can be applied locally to a given group.  These operations will all reverse the weight of the group, while globally replacing the sorted ShBT with another sorted ShBT.  They will also have predictable effects on which diagonal boxes carry primes (which will be important when considering the symmetry of the generating function $\jp_{\lambda/\mu}$).

First consider a one-row-group. We define an operation called \defn{toggling the labels} of the group, as follows.
If the group has a (necessarily unique) primed diagonal entry then remove its prime.
Suppose this results in a group with $p$ bars containing $1$ and $q$ bars containing $2$.
We order these bars from left to right in the usual way, unless there is a one-box bar on the diagonal containing $2$. In the latter case we consider the diagonal bar
to be the last bar in the group and order the remaining bars from left to right.
Finally, relabel   these $p+q$ bars so that first $q$ bars contain $1$ and 
the last $p$ bars contain $2$,
and then add back a prime to the one-box diagonal bar if this was removed earlier.
For example, this would give
 \be\label{toggling-eq} 
\ba
   \text{$\begin{young}[15.5pt][c]  ![lightgray!99]1 &  ]=![yellow!99]1 &  =]![yellow!99] &  ![cyan!99]1 & ]=![pink!99]2& ![pink!99] & =]![pink!99]   \end{young}$}
\mapsto
   \text{$\begin{young}[15.5pt][c]  ![lightgray!99]1 &  ]=![yellow!99]2 &  =]![yellow!99] &  ![cyan!99]2 & ]=![pink!99]2& ![pink!99] & =]![pink!99]   \end{young}$} & \;(p=3\text{ and }q=1),
 \\[-10pt]\\
   \text{$\begin{young}[15.5pt][c]  ![lightgray!99]1' &  ]=![yellow!99]1 &  =]![yellow!99] &  ![cyan!99]1 & ]=![pink!99]2& ![pink!99] & =]![pink!99]   \end{young}$}
\mapsto
   \text{$\begin{young}[15.5pt][c]  ![lightgray!99]1' &  ]=![yellow!99]2 &  =]![yellow!99] &  ![cyan!99]2 & ]=![pink!99]2& ![pink!99] & =]![pink!99]   \end{young}$}
 &\; (p=3\text{ and }q=1), 
 \\[-10pt]\\
\text{$\begin{young}[15.5pt][c] 2' &  ![lightgray!99]1 &  ]=![yellow!99]2 &  =]![yellow!99] &  ![cyan!99]2 & ]=![pink!99]2 & =]![pink!99]   \end{young}$}
\mapsto
\text{$\begin{young}[15.5pt][c] 2' &  ![lightgray!99]1 &  ]=![yellow!99]1 &  =]![yellow!99] &  ![cyan!99]1 & ]=![pink!99]1 & =]![pink!99]   \end{young}$}
 &\;(p=1\text{ and }q=4), \text{ and}
 \\[-10pt]\\
 \text{$\begin{young}[15.5pt][c]  ![lightgray!99]1 &  ]=![yellow!99]1 &  =]![yellow!99] &  ![cyan!99]1 & ]=![pink!99]1& ![pink!99] & =]![pink!99]   \end{young}$}
\mapsto
   \text{$\begin{young}[15.5pt][c]  ![lightgray!99]2 &  ]=![yellow!99]2 &  =]![yellow!99] &  ![cyan!99]2 & ]=![pink!99]2& ![pink!99] & =]![pink!99]   \end{young}$} & \;(p=4\text{ and }q=0).
 \ea\ee

Toggling the labels in a one-column-group is precisely the conjugate operation; for example, toggling the labels transforms
 \be\label{toggling-one-column-eq}
 \begin{young}[15.5pt][c]
 ![lightgray!99] 2' \\
 ![yellow!99] 2' \ynobottom \\
 ![yellow!99]  \ynotop \\
 ![cyan!99] 2' \\
 ![pink!99] 1' \ynobottom\\
 ![pink!99]  \ynotop\ynobottom\\
 ![pink!99]  \ynotop
 \end{young} \mapsto
  \begin{young}[15.5pt][c]
 ![lightgray!99] 2' \\
 ![yellow!99] 1' \ynobottom \\
 ![yellow!99]  \ynotop \\
 ![cyan!99] 1' \\
 ![pink!99] 1' \ynobottom\\
 ![pink!99]  \ynotop\ynobottom\\
 ![pink!99]  \ynotop
 \end{young},
 \qquad
  \begin{young}[15.5pt][c]
 ![lightgray!99] 2 \\
 ![yellow!99] 2' \ynobottom \\
 ![yellow!99]    \ynotop \\
 ![cyan!99] 2' \\
 ![pink!99] 1' \ynobottom \\
 ![pink!99]    \ynotop\ynobottom\\
 ![pink!99]    \ynotop
 \end{young} \mapsto
  \begin{young}[15.5pt][c]
 ![lightgray!99] 2 \\
 ![yellow!99] 1' \ynobottom \\
 ![yellow!99]    \ynotop \\
 ![cyan!99] 1' \\
 ![pink!99] 1' \ynobottom \\
 ![pink!99]    \ynotop\ynobottom\\
 ![pink!99]    \ynotop
 \end{young},
 \qquad
  \begin{young}[15.5pt][c]
   1 \\
 ![lightgray!99] 2' \\
 ![yellow!99] 1' \ynobottom \\
 ![yellow!99]    \ynotop \\
 ![cyan!99] 1' \\
 ![pink!99] 1' \ynobottom \\
 ![pink!99]    \ynotop 
 \end{young} \mapsto
  \begin{young}[15.5pt][c]
   1 \\
 ![lightgray!99] 2' \\
 ![yellow!99] 2' \ynobottom \\
 ![yellow!99]    \ynotop \\
 ![cyan!99] 2' \\
 ![pink!99] 2' \ynobottom \\
 ![pink!99]    \ynotop 
 \end{young},
\qquand
 \begin{young}[15.5pt][c]
 ![lightgray!99] 2' \\
 ![yellow!99] 2' \ynobottom \\
 ![yellow!99]  \ynotop \\
 ![cyan!99] 2' \\
 ![pink!99] 2' \ynobottom\\
 ![pink!99]  \ynotop\ynobottom\\
 ![pink!99]  \ynotop
 \end{young} \mapsto
  \begin{young}[15.5pt][c]
 ![lightgray!99] 1' \\
 ![yellow!99] 1' \ynobottom \\
 ![yellow!99]  \ynotop \\
 ![cyan!99] 1' \\
 ![pink!99] 1' \ynobottom\\
 ![pink!99]  \ynotop\ynobottom\\
 ![pink!99]  \ynotop
 \end{young}.
 \ee

Next consider a two-row-group with boxes in rows $i$ and $i+1$. First assume the group has at most one diagonal box. (In this case, by Proposition~\ref{2row-prop}, all bars in the group are filled with unprimed entries.)  We define a weight-reversing operation called \defn{toggling the bar divisions} in the group, as follows.  We form another two-row-group with the same entries 
that has a bar division between columns $j$ and $j+1$ in row $i$ (respectively, $i+1$)
 if and only if 
a bar division occurs between columns $j$ and $j+1$ in row $i+1$ (respectively, $i$)
in the starting group, as in
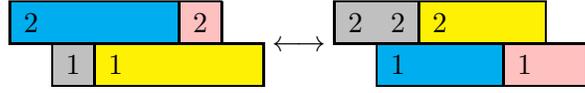
\[
  \begin{young}[15.5pt][c] 
]=![cyan!99] 2 & ==![cyan!99]  & ==![cyan!99]  & =]![cyan!99] &  ![pink!99]2    \\
, &  ![lightgray!99]1 &  ]=![yellow!99]1 &  ==![yellow!99] &  ==![yellow!99] & =]![yellow!99]   \end{young}
 \longleftrightarrow
\begin{young}[15.5pt][c]
 ]=![lightgray!99]2 & =]![lightgray!99]2 & ]=![yellow!99] 2 & ==![yellow!99] &  =]![yellow!99]    \\
, &  ]=![cyan!99]1 &  ==![cyan!99] &  =]![cyan!99] &  ]=![pink!99]1 & =]![pink!99]
\end{young}
 .
\]
This operation is obviously weight-reversing.

The situation of two-row-groups that contain two diagonal boxes is more complicated.  We again introduce a notion of \defn{toggling the bar divisions} that is not necessarily weight-reversing (or even well defined on ShBT); it will be used as one step in the final weight-reversing map of Definition~\ref{def:weight reversal} below.  Given a two-row-group with two diagonal boxes $(i,i)$ and $(i,i+1)$, 
we define the operation of toggling the bar divisions  
as above, except that we mandate that boxes $(i,i)$ and $(i,i+1)$ 
belong to the same bar in the toggled group if and only if they belong to the same bar in the original group; for example, as in 
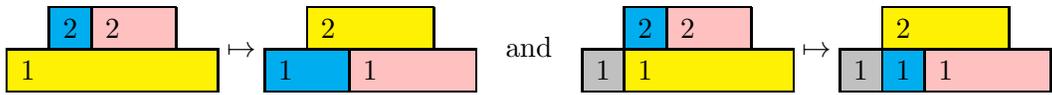
\[
  \begin{young}[15.5pt][c] , &![cyan!99] 2 & ]=![pink!99]2 &  =]![pink!99]    \\
 ]=![yellow!99]1 &  ==![yellow!99] &  ==![yellow!99] &  ==![yellow!99] & =]![yellow!99]   \end{young}
 \mapsto
   \begin{young}[15.5pt][c] , &]=![yellow!99] 2 & ==![yellow!99] &  =]![yellow!99]    \\
 ]=![cyan!99]1 &  =]![cyan!99] &  ]=![pink!99]1 &  ==![pink!99] & =]![pink!99]   \end{young}
 \quand
   \begin{young}[15.5pt][c] , &![cyan!99] 2 & ]=![pink!99]2 &  =]![pink!99]    \\
 ![lightgray!99]1 &  ]=![yellow!99]1 &  ==![yellow!99] &  ==![yellow!99] & =]![yellow!99]   \end{young}
 \mapsto
   \begin{young}[15.5pt][c] , &]=![yellow!99] 2 & ==![yellow!99] &  =]![yellow!99]    \\
 ![lightgray!99]1 &  ]=]![cyan!99]1 &  ]=![pink!99]1 &  ==![pink!99] & =]![pink!99]   \end{young}
 .
\]
If the box $(i+1,i+1)$ contains $2'$ and box $(i+1,i+2)$ is part of the group, then our toggling operation
 could result in a tableau that is not a valid ShBT; in practice we will never apply our operation when this case would arise.
%

There are analogous toggling operations for two-column-groups. They are precisely the conjugate (in the sense of Definition~\ref{def:conjugation}) of the operations defined above.  For example, to \defn{toggle the bar divisions} in a two-column-group transforms
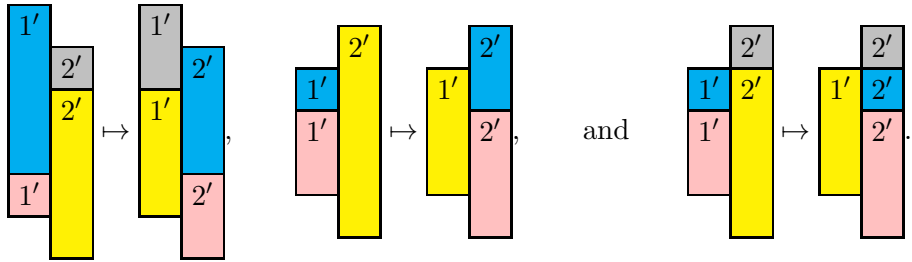
\[
  \begin{young}[15.5pt][c]
![cyan!99] 1' \ynobottom \\
![cyan!99]    \ynotop\ynobottom &![lightgray!99]  2' \\
![cyan!99]    \ynotop\ynobottom &![yellow!99]  2' \ynobottom \\
![cyan!99]    \ynotop &![yellow!99]   \ynotop\ynobottom\\
![pink!99] 1'&![yellow!99]  \ynotop\ynobottom \\
 , &![yellow!99]  \ynotop
 \end{young}
 \mapsto
   \begin{young}[15.5pt][c]
![lightgray!99] 1' \ynobottom \\
![lightgray!99]   \ynotop &![cyan!99]  2'\ynobottom \\
![yellow!99] 1' \ynobottom &![cyan!99]   \ynotop\ynobottom \\
![yellow!99] \ynotop\ynobottom &![cyan!99]   \ynotop\\
![yellow!99] \ynotop &![pink!99]  2' \ynobottom \\
 , &![pink!99]  \ynotop
  \end{young},
  \qquad
  \begin{young}[15.5pt][c]
 , &![yellow!99] 2' \ynobottom \\
![cyan!99] 1'&![yellow!99] \ynobottom\ynotop\\
![pink!99] 1' \ynobottom &![yellow!99] \ynobottom\ynotop \\
![pink!99]  \ynotop &![yellow!99]  \ynobottom\ynotop\\
, &![yellow!99]  \ynotop 
  \end{young}
 \mapsto
   \begin{young}[15.5pt][c]
 , &![cyan!99] 2' \ynobottom \\
![yellow!99] 1' \ynobottom &![cyan!99] \ynotop \\
![yellow!99] \ynobottom\ynotop &![pink!99] 2' \ynobottom\\
![yellow!99] \ynotop &![pink!99]  \ynobottom \ynotop\\
, &![pink!99]  \ynotop 
  \end{young},
     \qquand
  \begin{young}[15.5pt][c]
 , &![lightgray!99] 2' \\
![cyan!99] 1'&![yellow!99] 2' \ynobottom\\
![pink!99] \ynobottom 1'&![yellow!99]  \ynobottom\ynotop\\
![pink!99] \ynotop &![yellow!99]   \ynobottom\ynotop\\
, &![yellow!99]  \ynotop
  \end{young}
 \mapsto
   \begin{young}[15.5pt][c]
 , &![lightgray!99] 2' \\
![yellow!99] 1' \ynobottom &![cyan!99] 2'\\
![yellow!99]    \ynobottom\ynotop &![pink!99] 2' \ynobottom\\
![yellow!99]    \ynotop &![pink!99]   \ynobottom\ynotop\\
, &![pink!99]   \ynotop
  \end{young}.
  \]
We now combine all of these local transformations
into the following weight-reversing operator.

\begin{definition}
\label{def:weight reversal}
Let $T$ be a sorted shifted bar tableau with all entries in $\{1' , 2' , 1 , 2\}$.
We form another shifted bar tableau, denoted $\wtrev(T)$, by transforming $T$ as follows:
\ben
\item First toggle the labels in each one-row-group and one-column-group.
Then toggle the bar divisions in each two-row-group and two-column-group,
excluding the unique group with two diagonal boxes (if this exists).

\item Suppose $T$ has a group with two diagonal boxes.
If this is a two-row-group with diagonal boxes $(i,i)$ and $(i+1,i+1)$,
then proceed according to the following cases:
\ben
\item[(R1)] If $(i,i)$ is a one-box bar containing
  $1$ (respectively, $2'$) and the entry in $(i+1,i+1)$ is $2$,
then  toggle the bar divisions 
and  change the entry in $(i,i)$ to $2'$ (respectively, $1$):
  
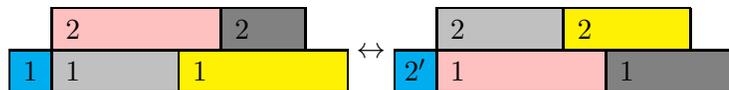
\[
  \begin{young}[15.5pt][c] , &  ]=![pink!99] 2 &   ==![pink!99] &   ==![pink!99]  &  =]![pink!99]  &  ]=![gray!99] 2  &  =]![gray!99]          \\
 ![cyan!99]1 &  ]=![lightgray!99] 1 &  ==![lightgray!99] &  =]![lightgray!99]  & ]=![yellow!99] 1  & ==![yellow!99]    & ==![yellow!99]   & =]![yellow!99] 
\end{young}
\leftrightarrow
  \begin{young}[15.5pt][c] , &  ]=![lightgray!99] 2 &   ==![lightgray!99] &   =]![lightgray!99]  &  ]=![yellow!99] 2 &  ==![yellow!99]   &  =]![yellow!99]          \\
 ![cyan!99]2' &  ]=![pink!99] 1 &  ==![pink!99] &  ==![pink!99]  & =]![pink!99]   & ]=![gray!99] 1   & ==![gray!99]   & =]![gray!99] 
\end{young}.
 \]
 
 \item[(R2)] Suppose $(i,i)$ is a one-box bar containing
  $1'$, 
   boxes $(i,i+1)$ and $(i,i+2)$ belong to the same bar, and
   $(i+1,i+2)$ is in the group.
     In this case the entry in $(i+1,i+1)$ is necessarily $2$ by Proposition~\ref{2row-prop}.
  First  change the entry in $(i,i)$ to $1$ and merge this box into the bar of
$(i,i+1)$. Next, toggle the bar divisions.  Finally, change $(i+1,i+1)$ to a one-box bar containing $2'$:
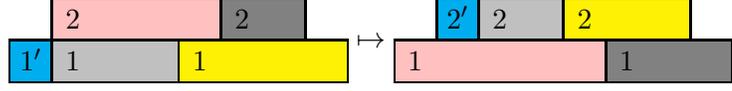
  \[
  \begin{young}[15.5pt][c] , &  ]=![pink!99] 2 &   ==![pink!99] &   ==![pink!99]  &  =]![pink!99]  &  ]=![gray!99] 2  &  =]![gray!99]          \\
 ![cyan!99]1' &  ]=![lightgray!99] 1 &  ==![lightgray!99] &  =]![lightgray!99]  & ]=![yellow!99] 1  & ==![yellow!99]    & ==![yellow!99]    & =]![yellow!99]  \end{young}
\mapsto
  \begin{young}[15.5pt][c] , &  ![cyan!99] 2' &   ]=![lightgray!99] 2&   =]![lightgray!99]  &  ]=![yellow!99] 2 &  ==![yellow!99]   &  =]![yellow!99]         \\
 ]=![pink!99]1 &  ==![pink!99]  &  ==![pink!99] &  ==![pink!99]  & =]![pink!99]   & ]=![gray!99] 1   & ==![gray!99]   & =]![gray!99]  \end{young}.
 \]

 \item[(R3)] Suppose $(i+1,i+1)$ is a one-box bar containing $2'$,
 the boxes $(i,i)$ and $(i,i+1)$ belong to the same bar (necessarily containing $1$), and 
 $(i+1,i+2)$ 
 is in the group. First change the entry in $(i+1,i+1)$ to $2$ and merge this box 
 into the bar of $(i+1,i+2)$. Next, toggle the bar divisions.  Finally, change $(i,i)$ to a one-box bar containing $1'$.
 This is illustrated by reversing the direction of the example in case (R2).

  \item[(R4)] Suppose $(i,i)$ is a one-box bar containing
  $1'$, but the boxes $(i,i+1)$ and $(i,i+2)$ do not belong to the same bar or
  the box $(i+1,i+2)$ is not part of the group.
  In this case the entry in $(i+1,i+1)$ is necessarily $2$ by Proposition~\ref{2row-prop}.
  Then toggle the bar divisions 
and  change the entries in both $(i,i)$ and $(i+1,i+1)$ to $2'$:
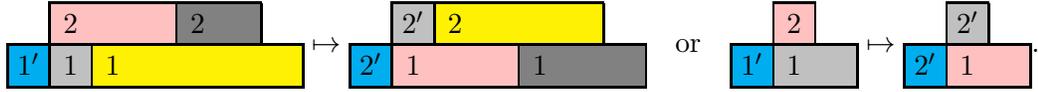
  \[
  \begin{young}[15.5pt][c] , &  ]=![pink!99] 2 &   ==![pink!99] &     =]![pink!99]  &  ]=![gray!99] 2  &  =]![gray!99]         \\
 ]=]![cyan!99]1' &  ![lightgray!99] 1 &  ]=![yellow!99] 1&  ==![yellow!99]  & ==![yellow!99]   & ==![yellow!99]    & =]![yellow!99]   \end{young}
 \mapsto
  \begin{young}[15.5pt][c] , &  ![lightgray!99] 2' &   ]=![yellow!99] 2&   ==![yellow!99]  &  ==![yellow!99]   &  =]![yellow!99]          \\
 ]=]![cyan!99]2' &  ]=![pink!99] 1 &  ==![pink!99] &  =]![pink!99]  & ]=![gray!99] 1  & ==![gray!99]    & =]![gray!99]   \end{young}
 \quord
  \begin{young}[15.5pt][c] , &  ![pink!99] 2          \\
 ]=]![cyan!99]1' &  ]=![lightgray!99] 1 &  =]![lightgray!99]     \end{young}
 \mapsto
  \begin{young}[15.5pt][c] , &  ![lightgray!99] 2'          \\
 ]=]![cyan!99]2' &  ]=![pink!99] 1 &  =]![pink!99]       \end{young}.
 \]

  \item[(R5)]
  If  $(i,i)$ and $(i+1,i+1)$ are both one-box bars containing
  $2'$, then change the entries in $(i,i)$ and $(i+1,i+1)$ to $1'$ and $2$, respectively, and toggle the bar divisions.
   This is illustrated by reversing the examples in case (R4).
   
   \item[(R6)] If none of (R1)--(R5) apply, just toggle the bar divisions in the two-row-group.
 \een
 If
   the two diagonal boxes instead occur in a two-column-group,
  then replace this group by its conjugate, transform
  that two-row-group according to cases (R1)--(R6), and finally conjugate the result.  Explicitly, this means that if there is a two-column-group with diagonal boxes $(i-1,i-1)$ and $(i,i)$,
then we proceed according to the following cases (illustrated in Figure~\ref{fig:column weight reversal}):
\ben
\item[(C1)] If $(i,i)$ is a one-box bar containing
  $2'$ (respectively, $1$) and the entry in $(i-1,i-1)$ is $1'$,
then  toggle the bar divisions 
and  change the entry in $(i,i)$ to $1$ (respectively, $2'$).

 \item[(C2)] Suppose $(i,i)$ is a one-box bar containing
  $2$, 
   boxes $(i-1,i)$ and $(i-2,i)$ belong to the same bar, and
   $(i-1,i-2)$ is in the group.
    In this case the entry in $(i-1,i-1)$ is necessarily $1'$ by Proposition~\ref{2row-prop}.
First  change the entry in $(i,i)$ to $2'$ and merge this box into the bar of
$(i-1,i)$. Next, toggle the bar divisions.  Finally, change $(i-1,i-1)$ to a one-box bar containing $1$.

 \item[(C3)] Suppose $(i-1,i-1)$ is a one-box bar containing $1$,
 the boxes $(i,i)$ and $(i-1,i)$ belong to the same bar (necessarily containing $2'$), and 
 $(i-2,i-1)$ 
 is in the group. First change the entry in $(i-1,i-1)$ to $1'$ and merge this box 
 into the bar of $(i-2,i-1)$. Next, toggle the bar divisions.
 Finally, change $(i,i)$ to a one-box bar containing $2$.

  \item[(C4)] Suppose $(i,i)$ is a one-box bar containing
  $2$, but the boxes $(i-1,i)$ and $(i-2,i)$ do not belong to the same bar or
  the box $(i-2,i-1)$ is not part of the group.
      In this case the entry in $(i-1,i-1)$ is necessarily $1'$ by Proposition~\ref{2row-prop}.
  Then toggle the bar divisions 
and  change the entries in both $(i,i)$ and $(i-1,i-1)$ to $1$.

  \item[(C5)]
  If  $(i,i)$ and $(i-1,i-1)$ are both one-box bars containing
  $1$, then change the entries in $(i,i)$ and $(i-1,i-1)$ to $2$ and $1'$, respectively, and 
  toggle the bar divisions.
   
   \item[(C6)] If none of (C1)--(C5) apply, just toggle the bar divisions in the two-column-group.
 \een

\een

\end{definition}

 \begin{figure}
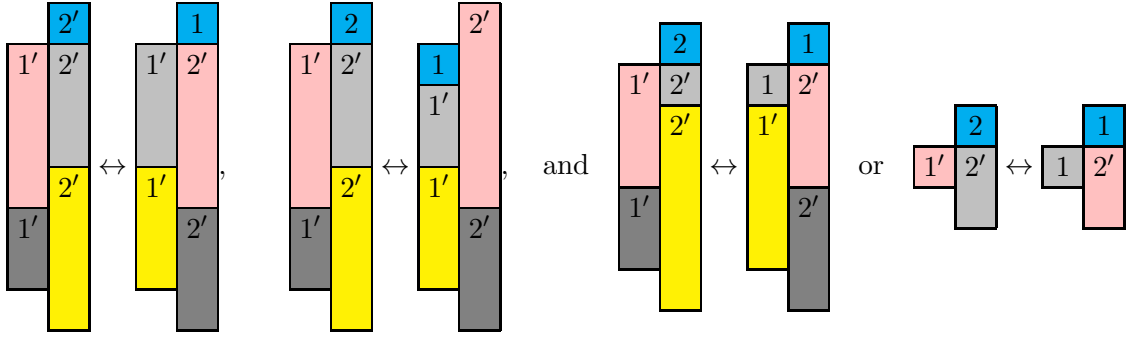

 \[
  \begin{young}[15.pt][c] , & ![cyan!99] 2' \\
 ![pink!99]  1' \ynobottom & ![lightgray!99]  2' \ynobottom \\
  ![pink!99] \ynobottom\ynotop& ![lightgray!99]  \ynobottom\ynotop\\
  ![pink!99] \ynobottom\ynotop & ![lightgray!99]\ynotop \\
  ![pink!99] \ynotop& ![yellow!99] 2'  \ynobottom\\
  ![gray!99] 1' \ynobottom& ![yellow!99]  \ynobottom\ynotop\\
 ![gray!99]  \ynotop& ![yellow!99] \ynobottom\ynotop \\
 , & ![yellow!99] \ynotop \end{young}
\leftrightarrow
  \begin{young}[15.pt][c] , & ![cyan!99] 1 \\
 ![lightgray!99]  1' \ynobottom & ![pink!99]  2'  \ynobottom \\
  ![lightgray!99] \ynobottom\ynotop& ![pink!99] \ynobottom\ynotop\\
  ![lightgray!99]  \ynotop& ![pink!99] \ynobottom\ynotop \\
  ![yellow!99] 1' \ynobottom& ![pink!99] \ynotop \\
  ![yellow!99]  \ynobottom\ynotop& ![gray!99] 2' \ynobottom \\
 ![yellow!99] \ynotop & ![gray!99] \ynobottom\ynotop\\
 , & ![gray!99] \ynotop  \end{young},
 \qquad
 \begin{young}[15.pt][c] , & ![cyan!99] 2 \\
 ![pink!99]  1' \ynobottom & ![lightgray!99]  2' \ynobottom \\
  ![pink!99] \ynotop\ynobottom & ![lightgray!99] \ynotop\ynobottom \\
  ![pink!99] \ynotop\ynobottom& ![lightgray!99] \ynotop \\
  ![pink!99] \ynotop & ![yellow!99] 2' \ynobottom \\
  ![gray!99] 1' \ynobottom & ![yellow!99] \ynotop\ynobottom \\
 ![gray!99] \ynotop & ![yellow!99] \ynotop\ynobottom  \\
 , & ![yellow!99] \ynotop  \end{young}
\leftrightarrow
  \begin{young}[15.pt][c] , & ![pink!99] 2' \ynobottom \\
 ![cyan!99]  1 & ![pink!99]  \ynotop\ynobottom \\
  ![lightgray!99] 1' \ynobottom & ![pink!99] \ynotop\ynobottom \\
  ![lightgray!99] \ynotop & ![pink!99] \ynotop\ynobottom \\
  ![yellow!99] 1' \ynobottom& ![pink!99] \ynotop \\
  ![yellow!99] \ynotop\ynobottom & ![gray!99] 2' \ynobottom\\
 ![yellow!99] \ynotop & ![gray!99] \ynotop\ynobottom \\
  , & ![gray!99] \ynotop  \end{young},
  \quand
 \begin{young}[15.pt][c] , & ![cyan!99] 2 \\
 ![pink!99]  1' \ynobottom & ![lightgray!99]  2' \\
  ![pink!99] \ynobottom\ynotop & ![yellow!99] 2'  \ynobottom\\
  ![pink!99] \ynotop & ![yellow!99] \ynobottom\ynotop \\
  ![gray!99] 1'  \ynobottom& ![yellow!99] \ynobottom\ynotop \\
  ![gray!99] \ynotop & ![yellow!99] \ynobottom\ynotop \\
 ,  & ![yellow!99] \ynotop  \end{young}
\leftrightarrow
  \begin{young}[15.pt][c] , & ![cyan!99] 1 \\
 ![lightgray!99]  1 & ![pink!99]  2' \ynobottom \\
  ![yellow!99] 1' \ynobottom& ![pink!99] \ynobottom\ynotop \\
  ![yellow!99] \ynobottom\ynotop & ![pink!99] \ynotop \\
  ![yellow!99] \ynobottom\ynotop & ![gray!99] 2' \ynobottom \\
  ![yellow!99] \ynotop & ![gray!99] \ynobottom\ynotop \\
 ,& ![gray!99] \ynotop  \end{young}
 \quord
  \begin{young}[15.pt][c] , & ![cyan!99] 2 \\
 ![pink!99]  1' & ![lightgray!99]  2' \ynobottom \\
  , & ![lightgray!99] \ynotop   \end{young}
\leftrightarrow
  \begin{young}[15.pt][c] , & ![cyan!99] 1 \\
 ![lightgray!99]  1 & ![pink!99]  2' \ynobottom \\
  , & ![pink!99] \ynotop   \end{young}
  \]
  \caption{The special cases (C1), (C2)--(C3), and (C4)--(C5) of Definition~\ref{def:weight reversal}}
  \label{fig:column weight reversal}
\end{figure}

\begin{theorem}
\label{thm:wt reversal}
The operation $\wtrev$ is a shape-preserving and weight-reversing involution of the set of sorted shifted bar tableaux
with all entries in $\{1',1,2',2\}$.
\end{theorem}

\begin{proof}
By construction, every bar in a sorted ShBT belongs to a group, and each local transformation (toggling the labels on one-row- and one-column-groups, toggling the bar divisions on two-row- and two-column-groups that contain at most one diagonal box, applying rules (R1)--(R6) and (C1)--(C6) to two-row- and two-column-groups that contain two diagonal boxes) reverses the weight and preserves the shape of the group on which it acts, so the map $\wtrev$ is weight-reversing and shape-preserving.  It remains to show first that $\wtrev$ is well defined (i.e., that for any sorted ShBT $T$, the image $\wtrev(T)$ is also a sorted ShBT), and second that the groups of $\wtrev(T)$ are supported on the same boxes as the groups of $T$. Since each local transformation is an involution by definition, this will establish that $\wtrev$ is an involution.

Given any group in $T$, its image under the appropriate local transformation is compatible with the definition of sorted tableaux.  Thus, to check that $\wtrev$ is well defined as a map from sorted ShBT to sorted ShBT, we must establish that (1) adjacent boxes in different groups satisfy the correct order relation, and (2) if there are two diagonal boxes, then they and the box between them are validly filled.

We begin with some basic observations.  All of the local transformations have the property that if a non-diagonal box $(i, j)$ in a sorted ShBT $T$ is filled with a primed (respectively, unprimed) entry, then the box $(i, j)$ in $\wtrev(T)$ is also filled with a primed (respectively, unprimed) entry.  Moreover, every non-diagonal box of a one- or two-column-group is filled with a primed entry, and every non-diagonal box of a one- or two-row-group is filled with an unprimed entry.  Therefore, to prove that $\wtrev(T)$ is a sorted ShBT, we need only check conditions at the diagonal and between groups of the same ``type'' (both row or both column).  We handle the second case first.

In the case of two-row-groups that do not contain two diagonal boxes, each box is filled with the same number after toggling the bar divisions as it was before (only the division of these boxes into bars changes), and the new bars are still connected under $\approx$ (a division between bars that fails to disconnect the group still fails to disconnect the group when it switches rows).  Thus, the operation $\wtrev$ preserves the order relation between groups of this type.  

Two one-row-groups cannot have any adjacent boxes: if a box in one were directly below a box in the other, then these two boxes would belong to a two-row-group (and so, by definition, could not belong to one-row-groups), while if a box in one were directly to the right of a box in the other, then these two boxes would belong to the same one-row-group.  

Finally, suppose that $\cA$ is a two-row-group occupying rows $i$ and $i + 1$ and $\cB$ is a one-row-group that contains a box adjacent to a box in $\cA$.  These adjacent boxes must lie in the same row (rather than the same column) because otherwise the box in the one-row-group would belong instead to a two-row-group.  It follows as 
in the proof of Proposition~\ref{2row-prop} that 
if the boxes in the one-row-group $\cB$ are to the right of the boxes in $\cA$, then it must lie in row $i$, because there is no way to complete
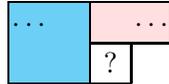
\[
\begin{young}[15.pt][c] 
]=![cyan!50] \ynobottom \cdots & =]![cyan!50] \ynobottom & ]=![pink!50] & =]![pink!50] \cdots  \\
]=![cyan!50] \ynotop  & =]![cyan!50] \ynotop & ]=] ? 
\end{young}
\]
to a sorted ShBT. Similarly, if the boxes in $\cB$ lie to the left of the boxes in $\cA$, then they must lie in row $i + 1$.  (In the latter case, there is an additional consideration concerning the situation in which the one-row-block consists of a single diagonal box, but in this case the putative one-row-block would actually be part of the two-row-block.)  In the first situation, after applying $\wtrev$, the box in the two-row-group is still filled with $1$ and still to the left of the box in the one-row-group (which will either end up filled with $1$ or $2$), and in the second case, the box in the two-row-group is still filled with $2$ and still to the right of the box in the one-row-group. Thus the order relations are preserved in both cases.  The situation for one- and two-column-groups is precisely the conjugate.  This completes the proof that $\wtrev(T)$ obeys the order relations of a sorted ShBT among all pairs of adjacent boxes except possibly those at the diagonal (in the case that $T$ contains two diagonal boxes).

Now consider the collection of sorted ShBT with two diagonal boxes in positions $(i, i)$ and $(i + 1, i + 1)$.  There are twelve possible fillings of the three boxes $(i, i)$, $(i, i + 1)$, $(i + 1, i + 1)$: six of them are shown in the statement of Proposition~\ref{2row-prop}, and the other six are conjugate to these.  In the cases
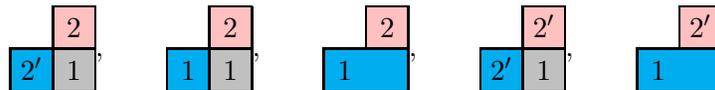
\[
\begin{young}[15.5pt][c]
, & ![pink!99] 2 \\
![cyan!99] 2' & ![lightgray!99] 1
\end{young},
\qquad
\begin{young}[15.5pt][c]
, & ![pink!99] 2 \\
![cyan!99] 1 & ![lightgray!99] 1
\end{young},
\qquad
\begin{young}[15.5pt][c]
, & ![pink!99] 2 \\
]=![cyan!99] 1 & =]![cyan!99]
\end{young},
\qquad
\begin{young}[15.5pt][c]
, & ![pink!99] 2' \\
![cyan!99] 2' & ![lightgray!99] 1
\end{young},
\qquad
\begin{young}[15.5pt][c]
, & ![pink!99] 2' \\
]=![cyan!99] 1 & =]![cyan!99]
\end{young}
\]
and their conjugates, it is easy to see that the three boxes must belong to a single two-row-group.  (The nontrivial cases are the first and fourth; the proof of Proposition~\ref{2row-prop} shows explicitly that in the box $(i, i)$ must be a one-bar group in the fourth case, and the proof in the first case is identical.)  The boxes directly below this group (if there are any) can only be filled with primed entries; after applying $\wtrev$ (checking the six cases (R1)--(R6)), those boxes will still be filled with primed entries and so the necessary order relations are respected.  Similarly, the boxes directly to the right of the group (if there are any) will be filled with unprimed entries, and any box in row $i + 1$ will be part of a two-row-group and filled with $2$; after applying $\wtrev$, we again have that all order relations are respected.  The same holds for the conjugates of these cases after swapping the words ``row'' with ``column'', ``primed'' with ``unprimed'', ``below'' with ``right'', and ``$(i, i)$'' with ``$(i + 1, i + 1)$''.  
This leaves the two mutually conjugate cases
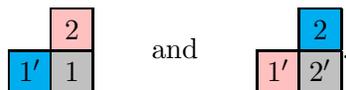
\[
\begin{young}[15.5pt][c]
, & ![pink!99] 2 \\
![cyan!99] 1' & ![lightgray!99] 1
\end{young}
\qquand
\begin{young}[15.5pt][c]
, & ![cyan!99] 2 \\
![pink!99] 1' & ![lightgray!99] 2'
\end{young}.
\]
Here we have two possibilities to consider: box $(i, i + 1)$ is automatically in the same two-row- or two-column-group as one of the other diagonal boxes, but the odd box out could be part of a bar with more than one box. For example, in the tableau
$
\begin{young}[15.5pt][c]
, & ]=![cyan!99] 2 & =]![cyan!99] 2\\
]=]![pink!99] 1' & ]=]![lightgray!99] 2' & ]=]![yellow!99] 1
\end{young}
$
the unprimed bars are part of a two-row-group and the primed bars are part of a two-column-group.  If the two diagonal boxes are in different groups, then the three boxes are unchanged by the operation $\wtrev$, all boxes below them (if any) are primed and stay this way after applying $\wtrev$, and all boxes to their right (if any) are unprimed and stay this way after applying $\wtrev$, so all order relations involving these boxes are preserved.  
 If the two diagonal boxes are part of the same group, and when we apply $\wtrev$ we fall into case (R4) or (C4), 
 and the rest of the argument is essentially the same as for the cases discussed above.
\end{proof}

\subsection{Composing operations}

Putting together everything above, we arrive at the main theorem of this section:

\begin{theorem}
\label{thm:simple transposition}
The composition \[ \tauShBT := \unswap \circ \wtrev \circ \swap\] is a shape-preserving and weight-reversing involution
of the set of semistandard shifted bar tableaux with all entries in $\{1',1,2',2\}$.
Moreover, if $T$ is a semistandard shifted bar tableau then $\tauShBT(T)$ and $T$ have the same number of primed entries on the main diagonal.
\end{theorem}

\begin{proof}
Let $T$ be a semistandard shifted bar tableau.
In view of Theorems~\ref{thm:inverse bijections} and \ref{thm:wt reversal}, it 
remains only to prove our claim that $\tauShBT(T)$ and $T$ have the same number of primed diagonal entries.
 
By considering the various cases in Definition~\ref{def:swap}, we see that the only swaps in the initial application of $\swap$ that can change a diagonal box are moves of the form
\begin{equation}
\label{slide}
\begin{young}[15.5pt][c] 
]=![cyan!99] 1  & =]![cyan!99] & ![pink!99] 2'  \end{young}
 \mapsto  
  \begin{young}[15.5pt][c] ![pink!99] 2' & ]=![cyan!99] 1& =]![cyan!99] \end{young}
\end{equation}
or its conjugate, where the leftmost (respectively, topmost) affected box belongs to the diagonal.  In either case, if such a swap is applied, then it will be the \emph{last} swap in the computation of $\swap(T)$ that changes anything (i.e., any remaining swaps will fall into case (c) of Definition~\ref{def:swap}).  
Similarly, the only swaps in the final application of $\unswap$ that can change a diagonal box are moves that are the reverse of \eqref{slide} or its conjugate, and when such swaps are applied, they occur as the \emph{first} swap in the computation when $\unswap$ is applied to $\wtrev(\swap(T))$.
We now consider exhaustively the possible changes that can occur at the diagonal.

First suppose $T$ contains a single diagonal box $(i,i)$, and assume that this box of $T$ 
 is filled with $1'$ or $2$.
Then $\swap(T)$ also possesses a single diagonal box filled with the same entry.  In the sorted ShBT $\swap(T)$, this box belongs either to a one-row- or one-column-group, or to a two-row- or two-column-group that contains exactly one diagonal box.  On groups of this type, the operation $\wtrev$ preserves the entry in the diagonal box with a unique exception, namely, if the diagonal box belongs to a one-row- or one-column-group whose entries are all equal (as in the last figure in \eqref{toggling-one-column-eq}).  

In the non-exceptional case, $\wtrev(\swap(T))$ has its unique diagonal box $(i, i)$ filled with the same value ($1'$ or $2$) as $T$. Applying $\unswap$ preserves this property
since it only affects boxes containing $1$ and $2'$, so $\tauShBT(T)$ also has its diagonal box filled with the same value as $T$.

Assume we are in the exceptional case, so that 
the unique diagonal box $(i, i)$ of $\swap(T)$ is part of a one-row- or one-column-group whose entries are all equal.
Without loss of generality (up to conjugation) we can assume that it is a one-row-group whose entries are all $2$, and so that the diagonal box in $\wtrev(\swap(T))$ is filled with $1$.  
Possibly the box $(i - 1, i)$ directly below the diagonal is empty (i.e., it does not belong to $T$).  Then the very first step of the application of $\unswap$ cannot be a move of the form 
\begin{equation}
\label{eq:this doesn't happen}
  \begin{young}[15.5pt][c]     
  ![cyan!99]1 \\ 
  ![pink!99]2'  \ynobottom\\ 
  ![pink!99] \ynotop \end{young}
  \mapsto \begin{young}[15.5pt][c]    
  ![pink!99] 2' \ynobottom \\ 
  ![pink!99] \ynotop \\ 
  ![cyan!99] 1    \end{young}
\end{equation}
and hence the resulting tableau $\tauShBT(T)$ still has unprimed diagonal entry.  

If, on the other hand, the box $(i - 1, i)$ is not empty, then in $\swap(T)$ it cannot be filled with $2$ (because $\swap(T)$ is sorted), and it cannot be filled with $1$ (because $(i, i)$ does not belong to a two-row-group), so it must be filled with a primed entry.  Furthermore, we claim that $(i - 1, i)$ cannot be filled with $2'$: if box $(i-1,i)$ had $2'$ and the bar of $(i,i)$ had multiple boxes then $(i-1,i+1)$ would be filled with $1$ to be sorted and $(i,i)$ would be in a two-row-group, whereas if box $(i-1,i)$ had $2'$ and the bar of $(i,i)$ had just one box then $(i-1,i)$ would part of a one-column-group that includes $(i,i)$,
which is a contradiction in either case.  

Thus, when the box $(i - 1, i)$ belongs to $T$, it must be filled with $1'$ in $\swap(T)$, and this $1'$ must be part of a two-column group.  After applying $\wtrev$ to $\swap(T)$, it follows that the diagonal box will be filled with the value $1$ and the box $(i - 1, i)$ will be filled with $1'$. This means that applying $\unswap$ to $\wtrev(\unswap(T))$ cannot begin with a move of the form
\eqref{eq:this doesn't happen}
and hence the resulting tableau $\tauShBT(T)$ still has unprimed diagonal entry.  This completes the claim in the case that $T$ has one diagonal box and it is filled with $1'$ or $2$.
 
Continue to assume that $T$ has a single diagonal box $(i,i)$, but now 
suppose this box is filled with $1$ or $2'$. 
Up to conjugacy we may assume that the box is filled with $1$, so that $T$ looks like
\[
\begin{young}[15.5pt][c]
, & , & ,\cdot \\
, & ![cyan!99]1 & \text{??} \\
,\cdot & ,\cdot   & \text{??}
\end{young}
\qquord
\begin{young}[15.5pt][c]
, & , & ,\cdot \\
, & ![cyan!99]1 & \text{??} \\
,\cdot & 1' & \text{??}
\end{young},
\]
where possibly the bar containing the diagonal box extends to the right.  
Our objective is to show that the diagonal box $(i,i)$ is still unprimed in $\tauShBT(T)$.
To this end, we first observe that the bar tableau $\swap(T)$ must look like
\begin{equation}
\label{eq:four cases}
\begin{young}[15.5pt][c]
, & , & ,\cdot \\
, & ![cyan!99]1 & \text{??} \\
,\cdot & 1' & \text{??}
\end{young}
\qquord
\begin{young}[15.5pt][c]
, & , & ,\cdot \\
, & ![cyan!99]1 & \text{??} \\
,\cdot & ,\cdot   & \text{??}
\end{young}
\qquord
\begin{young}[15.5pt][c]
, & , & ,\cdot \\
, & 2' & ![cyan!99] 1  \\
,\cdot & ,\cdot & \text{??}
\end{young}
\qquord
\begin{young}[15.5pt][c]
, & , & ,\cdot \\
, & 2' & ![cyan!99] 1  \\
,\cdot & ]=] 1' & \text{??}
\end{young},
\end{equation}
where again ``the bar of $1$s'' may extend to the right. Since $(i, i)$ is the unique diagonal box of $T$, it must be the case that $T$ does not contain any boxes in row $i + 1$, and consequently in all four cases ``the bar of $1$s'' in $\swap(T)$ is not related by $\approx$ to any other box.  In the second, third, and fourth cases of \eqref{eq:four cases}, it follows that this bar is in a one-row-group; in the first case, it might be part of either a one-row- or one-column-group, depending on the presence and contents of the cells marked ``$??$'' in $\swap(T)$.  We consider the cases in order.

In the first case, suppose first that the diagonal box $(i, i)$ is part of a one-row-group in $\swap(T)$.  Then the bar containing $(i - 1, i)$ that is filled with $1'$ must belong to a two-column-group in $\swap(T)$ (or else the two bars would be related by $\sim$).  Therefore, after applying $\wtrev$, the diagonal box in $\wtrev(\swap(T))$ will be unprimed and the $(i - 1, i)$ box will still be filled with $1'$.  The operation $\unswap$ can only convert an unprimed diagonal box into a primed diagonal box by an application of the operation
\[  
\begin{young}[15.5pt][c]     
  ![cyan!99]1 \\ 
  ![pink!99]2'  \ynobottom\\ 
  ![pink!99] \ynotop \end{young}
  \mapsto \begin{young}[15.5pt][c]    
  ![pink!99] 2' \ynobottom \\ 
  ![pink!99] \ynotop \\ 
  ![cyan!99] 1    \end{young},
\]
but this cannot happen when the box $(i - 1, i)$ is filled with $1'$.  Thus the diagonal box in $\tauShBT(T)$ is unprimed in this case.

Continuing with the first case, suppose instead that the diagonal box $(i, i)$ is part of a one-column-group in $\swap(T)$.  Since $(i - 1, i)$ is filled with $1'$ in $\swap(T)$, the entire one-column-group is filled with $1'$ and $1$.  Therefore, after toggling the labels of the group, the diagonal box in $\wtrev(\swap(T))$ will be filled with $2$ (and the rest of the one-column-group with $2'$).  No step in the application of $\unswap$ moves a box filled with $2$, hence the diagonal box in $\tauShBT(T)$ is unprimed in this case.

In the second case, the diagonal box in $\swap(T)$ can only belong to a one-row-group, so the diagonal box in $\wtrev(\swap(T))$ is unprimed.  The operation $\unswap$ can only convert an unprimed diagonal box into a primed diagonal box by an application of the operation
\[  
\begin{young}[15.5pt][c]     
  ![cyan!99]1 \\ 
  ![pink!99]2'  \ynobottom\\ 
  ![pink!99] \ynotop \end{young}
  \mapsto \begin{young}[15.5pt][c]    
  ![pink!99] 2' \ynobottom \\ 
  ![pink!99] \ynotop \\ 
  ![cyan!99] 1    \end{young},
\]
but this cannot happen when the box $(i - 1, i)$ is not part of the shape.  Thus the diagonal box in $\tauShBT(T)$ is also unprimed in this case.

In the third case, the diagonal one-box-bar filled with $2'$ belongs to the same one-row-group in $\swap(T)$ as ``the bar of $1$s''.  
These two bars are filled the same in $\wtrev(\swap(T))$ as in $\swap(T)$ since 
when we apply $\wtrev$
we are in the case of toggling the labels with $p, q > 0$ in the sense of \eqref{toggling-eq}. Then the very first step of $\unswap$ applies the rule
\[
 \begin{young}[15.5pt][c] ![pink!99] 2' & ]=![cyan!99] 1& =]![cyan!99] \end{young}
  \mapsto
    \begin{young}[15.5pt][c]  ]=![cyan!99] 1  & =]![cyan!99] & ![pink!99] 2'  \end{young},
\]
so the diagonal box in $\tauShBT(T)$ is unprimed in this case.

Finally, in the fourth case, we have by sortedness that ``the bar of $1$s'' must be a singleton bar, and the box filled with ``$??$'' must actually be filled with $2'$, so $\swap(T)$ is of the form
\[
\begin{young}[15.5pt][c]
, & , & ,\cdot \\
, & 2' & ![cyan!99] 1 & ,\cdot \\
,\cdot & ]=] 1' & 2' & \text{??}
\end{young}
.
\]
In this case, the bars containing boxes $(i - 1, i)$ and $(i - 1, i + 1)$ in $\swap(T)$ belong to a two-column-group, and so the one-box bars $(i, i)$ and $(i, i + 1)$ form a one-row-group in $\swap(T)$.  It follows that $\wtrev(\swap(T))$ has the same entries in these four boxes.  Then, as in the third case, the very first step of $\unswap$ applies the rule
\[
 \begin{young}[15.5pt][c] ![pink!99] 2' & ]=![cyan!99] 1& =]![cyan!99] \end{young}
  \mapsto
    \begin{young}[15.5pt][c]  ]=![cyan!99] 1  & =]![cyan!99] & ![pink!99] 2'  \end{young},
\]
and so the diagonal box in $\tauShBT(T)$ is unprimed in this case.

Finally, we consider the case when $T$ has two diagonal boxes.
If both diagonal boxes are unprimed, then there are three ways that they and the box between them can be filled in $T$:
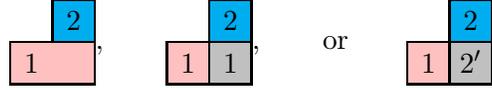
\[
\begin{young}[15.5pt][c]
, & ![cyan!99] 2 \\
]=![pink!99] 1 & =]![pink!99]
\end{young},
\qquad
\begin{young}[15.5pt][c]
, & ![cyan!99] 2 \\
![pink!99] 1 & ![lightgray!99] 1
\end{young},
\qquord
 \begin{young}[15.5pt][c]
, & ![cyan!99] 2 \\
![pink!99] 1 & ![lightgray!99] 2'
\end{young}.
\]
After applying $\swap$ to such a tableau, we might end up with one of the near-diagonal arrangements 
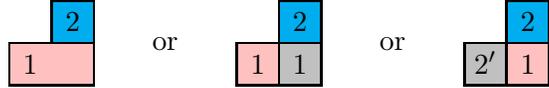
\[
\begin{young}[15.5pt][c]
, & ![cyan!99] 2 \\
]=![pink!99] 1 &  =]![pink!99] 
\end{young}
\qquord
\begin{young}[15.5pt][c]
, & ![cyan!99] 2 \\
![pink!99] 1 & ![lightgray!99] 1 
\end{young}
\qquord
\begin{young}[15.5pt][c]
, & ![cyan!99] 2 \\
![lightgray!99] 2' & ![pink!99] 1 
\end{young},
\]
where in the last case necessarily the diagonal box $\ytab{*(lightgray) 2'}$ forms a one-box bar (following a move of the form \eqref{slide}).  In the first case, applying $\wtrev$ toggles the bar divisions and so does not change the contents of the diagonal boxes or the box $(i, i + 1)$, and applying $\unswap$ to the result preserves the diagonal boxes.  The second and third cases are precisely those to which the rule (R1) in the definition of $\wtrev$ applies; they are interchanged, and then $\unswap$ restores the unprimed diagonal boxes.

The argument when both diagonal boxes of $T$ are primed is precisely the conjugate of the previous case.  This leaves the situation when $T$ has exactly two diagonal boxes, only one of which is primed.  Then there are six possible arrangements for the two diagonal boxes and the box between them, in three conjugate pairs:
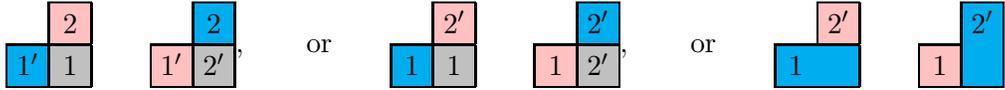
\[
\begin{young}[15.5pt][c]
, & ![pink!99] 2 \\
![cyan!99] 1' & ![lightgray!99] 1
\end{young}
\qquad
\begin{young}[15.5pt][c]
, & ![cyan!99] 2 \\
![pink!99] 1' & ![lightgray!99] 2'
\end{young},
\qquord
\begin{young}[15.5pt][c]
, & ![pink!99] 2' \\
![cyan!99] 1 & ![lightgray!99] 1
\end{young}
\qquad
\begin{young}[15.5pt][c]
, & ![cyan!99] 2' \\
![pink!99] 1 & ![lightgray!99] 2'
\end{young},
\qquord
\begin{young}[15.5pt][c]
, & ![pink!99] 2' \\
]=![cyan!99] 1 & =]![cyan!99]
\end{young}
\qquad
\begin{young}[15.5pt][c]
, & ![cyan!99]\ynobottom 2' \\
]=]![pink!99] 1 & ![cyan!99]\ynotop
\end{young}.
\]
Suppose $T$ belongs to one of the first two cases. Then $\swap(T)$ still has the near-diagonal arrangement $
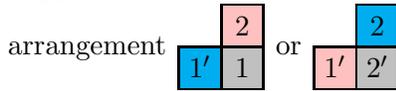
\begin{young}[15.5pt][c]
, & ![pink!99] 2 \\
![cyan!99] 1' & ![lightgray!99] 1
\end{young}
$ or $
\begin{young}[15.5pt][c]
, & ![cyan!99] 2 \\
![pink!99] 1' & ![lightgray!99] 2'
\end{young}$ (possibly reversed from what it had before); up to conjugacy, we assume the former.  If the diagonal boxes $\swap(T)$ belong to different groups, then after applying $\wtrev$ the near-diagonal arrangement is unchanged, and then $\unswap$ does not affect the diagonal boxes.  Suppose instead that the two diagonal boxes belong to the same group in $\swap(T)$.   Depending on whether the bars containing boxes $(i, i + 1)$ and $(i + 1, i + 1)$ extend to the right, applying $\wtrev$ to $\swap(T)$ belongs to either case (R2) or (R4), and the image necessarily has the near-diagonal arrangement
\[
\begin{young}[15.5pt][c]
, & ![pink!99] 2' \\
]=![cyan!99] 1 & =]![cyan!99]
\end{young} \text{$\quad$in case (R2)}
\qquord
\begin{young}[15.5pt][c]
, & ![pink!99] 2' \\
![lightgray!99] 2' & ![cyan!99] 1
\end{young}\text{$\quad$in case (R4)}.
\]
Applying $\unswap$ to such a tableau results in a tableau with near-diagonal arrangement 
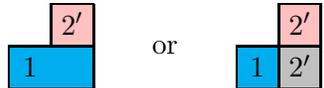
\[
\begin{young}[15.5pt][c]
, & ![pink!99] 2' \\
]=![cyan!99] 1 & =]![cyan!99]
\end{young}
\qquord
\begin{young}[15.5pt][c]
, & ![pink!99] 2' \\
![cyan!99] 1 & ![lightgray!99] 2'
\end{young},
\]
and so we conclude that $\tauShBT(T)$ has exactly one primed diagonal box, as needed.

The remaining four cases are, up to conjugacy, precisely the reverse direction of the cases just considered; since $\tauShBT$ is an involution, it follows that its action in these cases also preserves the number of primed diagonal boxes.  This completes the proof.
%
%
\end{proof}

\section{Generating function derivations}\label{gen-fun-der-sect}

Let $\mu$ and $\lambda$ be strict partitions
and recall the definitions of 
the sets of shifted plane partitions $ \MRPP_P(\lambda/\mu) \subseteq  \MRPP_Q(\lambda/\mu)$
and semistandard shifted bar tableaux $\ShBTP(\lambda/\mu) \subseteq \ShBTQ(\lambda/\mu)$
from Section~\ref{skew-sect}.
Throughout this section we work with the concrete generating functions 
\[
\ba
\wtgp_{\lambda/\mu} &:= \sum_{T\in \MRPP_P(\lambda/\mu)} (-\beta)^{|\lambda/\mu| - |\weight_\RPP(T)|} {\bf x}^{\weight_\RPP(T)} 
,\\
\wtgq_{\lambda/\mu} &:= \sum_{T\in \MRPP_Q(\lambda/\mu)} (-\beta)^{|\lambda/\mu| -  |\weight_\RPP(T)|} {\bf x}^{\weight_\RPP(T)}, \ea
\qquad
\ba
 \wtjp_{\lambda/\mu} &:= \sum_{T\in \ShBTP(\lambda/\mu)} (-\beta)^{|\lambda/\mu|-|T|} {\bf x}^{T} 
,\\
\wtjq_{\lambda/\mu} &:= \sum_{T\in \ShBTQ(\lambda/\mu)} (-\beta)^{|\lambda/\mu| -|T|} {\bf x}^{T} .
\ea
\]
Our   goal is to prove that these power series coincide 
with the symmetric functions 
$\gp_{\lambda/\mu}$, $\gq_{\lambda/\mu}$, $\jp_{\lambda/\mu}$, and $\jq_{\lambda/\mu}$ defined in Section~\ref{skew-sect}. On setting $\mu=\emptyset$, this will establish Theorems~\ref{main-thm1} and \ref{main-thm2}. 

\subsection{Base cases and symmetry}

Our proof strategy uses an inductive algebraic argument whose base case is the following:

\begin{proposition}[{\cite[Prop.~7.5]{ChiuMarberg}; \cite[Prop.~5.2]{NakagawaNaruse}}] \label{one-part-prop}
If $\lambda$ has at most one part then 
\[
\gp_{\lambda/\mu}=\wtgp_{\lambda/\mu},\quad
\gq_{\lambda/\mu}=\wtgq_{\lambda/\mu},\quad
\jp_{\lambda/\mu}= \wtjp_{\lambda/\mu},\quand 
\jq_{\lambda/\mu}=\wtjq_{\lambda/\mu}.\]
\end{proposition}

 Our inductive step relies on the following consequence of Theorem~\ref{thm:simple transposition}:
 
\begin{theorem}\label{j-sym-thm}
The power series $\wtjp_{\lambda/\mu}$ and $\wtjq_{\lambda/\mu}$ are symmetric in the $x_i$ variables.
\end{theorem}

\begin{proof}
The \defn{weight} of a general ShBT is the tuple $a=(a_1,a_2,\dots)$ where $a_i$ is the number of bars that contain either $i$ or $i'$.
Let $|a| = a_1+a_2+\dots$.
The coefficient of $(-\beta)^{|\lambda/\mu|-|a|} x_1^{a_1} \cdots x_i^{a_i} x_{i + 1}^{a_{i + 1}} \cdots$ in $\wtjq_{\lambda/\mu}$ is the number of semistandard ShBT's of shape $\lambda/\mu$ with weight $a=(a_1,a_2,\dots)$.  By the first part of Theorem~\ref{thm:simple transposition}, applying the map $\tauShBT$ to the subtableaux generated by the entries $i', i, i' + 1, i + 1$ gives a bijection with the ShBT's counted by the coefficient of $(-\beta)^{|\lambda/\mu|-|a|} x_1^{a_1} \cdots x_i^{a_{i + 1}} x_{i + 1}^{a_{i}} \cdots$.  Since any finitely supported permutation is a product of adjacent transpositions, it follows that $\wtjq_{\lambda/\mu}$ is symmetric 
in the $x_i$ variables.

The second part of Theorem~\ref{thm:simple transposition} ensures that the map $\tauShBT$ restricts to an involution on the subset $\ShBTP(\lambda/\mu)\subseteq\ShBTQ(\lambda/\mu)$ consisting of those ShBT's with no primed diagonal entries.  
Thus $\wtjp_{\lambda/\mu}$ is symmetric by the same argument. 
\end{proof}

In the next three subsections we derive a Pieri rule for the $\wtjp$- and $\wtjq$-functions,
in order to prove that $\wtjp_{\lambda/\mu}=\jp_{\lambda/\mu}$ and $\wtjq_{\lambda/\mu}=\jq_{\lambda/\mu}$.
We then use these results to derive the analogous identities for $\gp_{\lambda/\mu}$ and $\gq_{\lambda/\mu}$ by a formal argument
in Section~\ref{rpp-gf-sect}.
 
\subsection{Product formulas and one-row identities}

There is an easy way to express 
$\wtjq_{\lambda/\mu}\wtjq_{\nu/\kappa}$ and $\wtjp_{\lambda/\mu}\wtjq_{\nu/\kappa} $
as sums of other functions indexed by skew shapes.
 Suppose $\rho$ and $\tau$ are finite, nonempty subsets of $\{1,2,3,\dots\}\times \{1,2,3,\dots\}$.
 Let 
 \[ i = \min \{ a : (a,b) \in \rho \text{ for some } b\}\quand j = \max\{ b : (i,b) \in \rho\}.\]
 In French notation, $(i, j)$ is the box of $\rho$ in its bottom row that is farthest to the right.
 Then let
 \[ k = \max \{ a : (a,b) \in \tau \text{ for some } b\}\quand l = \min\{ b : (k,b) \in \tau\}.\]
  In French notation, $(k, l)$ is the box of $\tau$ in its top row that is farthest to the left.
  Form a new diagram $ \rho \vartriangleright \tau$ by translating the sets $\rho$ and $\tau$  so that $(i,j)$ is directly to the left of $(k,l)$  in the same row;
form $ \rho \vartriangleleft \tau$ by translating the  two sets  so that $(i,j)$ is directly above $(k,l)$  in the same column;
and form $ \rho \circ \tau$ by translating the two sets  so that the boxes $(i,j)$ are $(k,l)$ coincide.
If \[\Delta_1 := k-i\quand \Delta_2 := \Delta_1+ j+1 - l\] 
then these shapes are defined precisely by setting 
 \[ 
 \ba
 \rho \vartriangleright \tau &:= \Bigl[ \rho +  ( \Delta_1 ,\Delta_1) \Bigr]\sqcup \Bigl[ \tau + (0,\Delta_2)\Bigr],
 \\
 \rho \vartriangleleft \tau &:=\Bigl[ \rho +  ( \Delta_1+1 ,\Delta_1+1) \Bigr]\sqcup \Bigl[ \tau + (0,\Delta_2)\Bigr],
  \\
 \rho \circ \tau &:= \Bigl[ \rho +  ( \Delta_1 ,\Delta_1) \Bigr] \cup \Bigl[ \tau + (0,\Delta_2-1)\Bigr].
 \ea
   \]
The third union is not disjoint. All three shapes have the same number of diagonal boxes as $\rho$. 
   
   \begin{example}\label{our-two-shapes-ex}
   If our two shapes are
   \[
   \rho = \SD_{(3,2)}=
   \ytabsmall{
   \none & \ & \ \\ 
 \ & \ & \
   }
\qquand
      \tau = \D_{(2,2)/(1)}=
   \ytabsmall{
     \ & \ \\ 
\none[\cdot] & \
   }
   \] then we have $(i,j) = (1,3)$ and $(k,l) = (2,1)$, so $\Delta_1 = 1$ and $\Delta_2 = 4$.  Thus 
%
  \[
 \rho \vartriangleright \tau=
    \ytabsmall{
\none &   \none & *(pink) & *(pink) \\ 
\none &  *(pink) & *(pink) & *(pink) & *(cyan) & *(cyan) \\
 \none[\cdot] & \none[\cdot] & \none[\cdot] & \none[\cdot] & \none[\cdot] & *(cyan)
   },
   \quad
   \rho \vartriangleleft \tau=
    \ytabsmall{
  \none& \none&  \none & *(pink) & *(pink) \\ 
 \none& \none& *(pink) & *(pink) & *(pink)  \\
 \none& \none[\cdot] & \none[\cdot] & \none[\cdot] & *(cyan) & *(cyan) \\
  \none[\cdot] & \none[\cdot]  &  \none[\cdot] & \none[\cdot] & \none[\cdot] & *(cyan)
   },
   \quand
    \rho \circ \tau=
    \ytabsmall{
\none &   \none & *(pink) & *(pink) \\ 
\none &  *(pink)  & *(pink) & *(magenta) & *(cyan) \\
 \none[\cdot] & \none[\cdot] &  \none[\cdot] & \none[\cdot] & *(cyan)
   }.
   \]
   \end{example}

Suppose $\rho = \SD_{\lambda/\mu}$ 
for strict partitions $\mu \subsetneq \lambda$ and $\tau = \D_{\nu/\kappa}$ for
  arbitrary partitions
 $\kappa \subsetneq \nu$.
We may view $\tau = \D_{\nu/\kappa}$ as a translation of a shifted skew shape:
if
$m:=\ell(\nu)-1$
and $\delta:=(m,\dots,1,0)$, then
$\tau$ is the tableau formed by translating   $\SD_{(\nu+\delta)/(\kappa+\delta)}$ to the left by $m$ columns.  
In this case, we define $\ShBTQ(\tau) := \ShBTQ((\nu+\delta)/(\kappa+\delta))$ and 
$\wtjq_{\tau} := \wtjq_{(\nu+\delta)/(\kappa+\delta)}$.

\begin{proposition}\label{easy-prop}
Suppose $\rho$ is a nonempty shifted skew shape and $\tau$ is a nonempty unshifted skew shape, as above. Then
 $ \rho \vartriangleright \tau$, $ \rho \vartriangleleft \tau$, and $ \rho \circ \tau$
are shifted skew shapes, and
it holds that 
\[\ba 
\wtjq_{\rho}\wtjq_\tau &= \wtjq_{\rho \vartriangleright \tau} + \wtjq_{\rho \vartriangleleft \tau} + \beta \wtjq_{\rho \circ \tau},
\\
\wtjp_{\rho}\wtjq_\tau &= \wtjp_{\rho \vartriangleright \tau} + \wtjp_{\rho \vartriangleleft \tau} + \beta \wtjp_{\rho \circ \tau}.
\ea\]
\end{proposition}

\begin{proof}
The following argument is similar to the proof of \cite[Lem.~9.12]{LamPyl}.
Checking that $ \rho \vartriangleright \tau$, $ \rho \vartriangleleft \tau$, and $ \rho \circ \tau$
are shifted skew shapes is straightforward.
Suppose $T \in \ShBTQ(\rho)$ and $U \in \ShBTQ(\tau)$.
Let $T \vartriangleright U $ and $T \vartriangleleft U$ be the shifted bar tableaux of shapes $\rho \vartriangleright \tau$
and $\rho \vartriangleleft \tau$ given by translating the boxes of $T$ and $U$ in the obvious way.

If $T_{ij} < U_{kl}$ or $T_{ij} = U_{kl}\in \ZZ$
then we have $T \vartriangleright U  \in \ShBTQ(\rho \vartriangleright \tau)$,
$|T \vartriangleright U| = |T|+|U|$, and $x^{T \vartriangleright U}  = x^T x^U$.
On the other hand,
if $T_{ij} > U_{kl}$ or $T_{ij} = U_{kl}\in \ZZ'$
then it holds that $T \vartriangleleft U  \in \ShBTQ(\rho \vartriangleleft \tau)$,
 $|T \vartriangleleft U| = |T|+|U|$, and $x^{T \vartriangleleft U}  = x^T x^U$.

Merging the 
boxes of $\rho \vartriangleright \tau $ or $\rho \vartriangleleft \tau$ corresponding to $(i,j) \in \rho$ and $(k,l) \in \tau$
gives a bijection from the set of elements in $ \ShBTQ(\rho \vartriangleright \tau)\sqcup \ShBTQ(\rho \vartriangleleft \tau)$
that do not arise from one of these cases (namely, those in which the boxes that are adjacent but of different colors in Example~\ref{our-two-shapes-ex} belong to the same bar) to  $ \ShBTQ(\rho \circ \tau)$.
The expansion for $\wtjq_{\rho}\wtjq_\tau $ then follows.
Under these bijections, the diagonal boxes in $T \vartriangleleft U$ and $T \vartriangleright U$ are filled with the same entries as the diagonal boxes in $T$, so the formula for $\wtjp_{\rho}\wtjq_\tau $ also follows.
\end{proof}

We write 
$\wtjp_n := \wtjp_{(n)}$ and
$\wtjq_n := \wtjq_{(n)}$ when $n>0$,
and set $
\wtjp_0 = \wtjq_0 := 1$.

\begin{corollary}\label{one-row-rules-cor}
Suppose $\lambda$ is a nonempty strict partition and $n$ is a positive integer. Then 
\[ 
\ba
\wtjq_\lambda \wtjq_n  &= \wtjq_{(n+\lambda_1,\lambda_2,\lambda_3,\dots)} +\beta \wtjq_{(n+\lambda_1-1,\lambda_2,\lambda_3,\dots)}
+ \wtjq_{(n+\lambda_1,\lambda_1,\lambda_2,\lambda_3,\dots)/(\lambda_1)},
\\
\wtjp_\lambda \wtjq_n  &= \wtjp_{(n+\lambda_1,\lambda_2,\lambda_3,\dots)} +\beta \wtjp_{(n+\lambda_1-1,\lambda_2,\lambda_3,\dots)}
+ \wtjp_{(n+\lambda_1,\lambda_1,\lambda_2,\lambda_3,\dots)/(\lambda_1)}.
\ea
\]
\end{corollary}

\begin{proof}
Take $\rho = \SD_\lambda$ and $\tau = \D_{(n)}$ in Proposition~\ref{easy-prop}.
\end{proof}

Given $f \in \ZZ[\beta]\llbracket x_1,x_2,\dots\rrbracket $ form $f(t) \in \ZZ[\beta]\llbracket t\rrbracket $ by substituting $x_1\mapsto t$ and $x_i \mapsto 0$ for all $i>0$.
 
\begin{proposition}\label{one-var-prop}
One has $\wtjp_1(t) = t$ and $\wtjq_1(t) = 2t$, and 
if $n\geq 2$ is an integer then
\[\wtjp_n(t) =t(t-\beta)^{n-1} \quand \wtjq_n(t) = (2t^2-\beta t)(t-\beta)^{n-2}.\] 
\end{proposition}

\begin{proof}
These identities are straightforward to check   from the definitions.
\end{proof}

\subsection{Bar tableau generating functions with diagonal primes}
 
 \def\nnu{\Psi}
 \def\kkappa{\Lambda}
 \def\llambda{\Gamma}
 
 \def\aparam{\mathsf{sc}}
 \def\bparam{\mathsf{mc}}
 \def\cparam{\mathsf{fb}}
 \def\dparam{\mathsf{res}}

Let $\kkappa = (\kkappa_1,\kkappa_2,\kkappa_3,\dots)$ be a strict partition and define $\llambda = (\kkappa_2,\kkappa_3,\dots)$.
 Then $\SD_{\kkappa/\nu}$ is a (not necessarily connected) shifted ribbon if and only if $\nu$ contains $\llambda$ and is contained in $\kkappa$.
 We fix one such strict partition $\nnu = (\nnu_1,\nnu_2,\dots)$ with  $\llambda \subseteq \nnu \subseteq \kkappa$.
 Throughout this subsection, the capitalized Greek letters $\kkappa$, $\llambda$, and $\nnu$ will denote these fixed partitions,
 and we will use the typical lower case symbols $\lambda$, $\mu$, $\nu$, etc., for partitions that may be arbitrary.

 Define a \defn{forced box} in  $ \SD_{\kkappa/\nnu}$ to be a position $(i,j) $
such that  $(i+1,j),(i+1,j-1) \in \SD_{\kkappa/\nnu}$ or $(i,j-1),(i+1,j-1) \in \SD_{\kkappa/\nnu}$.
These are the boxes containing $\bullet$ in 
$\ytabsmall{ \ & \ \\ \none & \bullet } $ and $\ytabsmall{\ & \none \\ \ & \bullet }$.
For this subsection, we also fix the meaning of certain integer parameters $\aparam$, $\bparam$, $\cparam$, $\dparam$.
Let $\aparam$ be the number of \underline{s}ingleton \underline{c}onnected components in $\SD_{\kkappa/\nnu}$
(i.e., components with only one box), let $\bparam$ be the number of \underline{m}ultiple box \underline{c}onnected components,
 let $\cparam$ be the number of \underline{f}orced \underline{b}oxes in $\SD_{\kkappa/\nnu}$,
and let $\dparam = |\kkappa| - |\nnu| - \aparam - 2\bparam-\cparam+2$ be a \underline{res}idual value. We always have  $\dparam\geq 2$.

\begin{lemma}\label{jq-lem1}
We have $ \wtjq_{\kkappa/\nnu}(t) =   (2t^2  -\beta t)  (t-\beta)^{\dparam-2}   2^{\aparam}   t^{\aparam+\bparam+\cparam-1}  (2t  -\beta)^{\bparam-1} $.
\end{lemma}

\begin{proof}
Consider the tableaux in  $\ShBTQ(\kkappa/\nnu)$ with all entries in  $\{1' ,1\}$, with boxes ordered in the usual row-reading order.
The first entry in each connected component may be arbitrarily primed or unprimed, while the entries in all other boxes (ignoring bars) are uniquely determined by the shape.  
In each connected component, the first box must be the beginning of a new bar, as must every forced box (if any are present).  In a connected component with multiple boxes,
 if the first two boxes are in the same row (respectively, column) then the second box is only required to start a new bar if the first box is primed (respectively, unprimed). Any boxes not included in these cases
 may either start a new bar or continue the bar of its predecessor.
Comparing these observations with the definition  $ \wtjq_{\kkappa/\nnu}$, we conclude 
 that 
 $ \wtjq_{\kkappa/\nnu}(t) = (2t)^{\aparam}  (t^2 + t(t-\beta))^{\bparam}       t^{\cparam}    (t-\beta)^{|\kkappa|-|\nnu| - 2 \aparam-\bparam-\cparam}$, which gives the result after rearranging terms.
\end{proof}

By Proposition~\ref{one-var-prop} the set $\{\wtjq_n(t): n\geq 0\}$ is a homogeneous $\QQ[\beta]$-basis
for  $\QQ[\beta,t]$,
and so
\be\label{yyy-eq}
 \wtjq_{\kkappa/\nnu}(t)  = \sum_{n\geq 0} \yyy{\kkappa}{\nnu}{n}  \cdot \beta^{|\kkappa|-|\nnu| - n} \cdot \wtjq_n(t)
 \ee for a unique choice of numbers $\yyy{\kkappa}{\nnu}{n} \in \QQ$. 
Recall that $\nnu$ is a strict partition with $\llambda \subseteq \nnu \subseteq \kkappa$.

\begin{corollary} \label{jq-cor1}
 The coefficients $\yyy{\kkappa}{\nnu}{n}$ are all nonnegative integers with the following properties:
 \ben
 \item[(a)] If $ 0=n< |\kkappa|-|\nnu|$  or  $n >|\kkappa|-|\nnu|$   then $\yyy{\kkappa}{\nnu}{n} =  0$.

 \item[(b)] If $0 < n = |\kkappa| - |\nnu|$ then $\yyy{\kkappa}{\nnu}{n} =  2^{ \aparam+\bparam-1}$.

 \item[(c)] If $0<n = |\kkappa| - |\nnu| - 1$ then $\yyy{\kkappa}{\nnu}{n} = 2^{\aparam+\bparam-2}(2\aparam+3\bparam+2\cparam-3)$.
 \een
 \end{corollary}
 
 \begin{proof}
 If $\bparam=0$ then $\cparam=0$, $\aparam = |\kkappa|-|\nnu|$, $\dparam = 2$,
and $\wtjq_{\kkappa/\nnu}(t)  = (2t)^{\aparam}$. In this case the desired properties
are easy to deduce using Proposition~\ref{one-var-prop}.


Assume $\bparam\geq 1$ and 
let $u=t-\beta$. By Proposition~\ref{one-var-prop} and Lemma~\ref{jq-lem1},
 $\yyy{\kkappa}{\nnu}{n}$
is the coefficient of $\beta^{|\kkappa|-|\nnu| -n} u^{n-\dparam}$
 in $2^{\aparam}   (u+\beta)^{\aparam+\bparam+\cparam-1}   (2u +\beta)^{\bparam-1} =  2^{\aparam}   t^{\aparam+\bparam+\cparam-1}  (2t  -\beta)^{\bparam-1} $.
Equivalently, this is the coefficient of $u^n$ in 
$2^{\aparam}   (u+1)^{\aparam+\bparam+\cparam-1}   (2u +1)^{\bparam-1}u^{\dparam}$,
which is clearly a nonnegative integer that is zero if $n<2 \leq \dparam$ or
$n > \aparam+2\bparam+\cparam-2 + \dparam = |\kkappa|-|\nnu|$.
It is also straightforward to extract the coefficients in parts (b) and (c) in this case.
\end{proof}
 
Recall that  $\llambda := (\kkappa_2,\kkappa_3,\kkappa_4,\dots)$.

\begin{lemma}\label{jq-lem2}
If $0 \leq n \leq \kkappa_1$ is an integer then
\[
\wtjq_{\kkappa/(n)} = \sum_{\substack{\mu \text{ strict} \\  \llambda \subseteq \mu \subseteq \kkappa}}  \yyy{\kkappa}{\mu}{n} \cdot \beta^{|\kkappa|-|\mu|-n}\cdot  \wtjq_\mu .
\]
\end{lemma}

\begin{proof}
Since we  know from Theorem~\ref{j-sym-thm} that the generating function $\wtjq_{\kkappa}$ is symmetric,
we can express $\wtjq_{\kkappa}(t,x_1,x_2,\dots,x_m) =\wtjq_{\kkappa}(x_1,x_2,\dots,x_m,t)$
as an element of $\ZZ[t][x_1,x_2,\dots,x_m]$ in two different ways,
by letting $t$ either record the weight contribution of  the bars of $T \in \ShBTQ(\kkappa)$ containing $1'$ or $1$,
or the weight contribution of the bars containing $(m+1)'$ or $m+1$. We use this ability in the following argument.

Because $\wtjq_{\mu}(t)$ is nonzero if and only if $\ell(\mu)\leq 1$,
 we have
\be 
\label{eq:first wtjq_s(t) coefficient}
\wtjq_{\kkappa}(t,x_1,x_2,\dots) = \sum_{\mu\subseteq \kkappa}   \wtjq_{\mu}(t)\wtjq_{\kkappa/\mu}(x_1,x_2,\dots)
= \sum_{n=0}^{\kkappa_1} \wtjq_{\kkappa/(n)}    \wtjq_{n}(t).
\ee
On the other hand, for any fixed $m>0$ we have
\[
 \wtjq_{\kkappa}(t,x_1,x_2,\dots,x_m)
 =
  \wtjq_{\kkappa}(x_1,x_2,\dots,x_m,t )
  =
   \sum_{\mu \subseteq \kkappa} \wtjq_{\mu}(x_1,x_2,\dots,x_m)   \wtjq_{\kkappa/\mu}(t).
   \]
As the last sum is finite and since $\wtjq_{\kkappa/\mu}(t)$ is nonzero if and only if 
$\llambda \subseteq \mu \subseteq \kkappa$, we can take the limit as $m\to \infty$
in the sense of formal power series and apply \eqref{yyy-eq} to obtain
\[
\ba
\wtjq_{\kkappa}(t, x_1,x_2,\dots) &
=  \sum_{\llambda \subseteq \mu \subseteq \kkappa} \wtjq_{\mu}   \wtjq_{\kkappa/\mu}(t)
= 
 \sum_{n\geq 0} \sum_{\llambda \subseteq \mu \subseteq \kkappa}\yyy{\kkappa}{\mu}{n}  \beta^{|\kkappa|-|\mu| -n}  \wtjq_{\mu}   \wtjq_{n}(t).
 \ea
\]
 The lemma follows by equating coefficients of   $\wtjq_n(t)$ in this equation and \eqref{eq:first wtjq_s(t) coefficient}.
\end{proof}

Since $\kkappa \supseteq \nnu \supseteq \llambda = (\kkappa_2,\kkappa_3,\dots)$, the skew shape $\SD_{\nnu/\llambda}$ is also a shifted ribbon.  A removable corner box $(i,j) \in \SD_\llambda$ can be added to $\SD_{\nnu/\llambda}$ to form a shifted ribbon
if and only if $(i+1,j+1) \notin \SD_{\nnu/\llambda}$. 
Let $\cR$ be the set of such boxes $(i,j) \in \SD_\llambda$.
If $\kkappa = (15,12,9,6,3,1)$ and $\nnu = (13,9,7,6,2)$
then the boxes $(i,j) \in \cR$ are the ones marked as $\cU$, $\cV$, or $\cW$ below, with $\SD_{\nnu/\llambda}$
shown in gray:
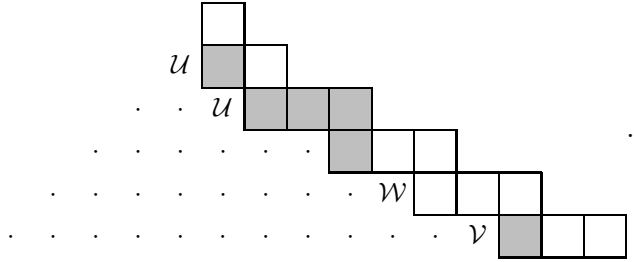
\[
\ytab{
\none & \none & \none & \none & \none  & \ \\ 
\none & \none & \none & \none & \none[\cU]  &  *(lightgray)  & \ \\ 
\none & \none & \none & \none[\cdot] & \none[\cdot]  & \none[\cU]  &  *(lightgray)  &  *(lightgray)  &  *(lightgray)  \\ 
\none & \none & \none[\cdot] & \none[\cdot] & \none[\cdot]  & \none[\cdot] & \none[\cdot] & \none[\cdot] &  *(lightgray)  & \  & \ \\ 
\none & \none[\cdot] & \none[\cdot] & \none[\cdot] & \none[\cdot]  & \none[\cdot] & \none[\cdot] & \none[\cdot] & \none[\cdot] & \none[\cW] & \ & \ & \ \\ 
\none[\cdot] & \none[\cdot] & \none[\cdot] & \none[\cdot] & \none[\cdot]  & \none[\cdot] & \none[\cdot] & \none[\cdot] & \none[\cdot] & \none[\cdot] & \none[\cdot] & \none[\cV] &  *(lightgray)  & \  & \
}.
\]
We divide $\cR$ into three disjoint subsets.
Let $\cU$ be the set of boxes $(i,j)\in \cR$ 
with
   $i=j$ and $(i,j+1) \in\SD_{\nnu/\llambda}$
 or with $i\neq j$ and 
$(i,j+1),(i+1,j) \in \SD_{\nnu/\llambda}$.
Let $\cV$ be the set of boxes $(i,j) \in \cR$ either with $i=j$ and $(i,j+1) \notin \SD_{\nnu/\llambda}$
 or such that $i\neq j$ and exactly one of $(i,j+1)$ or $(i+1,j)$ belongs to $\SD_{\nnu/\llambda}$.
Let $\cW$ be the set of all other boxes $(i,j) \in \cR$.
In our example above, the cells of $\cR$ are labeled according to the decomposition $\cR =\cU\sqcup \cV \sqcup\cW$.

\begin{lemma}\label{uvw-lem}
Suppose $\kkappa_1 - \nnu_1 \geq 2$.
Then $ |\cU| = \aparam $ and $ |\cV| + 2|\cW| = \bparam + \cparam -1$.
\end{lemma}

\begin{proof}
Since $\kkappa_1 \geq \nnu_1 +2$ 
it is clear that $ |\cU| = \aparam$ and $\bparam\geq 1$.
Observe that a position $(i,j)$ belongs to $\cV$ if and only if either
$(i+1,j+1)$ and $(i,j+1)$ are the first two boxes of a connected component of $\SD_{\kkappa/\nnu}$
in the usual row reading order,
or 
$(i+1,j)$ and $(i+1,j+1)$ are the last two boxes of a connected component of $\SD_{\kkappa/\nnu}$
that does not end in the first row.
On the other hand, a position $(i,j)$ belongs to $\cW$ if and only if
all three of the boxes $(i+1,j)$, $(i+1,j+1)$, $(i,j+1)$ belong to a connected component of $\SD_{\kkappa/\nnu}$,
in which case $(i,j+1)$ is a forced box.

Now choose a connected component $\cC\subseteq \SD_{\kkappa/\nnu}$ with multiple boxes.
In this component, compare the number $p $ of boxes $(i,j)$ with $(i-1,j-1) \in \cV\sqcup \cW $
with the number $q$ of forced boxes $(i,j)$ that have $(i,j-1) \notin \cW$,
that is, appearing as $\bullet $ in a configuration  like $\ytabsmall{\ & \none \\ \ & \bullet }$ but not  $\ytabsmall{ \ & \ \\ \none & \bullet }$.
If $\cF$ is the set of forced boxes 
and we set $\hat \cV := \{(i+1,j+1) : (i,j) \in \cV\}$ and $\hat\cW := \{(i+1,j+1):(i,j) \in \cW\}$,
then $p =  |\hat\cV \cap \cC| + |\hat\cW \cap \cC|$
and $q=  |\cF\cap \cC| - |\hat\cW \cap \cC| $. 

We can reinterpret $p$ and $q$ in terms of the shape of a modified version of $\cC$.
If $\cC$ starts with two boxes $(i,j),(i-1,j)$ in the same column (so that $(i,j) \in \hat \cV$ and $(i-1,j-1) \in \cV$), then add $(i,j-1)$ to the component. 
Likewise, if $\cC$ ends with two boxes $(i,j-1),(i,j)$ in the same row (so that if $i>1$ then $(i,j) \in \hat \cV$ and $(i-1,j-1) \in \cV$), then add $(i-1,j)$ to the component.
Then in our modified component, the 
configurations of the form $\ytabsmall{ \ & \ \\ \none & \ }$ and $\ytabsmall{\ & \none \\ \ & \ }$
interleave, with one more $\ytabsmall{ \ & \ \\ \none & \ }$ configuration than $\ytabsmall{\ & \none \\ \ & \ }$ configuration.  Moreover, by construction, the number of $\ytabsmall{\ & \none \\ \ & \ }$ configurations is $q$ and the number of $\ytabsmall{ \ & \ \\ \none & \ }$ configurations is $p$, except when $\cC$ 
ends with two boxes in the first row, in which case the number of configurations  
of the latter form is $p-1$.  Thus if $\cC$  ends with two boxes in the first row then 
 $p-q=0$ and otherwise we have $p-q=1$.

Since   $\kkappa_1 \geq \nnu_1 +2$, the last connected component of  $\SD_{\kkappa/\nnu}$
ends with two 
boxes in the first row, and no other connected component of $\SD_{\kkappa/\nnu}$ has this property.
By summing $p-q$ over all connected components $\cC$ with multiple boxes,
we obtain $(|\hat\cV| + |\hat\cW|) - (|\cF| - |\hat\cW|) = \bparam -1$.
As $|\cF| = \cparam$, $|\hat\cV| = |\cV|$, and $|\hat\cW| =|\cW|$,
we can rewrite this identity as $|\cV| + 2|\cW| - \cparam = \bparam-1$.
\end{proof}

\begin{lemma}\label{jq-lem2b}
Let $u$ be an indeterminate and suppose $\kkappa_1 - \nnu_1 \geq 2$.
Then \[\sum_{n \geq 0} \widehat b^\nnu_{\llambda,(n)} \cdot \beta^{|\kkappa|- |\nnu| - (\kkappa_1-n)}\cdot u^{ (\kkappa_1-n)-\dparam }  = 2^{\aparam}   (u+\beta)^{\aparam+\bparam+\cparam-1}   (2u +\beta)^{\bparam-1}   .\]

\end{lemma}

\begin{proof}
Call a position in a shifted ribbon \defn{free} if it is not on the diagonal but is the first box in its connected component
in row reading order.
By Proposition~\ref{hat-b-prop}, the number $ \widehat b^\nnu_{\llambda,(n)}$ counts the ways that we can add a subset of removable
corner boxes of $\SD_\llambda$ to $\SD_{\nnu/\llambda}$ to form a shifted ribbon, and then assign either $1$ or $1'$ or $1'1$ 
to each free box in this ribbon, such that the number of boxes plus the number of entries equal to $1'1$ is $n$.

The shifted ribbons that can be formed by adding corner boxes of $\SD_\llambda$ to $\SD_{\nnu/\llambda}$
are   the unions of $\SD_{\nnu/\llambda}$ with arbitrary subsets of $\cU\sqcup \cV \sqcup \cW$.
Consider the number of free boxes in a ribbon formed by successively adding positions from $\cU\sqcup \cV \sqcup \cW$ to $\SD_{\nnu/\llambda}$.
Each time we add a box from $\cU$, this number is reduced by one,
since the added box either merges two connected components or changes the first box in a connected component from an off-diagonal position $(i,i+1)$
to a diagonal position $(i,i)$.
The number of free boxes is unchanged when we add a box from $\cV$,
since the added box will either be an isolated position on the diagonal or adjacent to exactly one connected component.
Finally, the number of free boxes increases by one each time we add a box from $\cW$ to $\SD_{\nnu/\llambda}$,
since the added box will always be an isolated position not on the diagonal.

If  $\SD_{\nnu/\llambda}$ intersects the diagonal then 
it has the same number $\aparam+\bparam$ of connected components as $\SD_{\kkappa/\nnu}$; 
otherwise, its number of connected components is $\aparam+\bparam-1$.
Either way, the number of free boxes in $\SD_{\nnu/\llambda}$ is $\aparam+\bparam-1$.
Given these observations and  Lemma~\ref{uvw-lem}, we deduce that
\[\ba
\sum_{n \geq 0} \widehat b^\nnu_{\llambda,(n)}  u^{n  }  &=     u ^{|\nnu|-|\llambda|}(2 + u)^{\aparam+\bparam-1}(1 + \tfrac{u}{2+u})^{|\cU|}(1+u)^{|\cV|} (1 + u(2+u))^{|\cW|}
\\& = 2^{\aparam}  u ^{|\nnu|-|\llambda|}     (2+u)^{\bparam-1} (1+u)^{\aparam+\bparam+\cparam-1} .
\ea\]
 This becomes the desired identity after dividing both sides by $u ^{|\nnu|-|\llambda|}$,
   replacing $u$ by $\beta u^{-1}$, and then multiplying both sides by $u^{\aparam+2\bparam+\cparam-2} = u^{|\kkappa|-|\nnu|-\dparam} =u^{\kkappa_1-\dparam+ |\llambda|-|\nnu| }$.
\end{proof}

The following lemma is the most technical part of our argument. 

\begin{lemma}\label{jq-lem3}
If  $n = \kkappa_1 - \kkappa_2$ and $m=\kkappa_2>0$ then 
$ \widehat b^\nnu_{\llambda,(n)} - \yyy{\kkappa}{\nnu}{m} = \begin{cases} 1 &\text{if }\nnu = (\kkappa_1,\kkappa_3,\kkappa_4,\dots)\\
1 &\text{if }\nnu = (\kkappa_1-1,\kkappa_3,\kkappa_4,\dots)
\\ 0&\text{otherwise}.
\end{cases}$
\end{lemma}

\begin{proof}
If $\kkappa_1 - \nnu_1 \geq 2$
then $\widehat b^\nnu_{\llambda,(n)} = \yyy{\kkappa}{\nnu}{m}$ for any integers  $n$ and $m$ with $n+m = \kkappa_1$.
For $m \geq 2$, this is because if we set $u=t-\beta$, then comparing Proposition~\ref{one-var-prop} and Lemma~\ref{jq-lem1}
shows that
$\yyy{\kkappa}{\nnu}{m}$ is the coefficient of 
$\beta^{|\kkappa|-|\nnu| -m} u^{m-\dparam}$
 in $2^{\aparam}   (u+\beta)^{\aparam+\bparam+\cparam-1}   (2u +\beta)^{\bparam-1} $,
and by Lemma~\ref{jq-lem2b} this is also equal to $\widehat b^\nnu_{\llambda,(n)}$ when $m= \kkappa_1-n$.  For $m\leq 1$, we have that $\widehat b^\nnu_{\llambda,(n)}$, $ \yyy{\kkappa}{\nnu}{m}$, and the relevant
 coefficient 
 in $2^{\aparam}   (u+\beta)^{\aparam+\bparam+\cparam-1}   (2u +\beta)^{\bparam-1} $ are all zero.
 
It remains to check the desired identity when  $\kkappa_1 -\nnu_1 \in \{0,1\}$.
From this point on, we fix $n=\kkappa_1-\kkappa_2$ and $m=\kkappa_2>0$;
in particular, in these cases $\kkappa$ and $\llambda$ are both nonempty.
If $|\nnu| - |\llambda| >n$ then  
$|\kkappa| - |\nnu| =  \kkappa_1 - (|\nnu| - |\llambda| )<\kkappa_1 - n= m$
and it follows from Proposition~\ref{hat-b-prop}
and Corollary~\ref{jq-cor1}
that $\widehat b^\nnu_{\llambda,(n)} =  \yyy{\kkappa}{\nnu}{m}= 0$ as desired.
We may therefore assume $|\nnu| - |\llambda| \leq n$. There are three cases:
\ben
\item[(1)] Suppose $\kkappa_1 = \nnu_1$. Then  
$|\nnu| - |\llambda| \geq \nnu_1 - \llambda_1 =  n$
also holds, so  $|\nnu| - |\llambda| = n$ and $|\kkappa| - |\nnu| = m$.
Since $\nnu_1-\llambda_1 = \kkappa_1-\kkappa_2 = n$, it follows that
$
\nnu = (\kkappa_1,\kkappa_3,\kkappa_4,\dots) = \llambda + (n,0,0,\dots)$
and we wish to show that $ \widehat b^\nnu_{\llambda,(n)} =1 +  \yyy{\kkappa}{\nnu}{m}$.
This holds as     $\aparam+\bparam=1$ so
$ \yyy{\kkappa}{\nnu}{m} = 2^{\aparam+\bparam-1} = 1$ while
$ \widehat b^\nnu_{\llambda,(n)} = |\bribbons{\nnu}{\llambda;n}| =2$
 by Corollary~\ref{jq-cor1}
and  Proposition~\ref{hat-b-prop}.

\item[(2)] Suppose  $\kkappa_1 - \nnu_1 = 1$.
Then
$ |\nnu| - |\llambda|    \geq \nnu_1 - \llambda_1 
= n-1$  so 
 $ |\nnu| - |\llambda|   \in \{n-1,n\}$.
If $|\nnu| - |\llambda| =n$ then $|\kkappa| - |\nnu| = m$ and 
$\SD_{\nnu/\llambda}$ consists of the $n-1$ boxes $(1,j) $ with $ \kkappa_2 < j < \kkappa_1$ 
plus one box $(i,j) \in \SD_{\kkappa/\llambda}$ with $i>1$.
In this event, 
  Corollary~\ref{jq-cor1} tells us
that 
$\yyy{\kkappa}{\nnu}{m} = 2^{\aparam+\bparam-1}$
while    Proposition~\ref{hat-b-prop} tells us that
$\widehat b^\nnu_{\llambda,(n)}=2^e$ where $e \in \{0,1,2\} $ is the number of connected components in $\SD_{\nnu/\llambda}$ not intersecting the diagonal. As $\aparam+\bparam$ is  the number of connected components
in $\SD_{\kkappa/\nnu}$, it follows that $e=\aparam+\bparam-1$
so $\widehat b^\nnu_{\llambda,(n)} = \yyy{\kkappa}{\nnu}{m} $ as claimed.

\item[(3)] Suppose  $\kkappa_1 - \nnu_1 = 1$
and $|\nnu| - |\llambda| =n-1$ so that $|\kkappa| - |\nnu| = m+1$.
Since  $\nnu_1 - \llambda_1=n-1$, we can therefore write  
$\nnu = (\kkappa_1-1,\kkappa_3,\kkappa_4,\dots) = \llambda + (n-1,0,0,\dots)$,
so we wish to show that $ \widehat b^\nnu_{\llambda,(n)} - \yyy{\kkappa}{\nnu}{m}=1$.
There are three subcases:
\begin{itemize}

\item If $n=1$ then   $\nnu=\llambda$ so $\aparam=0$, $\bparam=1$, and $\yyy{\kkappa}{\nnu}{m}=2^{\aparam+\bparam-2}(2\aparam+3\bparam+2\cparam-3)=\cparam$ by Corollary~\ref{jq-cor1}.
At the same time, it follows from 
Proposition~\ref{hat-b-prop} that
  $\widehat b^\nnu_{\llambda,(n)}=f+2g$ where $f \in \{0,1\}$ and $g\geq0$ are the numbers of removable corner boxes of $\SD_\llambda$
  on and off the diagonal.
 One checks that $\cparam+1=f+2g$ in this situation, as illustrated in Figure~\ref{fig:case psi = gamma}.

\begin{figure}[ht]
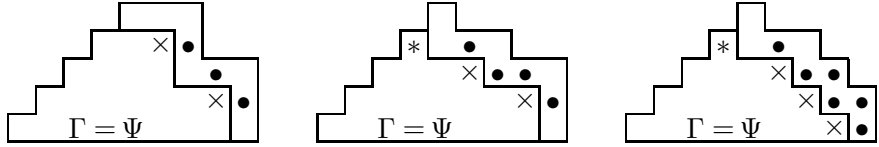

\[
\begin{young}[10pt][c] 
, & , &, &, & ]= & & =] \ynobottom  \\ 
, & , &, & \ynobottom  & \ynobottom & =]\ynobottom \times & \ynobottom\ynotop \bullet \\
, & , & ]= \ynobottom & , & , & =]\ynobottom\ynotop & ]= \ynotop & \bullet & =] \ynobottom \\
, & ]= \ynobottom & \ynobottom\ynotop & \ynobottom\ynotop  & , & , & \ynobottom & =]\ynobottom \times & ]=] \ynobottom\ynotop \bullet \\
]= & \ynotop& \ynotop \llambda & \ynotop = & \ynotop \nnu & \ynotop & \ynotop& \ynotop & ]=]\ynotop
\end{young}
\qquad
\begin{young}[10pt][c] 
, & , &, &, & ]=]\ynobottom  \\ 
, & , &, & ]=] \ynobottom * & ]=\ynotop & \bullet & =]\ynobottom \\
, & , & ]= \ynobottom & , & , & =]\ynobottom \times & ]= \ynotop \bullet & \bullet & =] \ynobottom \\
, & ]= \ynobottom & \ynobottom\ynotop & \ynobottom\ynotop  & , & , & \ynobottom & =]\ynobottom \times & ]=] \ynobottom\ynotop \bullet \\
]= & \ynotop& \ynotop \llambda & \ynotop = & \ynotop \nnu & \ynotop & \ynotop& \ynotop & ]=]\ynotop
\end{young}
\qquad
\begin{young}[10pt][c] 
, & , &, &, & ]=]\ynobottom  \\ 
, & , &, & ]=] \ynobottom * & ]=\ynotop & \bullet & =]\ynobottom \\
, & , & ]= \ynobottom & , & , & =]\ynobottom \times & ]= \ynotop \bullet &=]\ynobottom \bullet  \\
, & ]= \ynobottom & \ynobottom\ynotop & \ynobottom\ynotop  & , & ,  & =]\ynobottom \times & ]= \ynotop \bullet  & =]\ynobottom \bullet \\
]= & \ynotop& \ynotop \llambda & \ynotop = & \ynotop \nnu & \ynotop & \ynotop& \ynotop \times & ]=]\ynotop \bullet
\end{young}
\]
\caption{Case (3) in the proof of Lemma~\ref{jq-lem3}: the subcase when $n = 1$, $\nnu = \llambda$.  Forced boxes in $\SD_{\kkappa / \nnu}$ are denoted by $\bullet$, removable corners of $\SD_{\llambda}$ off the diagonal are denoted by $\times$, and removable corners on the diagonal are denoted by $*$.}
\label{fig:case psi = gamma}
\end{figure}


\item  If $\llambda=(1)$ and $n>1$, so that $\nnu=(n)$ and $\kkappa = (n+1,1)$, 
  then $\aparam=2$ and $\bparam=\cparam=0$ and 
  one can check via Proposition~\ref{hat-b-prop} 
and Corollary~\ref{jq-cor1} that $\widehat b^\nnu_{\llambda,(n)} =2$ while $\yyy{\kkappa}{\nnu}{m}=1$.

\item If $\llambda \neq (1)$ and $n>1$, then $\aparam=\bparam=1$ so $\yyy{\kkappa}{\nnu}{m} =2^{\aparam+\bparam-2}(2\aparam+3\bparam+2\cparam-3)= 2(\cparam+1)$
by Corollary~\ref{jq-cor1}, while  Proposition~\ref{hat-b-prop} implies that
  $\widehat b^\nnu_{\llambda,(n)} = 1 + 2i + 2j +4k$ where $i\in\{0,1\}$, $j\in\{0,1\}$, and $k\geq 0$ are the numbers of removable corner boxes of $\SD_\llambda$ which are respectively in the first row, on the diagonal, or neither in the first row nor on the diagonal.  One checks that $\cparam+1=i+j+2k$ in this situation, as illustrated in Figure~\ref{fig:case 3.3}.

\begin{figure}[ht]
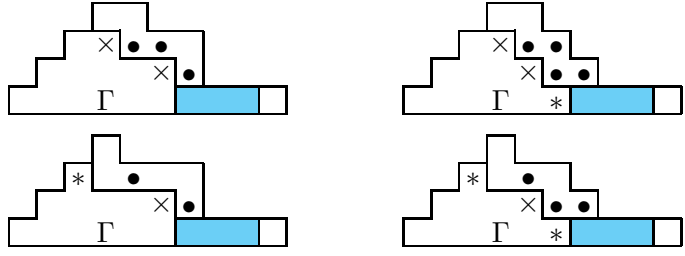

\[
\begin{young}[10pt][c] 
, & , &, & ]= & \ynobottom  \\ 
, & , & ]= \ynobottom & =]\ynobottom \times & ]= \ynotop \bullet & \bullet & =] \ynobottom \\
, & ]= \ynobottom & \ynobottom\ynotop & \ynobottom\ynotop & \ynobottom & =]\ynobottom \times & ]=] \ynotop \bullet \\
]= & \ynotop& \ynotop& \ynotop \Gamma & \ynotop& =]\ynotop & ]=![cyan!50] & ![cyan!50] & =]![cyan!50] & ]=]
\end{young}
\qquad\qquad
\begin{young}[10pt][c] 
, & , &, & ]= & \ynobottom  \\ 
, & , & ]= \ynobottom & =]\ynobottom \times & ]= \ynotop \bullet & =] \ynobottom \bullet \\
, & ]= \ynobottom & \ynobottom\ynotop & \ynobottom\ynotop & =]\ynobottom \times & ]= \ynotop \bullet & =] \bullet \\
]= & \ynotop& \ynotop& \ynotop \Gamma & \ynotop& =]\ynotop * & ]=![cyan!50] & ![cyan!50] & =]![cyan!50] & ]=]
\end{young}
\]
\[
\begin{young}[10pt][c] 
, & , &, & \ynobottom  \\ 
, & , & ]=] \ynobottom * & ]=\ynotop & \bullet &  & =] \ynobottom \\
, & ]= \ynobottom & \ynobottom\ynotop & \ynobottom\ynotop & \ynobottom & =]\ynobottom \times & ]=] \ynotop \bullet \\
]= & \ynotop& \ynotop& \ynotop \Gamma & \ynotop& =]\ynotop & ]=![cyan!50] & ![cyan!50] & =]![cyan!50] & ]=]
\end{young}
\qquad\qquad
\begin{young}[10pt][c] 
, & , &, & \ynobottom  \\ 
, & , & ]=] \ynobottom * & ]=\ynotop & \bullet &  =] \ynobottom \\
, & ]= \ynobottom & \ynobottom\ynotop & \ynobottom\ynotop & =]\ynobottom \times & ]= \ynotop \bullet & =] \bullet \\
]= & \ynotop& \ynotop& \ynotop \Gamma & \ynotop& =]\ynotop * & ]=![cyan!50] & ![cyan!50] & =]![cyan!50] & ]=]
\end{young}
\]
\caption{Case (3) in the proof of Lemma~\ref{jq-lem3}: the subcase when $n > 1$, $\nnu \neq (1)$.  Left column $i = 0$, right column $i = 1$; top row $j = 0$, bottom row $j = 1$.  The blue cells represent $\nnu / \llambda$. }
\label{fig:case 3.3}
\end{figure}

  \end{itemize}
    In each subcase we have $ \widehat b^\nnu_{\llambda,(n)} - \yyy{\kkappa}{\nnu}{m}=1$ as desired.
  \een
This case analysis concludes the proof.
\end{proof}

By putting everything together we can prove our first main theorem.

\begin{theorem}\label{jq-thm}
If $\mu$ and $ \lambda$ are strict partitions then $ \jq_{\lambda/\mu}=\wtjq_{\lambda/\mu}$.
\end{theorem}

\begin{proof}
We first claim that if $n>0$ then
$ \wtjq_\lambda \wtjq_n  = \sum_\nu  \widehat b^{\nu}_{\lambda,(n)}   \beta^{|\lambda|-|\nu|+n}   \wtjq_\nu$
where the sum is over all strict partitions $\nu$. This holds if $\lambda = \emptyset$
since $\wtjq_\lambda=1$ and  it is clear from Proposition~\ref{hat-b-prop} that
$\widehat b^{(n)}_{\emptyset,(n)} =1$ and $\widehat b^{\nu}_{\emptyset,(n)} =0$ for $\nu \neq (n)$. 
If $\lambda$ is nonempty 
then the desired formula 
follows by substituting 
 Lemmas~\ref{jq-lem2} and \ref{jq-lem3}
with $\kkappa := (n+\lambda_1,\lambda_1,\lambda_2,\dots)$ into
 Corollary~\ref{one-row-rules-cor}. 

Let $\prec$ be  the total order  on strict partitions with $\mu \prec \lambda$
if $\mu \neq \lambda$ and either
  $|\mu| < |\lambda|$
or $|\mu| = |\lambda|$ and the first index $i$ with $\mu_i \neq \lambda_i$ has $\mu_i > \lambda_i$.
Suppose $\lambda$ is a strict partition with $k>0$ nonzero parts.
Let $\mu := (\lambda_1,\lambda_2,\dots,\lambda_{k-1})$ and $n =\lambda_k$.
Then it follows from Proposition~\ref{hat-b-prop} that
$ \jq_{\mu} \jq_{n}  
= \jq_\lambda + \sum_{\nu \prec \lambda} \widehat b^{\nu}_{\mu,(n)}   \beta^{|\mu|-|\nu|+n}   \jq_\nu$
 so we can write $\jq_\lambda  =  \jq_{\mu} \jq_{n}  -  \sum_{\nu \prec \lambda} \widehat b^{\nu}_{\mu,(n)}   \beta^{|\mu|-|\nu|+n}   \jq_\nu$.
The same formula holds with each ``$\jq$'' replaced by ``$\wtjq$'' 
in view of the previous paragraph. 
Since $\jq_n=\wtjq_n$ for all integers $n\geq 0$ by Proposition~\ref{one-part-prop},
it follows by induction that $\jq_\lambda=\wtjq_\lambda$. 

The last thing to explain is how to generalize this identity to  
$\jq_{\lambda/\mu}=\wtjq_{\lambda/\mu}$.
As the power series $\wtjq_\lambda$ is symmetric, it is clear from its definition  that
$ \wtjq_\lambda({\bf x},{\bf y}) = \sum_{\mu} \wtjq_\mu({\bf x}) \wtjq_{\lambda/\mu}({\bf y})$.
Compare this with \eqref{jp-xy-eq}.
Since  $\wtjq_\lambda({\bf x},{\bf y})  =  \jq_\lambda({\bf x},{\bf y}) $ and 
since the symmetric functions $\jq_\mu({\bf x})  = \wtjq_\mu({\bf x}) $ are linearly independent,
we must have $\jq_{\lambda/\mu}=\wtjq_{\lambda/\mu}$ for all $\mu$.
\end{proof}

%
%

\subsection{Bar tableau generating functions without diagonal primes}

\def\APARAM{\aparam^\ast}
\def\BPARAM{\bparam^\ast}
\def\CPARAM{\cparam^\ast}
\def\DPARAM{\dparam^\ast}

We turn to our other family $\wtjp$ of bar tableau generating functions. 
The relevant arguments are similar to those above.
 Continue to assume that
 $\kkappa=(\kkappa_1,\kkappa_2,\dots) \supseteq \nnu \supseteq \llambda = (\kkappa_2,\kkappa_3,\dots)$
are strict partitions so that 
 $ \SD_{\kkappa/\nnu}$ is a shifted ribbon.
We retain the same meaning of a \defn{forced box} in $\SD_{\kkappa/\nnu}$.
We additionally define a position $(i,j) \in \SD_{\kkappa/\nnu}$ to be \defn{diagonally-forced}
 if $i=j$ or if $(i+1,j) = (j,j) \in \SD_{\kkappa/\nnu}$.
 
Throughout, we fix the meaning of four integer parameters $\APARAM$, $\BPARAM$, $\CPARAM$, $\DPARAM$
 which are slightly different from the ones $\aparam$, $\bparam$, $\cparam$, $\dparam$ defined in the previous subsection.
Let $\APARAM$ (respectively, $\BPARAM$) be the number of connected components in $\SD_{\kkappa/\nnu}$ that are disjoint from the diagonal
that are singleton sets  (respectively, have multiple boxes).
Let $\CPARAM$ be the number of forced or diagonally-forced boxes in $\SD_{\kkappa/\nnu}$.
Then define $\DPARAM := |\kkappa| - |\nnu| - \APARAM - 2\BPARAM-\CPARAM+1 \geq 1$.

\begin{lemma}\label{jp-lem1}
With the definitions above, $ \wtjp_{\kkappa/\nnu}(t) =    t(t-\beta)^{\DPARAM-1}   2^{\APARAM}   t^{\APARAM+\BPARAM+\CPARAM-1}  (2t  -\beta)^{\BPARAM} $.
\end{lemma}

\begin{proof}
If $\SD_{\kkappa/\nnu}$ has no diagonal positions then 
$\wtjp_{\kkappa/\nnu}(t) = \wtjq_{\kkappa/\nnu}(t)$ and $\aparam=\APARAM$, $\bparam=\BPARAM$, $\cparam=\CPARAM$, and $\dparam=\DPARAM+1$,
so the desired formula is equivalent to Lemma~\ref{jq-lem1}.

If instead $\SD_{\kkappa/\nnu}$ has just one connected component and this component intersects the diagonal, 
consider the elements of  $\ShBTP(\nnu/\llambda)$ with all entries in  $\{1' ,1\}$.
As one reads such a tableau in the row-reading order, 
each forced box and each diagonally-forced box must start a new bar,
while every remaining box is free to either start a new bar or continue the bar of its predecessor.
In particular, the first box of $\SD_{\kkappa/\nnu}$, which is some diagonal position $(j,j)$, must start a new bar and contain an unprimed entry,
and if the second box is the diagonally-forced position $(j-1,j)$ then it must contain a primed entry and start a new bar.
Comparing these observations with the definition  $ \wtjp_{\kkappa/\nnu}$, we see that 
 $ \wtjp_{\kkappa/\nnu}(t) = t^{\CPARAM}  (t-\beta)^{|\kkappa|-|\nnu| - \CPARAM}$,
 which is equivalent to the desired formula since in this special case $\APARAM=\BPARAM=0$ and $\DPARAM = |\kkappa|-|\nnu|-\CPARAM + 1$.
 
When not in these cases, $\SD_{\kkappa/\nnu} = \SD_{\kkappa^1/\nnu^1}\sqcup \SD_{\kkappa^2/\nnu^2}$
 for some strict partitions $\nnu^i\subsetneq \kkappa^i$ 
 such that $\SD_{\kkappa^1/\nnu^1}$ has no diagonal boxes and $\SD_{\kkappa^2/\nnu^2}$ has just one connected component
 which intersects the diagonal and is not adjacent to any box in $\SD_{\kkappa^1/\nnu^1}$.
Then $\wtjp_{\kkappa/\nnu}(t) = \wtjp_{\kkappa^1/\nnu^1}(t)  \wtjp_{\kkappa^2/\nnu^2}(t)$
 where    
 $ \wtjp_{\kkappa^i/\nnu^i}(t) =    (t-\beta)^{d_i-1}  2^{\APARAM_i}   t^{\APARAM_i+\BPARAM_i+\CPARAM_i}  (2t  -\beta)^{\BPARAM_i} $
 for integers with $\APARAM=\APARAM_1+\APARAM_2$, $\BPARAM=\BPARAM_1+\BPARAM_2$, $\CPARAM=\CPARAM_1+\CPARAM_2$, and $\DPARAM = \DPARAM_1 + \DPARAM_2 - 1$,
 which suffices.
\end{proof}

Since  $\{\wtjp_n(t): n\geq 0\}$ is a homogeneous $\ZZ[\beta]$-basis for  $\ZZ[\beta,t]$,
we can write 
\be\label{zzz-eq}
 \wtjp_{\kkappa/\nnu}(t)  =  \sum_{n\geq 0} \zzz{\kkappa}{\nnu}{n}   \cdot \beta^{|\kkappa|-|\nnu| - n}  \cdot \wtjp_n(t)
 \ee
 for unique numbers $\zzz{\kkappa}{\nnu}{n} \in \ZZ$.
\begin{corollary}\label{jp-cor1}
 The coefficients $\zzz{\kkappa}{\nnu}{n}$ are all nonnegative integers
with the following  properties:
 \ben
 \item[(a)] If $n = 0 < |\kkappa|-|\nnu|$ or $n > |\kkappa|-|\nnu|$ then $\zzz{\kkappa}{\nnu}{n} =  0$.
 \item[(b)] If $0 < n = |\kkappa| - |\nnu|$ then $\zzz{\kkappa}{\nnu}{n} =  2^{ \APARAM+\BPARAM}$.
 \item[(c)] If $0 < n = |\kkappa| - |\nnu| - 1$ then $\zzz{\kkappa}{\nnu}{n} = 2^{\APARAM+\BPARAM-1}(2\APARAM+3\BPARAM+2\CPARAM-2)$.
 \een
 \end{corollary}
 
 \begin{proof} 
 The argument is similar to the proof of Corollary~\ref{jq-cor1}, and follows as 
a simple exercise in algebra from Proposition~\ref{one-var-prop} and Lemma~\ref{jp-lem1}.
We omit the details.
\end{proof}


\begin{lemma}\label{jp-lem2}
If $0 \leq n \leq \kkappa_1$ is an integer then
\[
\wtjp_{\kkappa/(n)} = \sum_{\substack{\mu \text{ strict} \\  \llambda \subseteq \mu \subseteq \kkappa}}   \zzz{\kkappa}{\mu}{n} \cdot \beta^{|\kkappa|-|\mu|-n}\cdot  \wtjp_\mu .
\]
\end{lemma}

\begin{proof}
Repeat the proof of Lemma~\ref{jq-lem2}, replacing   ``$\wtjq$''
by ``$\wtjp$'' and   ``$\yyyletter$'' by ``$\zzzletter$''.
\end{proof}

Consider the set of removable corner boxes $(i,j) \in \SD_\llambda$ with $(i+1,j+1) \notin \SD_{\nnu/\llambda}$.
Let $\cU^\ast$, $\cV^\ast$, and $\cW^\ast$ be the sets of such corners $(i,j)$ 
for which the size of 
$\{(i,j+1),(i+1,j)\} \cap \SD_{\nnu/\llambda}$
is exactly two, one, or zero, respectively.
The union $\cU^\ast \sqcup \cV^\ast \sqcup \cW^\ast$ of these three sets is the same as the union
of  $\cU \sqcup \cV \sqcup \cW$ of the sets defined before Lemma~\ref{uvw-lem}.

\begin{lemma}\label{uvw-lem2}
Suppose $\kkappa_1 - \nnu_1 \geq 2$.
Then $ |\cU^\ast| = \APARAM $ and $ |\cV^\ast| + 2|\cW^\ast| = \BPARAM+\CPARAM-1$.
\end{lemma}

\begin{proof}
Let $(i,i)$ be the unique diagonal position in $\SD_{\kkappa/\llambda}$.
If this box is in $\SD_{\nnu/\llambda}$ rather than $\SD_{\kkappa/\nnu}$,
then  $\cU=\cU^\ast$, $\cV=\cV^\ast$, and $\cW = \cW^\ast$
along with $\aparam=\APARAM$, $\bparam=\BPARAM$, and $\cparam=\CPARAM$, so the desired properties hold by Lemma~\ref{uvw-lem}.
If  box $(i,i)$ belongs to $\SD_{\kkappa/\nnu}$ but is not adjacent to any other box
in $\SD_{\kkappa/\nnu}$, then we cannot have $(i,i+1) \in \SD_{\kkappa/\llambda}$
so it must hold that $(i-1,i) \in \SD_{\nnu/\llambda}$,
which means that  $\cU = \cU^\ast \sqcup \{(i-1,i-1)\} $ and  $ \cV^\ast = \cV \sqcup \{(i-1,i-1)\}$
and therefore
\[|\cU|-1 = |\cU^\ast|, \ |\cV| +1= |\cV^\ast|,  \ |\cW| = |\cW^\ast|
\quad\text{and}\quad
\aparam-1=\APARAM,\ \bparam=\BPARAM,\ \cparam+1=\CPARAM.
\]
Suppose finally that box $(i,i)$ belongs to a connected component of $\SD_{\kkappa/\nnu}$ with multiple boxes.
 Similar to our reasoning above, if the second box of $\SD_{\kkappa/\nnu}$ is $(i,i+1)$ 
then we have 
\[|\cU| = |\cU^\ast|, \ |\cV| = |\cV^\ast|,  \ |\cW| = |\cW^\ast|
\quad\text{and}\quad
\aparam=\APARAM,\ \bparam-1=\BPARAM,\ \cparam+1=\CPARAM,
\]
 while if the second box of $\SD_{\kkappa/\nnu}$ is $(i-1,i)$ 
then we have
\[|\cU| = |\cU^\ast|, \ |\cV| -1= |\cV^\ast|,  \ |\cW|+1 = |\cW^\ast|
\quad\text{and}\quad
\aparam=\APARAM,\ \bparam-1=\BPARAM,\ \cparam+2=\CPARAM.
\]
In each of these cases the desired identities follow by Lemma~\ref{uvw-lem}.
\end{proof}

\begin{lemma}\label{jp-lem2b}
If $\kkappa_1 - \nnu_1 \geq 2$ 
then 
\[\sum_{n \geq 0} \widehat c^\nnu_{\llambda,(n)} \cdot \beta^{|\kkappa|- |\nnu| - (\kkappa_1-n)} \cdot u^{ (\kkappa_1-n)-\DPARAM } = 2^{\APARAM}  (u+\beta)^{\APARAM+\BPARAM+\CPARAM-1}  (2u +\beta)^{\BPARAM}    .\]

\end{lemma}
\begin{proof}
By Proposition~\ref{hat-b-prop}, the number $ \widehat c^\nnu_{\llambda,(n)}$ counts the ways that we can add a subset of removable
corner boxes of $\SD_\llambda$ to $\SD_{\nnu/\llambda}$ to form a shifted ribbon, and then assign either $1$ or $1'$ or $1'1$ 
to the first box in each connected component of this ribbon, such that the number of boxes plus the number of entries equal to $1'1$ is $n$.
As in the proof of Lemma~\ref{jq-lem2b}, 
the shifted ribbons that arise in this way are the 
unions of $\SD_{\nnu/\llambda}$
with arbitrary subsets of $\cU^\ast \sqcup \cV^\ast \sqcup \cW^\ast$.
The ribbon $\SD_{\nnu/\llambda}$ starts out with $\APARAM+\BPARAM$ connected components, and 
each box added from $\cU^\ast $, $ \cV^\ast $, and $ \cW^\ast$ respectively adds $-1$, $0$, and $1$ to this number.
Combining these observations with  Lemma~\ref{uvw-lem2}, we deduce that
\[\ba
\sum_{n \geq 0} \widehat c^\nnu_{\llambda,(n)}  u^{n  }  &=     u ^{|\nnu|-|\llambda|}(2 + u)^{\APARAM+\BPARAM}(1 + \tfrac{u}{2+u})^{|\cU^\ast|}(1+u)^{|\cV^\ast|} (1 + u(2+u))^{|\cW^\ast|}
\\& = 2^{\APARAM}  u ^{|\nnu|-|\llambda|}     (2+u)^{\BPARAM} (1+u)^{\APARAM+\BPARAM+\CPARAM-1} .
\ea\]
 This becomes the desired identity after dividing both sides by $u ^{|\nnu|-|\llambda|}$,
   replacing $u$ by $\beta u^{-1}$, and then multiplying both sides by $u^{\APARAM+2\BPARAM+\CPARAM-1} = u^{|\kkappa|-|\nnu|-\DPARAM} =u^{\kkappa_1-\DPARAM+ |\llambda|-|\nnu| }$.
\end{proof}

\begin{lemma}\label{jp-lem3}
If $n = \kkappa_1 - \kkappa_2$ and $m=\kkappa_2>0$ then 
$ \widehat c^\nnu_{\llambda,(n)} - \zzz{\kkappa}{\nnu}{m} = \begin{cases} 1 &\text{if }\nnu = (\kkappa_1,\kkappa_3,\kkappa_4,\dots)\\
1 &\text{if }\nnu = (\kkappa_1-1,\kkappa_3,\kkappa_4,\dots)
\\ 0&\text{otherwise}.
\end{cases}$
\end{lemma}

\begin{proof}
As in 
 the proof of Lemma~\ref{jq-lem3},
if $\kkappa_1-\nnu_1 \geq 2$ and $n+m=\kkappa_1$, then it follows 
from Proposition~\ref{one-var-prop} and Lemmas~\ref{jp-lem1} and \ref{jp-lem2b}
 that 
$\widehat c^\nnu_{\llambda,(n)}$ and $ \zzz{\kkappa}{\nnu}{m} $ 
are both equal to the coefficient of $\beta^{|\kkappa|-|\nnu| -m} u^{m-\DPARAM}$
 in $2^{\APARAM}   (u+\beta)^{\APARAM+\BPARAM+\CPARAM-1}   (2u +\beta)^{\BPARAM} $.
 
To handle the cases when $\kkappa_1-\nnu_1 \in \{0,1\}$, we fix $n=\kkappa_1-\kkappa_2$ and $m=\kkappa_2>0$; in particular, $\kkappa$ and $\llambda$ are both nonempty.
If  $|\nnu|-|\llambda|>n$ then $|\kkappa|-|\nnu| < m$ and
it follows from Proposition~\ref{hat-b-prop}
and Corollary~\ref{jp-cor1} that 
 $\widehat c^\nnu_{\llambda,(n)} = \zzz{\kkappa}{\nnu}{m}  = 0$, as desired.
 We may therefore assume $|\nnu|-|\llambda|\leq n$. 
 There are three main cases, which are similar to the ones in the proof of Lemma~\ref{jq-lem3}:
 \ben
 \item[(1)] If $\kkappa_1=\nnu_1$ then $\nnu = (\kkappa_1,\kkappa_3,\kkappa_4,\dots)= \llambda + (n,0,0,\dots)$
 as in case (1) of the proof  of Lemma~\ref{jq-lem3}.
 Then $\APARAM=\BPARAM=0$ so $\zzz{\kkappa}{\nnu}{n} = 2^{\APARAM+\BPARAM}=1$ by Corollary~\ref{jp-cor1},
 and 
since $\llambda$ is nonempty, one has 
 $\widehat c^\nnu_{\llambda,(n)} = \widehat b^\nnu_{\llambda,(n)}=2$ by Proposition~\ref{hat-b-prop},
 so $\widehat c^\nnu_{\llambda,(n)} - \zzz{\kkappa}{\nnu}{m}=1 $, as claimed.

\item[(2)] If $\kkappa_1-\nnu_1=1$ then $|\nnu|-|\llambda| \in \{n-1,n\}$ as in case (2) of the proof  of Lemma~\ref{jq-lem3}.
If we also have $|\nnu|-|\llambda| = n$ then 
$\SD_{\nnu/\llambda}$
must again consist of the $n-1$   boxes $(1,j) $ with $ \kkappa_2 < j < \kkappa_1$ 
along with one additional box $(i,j) \in \SD_{\kkappa/\llambda}$ with $i>1$.
In this case
$\zzz{\kkappa}{\nnu}{m} = 2^{\APARAM+\BPARAM}$ by Corollary~\ref{jp-cor1} 
while 
$\widehat c^\nnu_{\llambda,(n)}=2^E$ where $E \in \{1,2\} $ is the number of connected components in $\SD_{\nnu/\llambda}$
by Proposition~\ref{hat-b-prop}. 
We have $\APARAM+\BPARAM = E$
so $\widehat c^\nnu_{\llambda,(n)} = \zzz{\kkappa}{\nnu}{m} $, as needed.

\item[(3)] Finally, if  $\kkappa_1 - \nnu_1 = 1$
and $|\nnu| - |\llambda| =n-1$,
then as in case (3) of 
the proof  of Lemma~\ref{jq-lem3}
we must   have 
$\nnu = (\kkappa_1-1,\kkappa_3,\kkappa_4,\dots) = \llambda + (n-1,0,0,\dots)$.
There are now only two subcases:  \begin{itemize}

\item If $n=1$ then   $\nnu=\llambda$, so $\APARAM=\BPARAM=0$ and $\zzz{\kkappa}{\nnu}{m}=2^{\APARAM+\BPARAM-1}(2\APARAM+3\BPARAM+2\CPARAM-2)=\CPARAM-1$ by Corollary~\ref{jp-cor1},
while 
Proposition~\ref{hat-b-prop} tells us that
  $\widehat c^\nnu_{\llambda,(n)}=2G$ where $G>0$ is the number of removable corner boxes of $\SD_\llambda$.
 One checks that $\CPARAM=2G$, as illustrated in Figure~\ref{fig:case psi = gamma 2}.

\begin{figure}[ht]
\[
\begin{young}[10pt][c] 
, & , &, &, & ]= \cdot & & =] \ynobottom  \\ 
, & , &, & \ynobottom  & \ynobottom & =]\ynobottom \times & \ynobottom\ynotop \bullet \\
, & , & ]= \ynobottom & , & , & =]\ynobottom\ynotop & ]= \ynotop & \bullet & =] \ynobottom \\
, & ]= \ynobottom & \ynobottom\ynotop & \ynobottom\ynotop  & , & , & \ynobottom & =]\ynobottom \times & ]=] \ynobottom\ynotop \bullet \\
]= & \ynotop& \ynotop \llambda & \ynotop = & \ynotop \nnu & \ynotop & \ynotop& \ynotop & ]=]\ynotop
\end{young}
\qquad
\begin{young}[10pt][c] 
, & , &, &, & ]=]\ynobottom \cdot  \\ 
, & , &, & ]=] \ynobottom \times & ]=\ynotop  \cdot & \bullet & =]\ynobottom \\
, & , & ]= \ynobottom & , & , & =]\ynobottom \times & ]= \ynotop \bullet & \bullet & =] \ynobottom \\
, & ]= \ynobottom & \ynobottom\ynotop & \ynobottom\ynotop  & , & , & \ynobottom & =]\ynobottom \times & ]=] \ynobottom\ynotop \bullet \\
]= & \ynotop& \ynotop \llambda & \ynotop = & \ynotop \nnu & \ynotop & \ynotop& \ynotop & ]=]\ynotop
\end{young}
\qquad
\begin{young}[10pt][c] 
, & , &, &, & ]=]\ynobottom \cdot  \\ 
, & , &, & ]=] \ynobottom \times & \ynotop\ynobottom \cdot\\
, & , & ]= \ynobottom & , & ]=\ynotop &  \bullet & &=]\ynobottom   \\
, & ]= \ynobottom & \ynobottom\ynotop & \ynobottom\ynotop  & , & ,  & =]\ynobottom \times & ]= \ynotop \bullet  & =]\ynobottom \bullet \\
]= & \ynotop& \ynotop \llambda & \ynotop = & \ynotop \nnu & \ynotop & \ynotop& \ynotop \times & ]=]\ynotop \bullet
\end{young}
\]
\caption{Case (3) in the proof of Lemma~\ref{jp-lem3}: the subcase when $n = 1$, $\nnu = \llambda$.  Forced boxes in $\SD_{\kkappa / \nnu}$ are denoted by $\bullet$, diagonally-forced boxes are denoted by $\cdot$, and removable corners of $\SD_{\llambda}$ are denoted by $\times$.}
\label{fig:case psi = gamma 2}
\end{figure}

\item If  $n>1$ then $\APARAM=1$ and $\BPARAM=0$ so $\yyy{\kkappa}{\nnu}{m} =2^{\APARAM+\BPARAM-1}(2\APARAM+3\BPARAM+2\CPARAM-2)= 2\CPARAM$
by Corollary~\ref{jp-cor1}. But in this case Proposition~\ref{hat-b-prop} implies that
  $\widehat c^\nnu_{\llambda,(n)} = 1 + 2J + 4K$ where $J\in\{0,1\}$ and $K\geq 0$ are the numbers of removable corner boxes of $\SD_\llambda$ which are in or not in the first row.  One checks that $\CPARAM=J+2K$ as needed.  
 
\begin{figure}[ht]
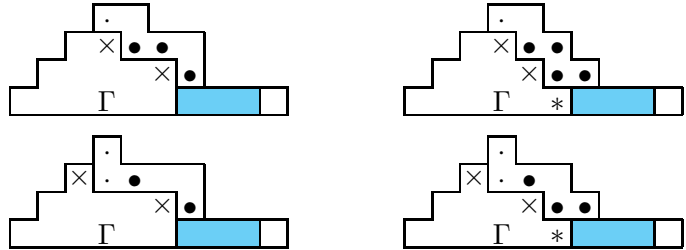

\[
\begin{young}[10pt][c] 
, & , &, & ]= \cdot & \ynobottom  \\ 
, & , & ]= \ynobottom & =]\ynobottom \times & ]= \ynotop \bullet & \bullet & =] \ynobottom \\
, & ]= \ynobottom & \ynobottom\ynotop & \ynobottom\ynotop & \ynobottom & =]\ynobottom \times & ]=] \ynotop \bullet \\
]= & \ynotop& \ynotop& \ynotop \Gamma & \ynotop& =]\ynotop & ]=![cyan!50] & ![cyan!50] & =]![cyan!50] & ]=]
\end{young}
\qquad\qquad
\begin{young}[10pt][c] 
, & , &, & ]= \cdot & \ynobottom  \\ 
, & , & ]= \ynobottom & =]\ynobottom \times & ]= \ynotop \bullet & =] \ynobottom \bullet \\
, & ]= \ynobottom & \ynobottom\ynotop & \ynobottom\ynotop & =]\ynobottom \times & ]= \ynotop \bullet & =] \bullet \\
]= & \ynotop& \ynotop& \ynotop \Gamma & \ynotop& =]\ynotop * & ]=![cyan!50] & ![cyan!50] & =]![cyan!50] & ]=]
\end{young}
\]
\[
\begin{young}[10pt][c] 
, & , &, & \ynobottom \cdot \\ 
, & , & ]=] \ynobottom \times  & ]=\ynotop \cdot & \bullet &  & =] \ynobottom \\
, & ]= \ynobottom & \ynobottom\ynotop & \ynobottom\ynotop & \ynobottom & =]\ynobottom \times & ]=] \ynotop \bullet \\
]= & \ynotop& \ynotop& \ynotop \Gamma & \ynotop& =]\ynotop & ]=![cyan!50] & ![cyan!50] & =]![cyan!50] & ]=]
\end{young}
\qquad\qquad
\begin{young}[10pt][c] 
, & , &, & \ynobottom \cdot \\ 
, & , & ]=] \ynobottom \times & ]=\ynotop \cdot & \bullet &  =] \ynobottom \\
, & ]= \ynobottom & \ynobottom\ynotop & \ynobottom\ynotop & =]\ynobottom \times & ]= \ynotop \bullet & =] \bullet \\
]= & \ynotop& \ynotop& \ynotop \Gamma & \ynotop& =]\ynotop * & ]=![cyan!50] & ![cyan!50] & =]![cyan!50] & ]=]
\end{young}
\]
\caption{Case (3) in the proof of Lemma~\ref{jp-lem3}: the subcase when $n > 1$.  Left column $J = 0$, right column $J = 1$.  The blue cells represent $\nnu / \llambda$. }
\end{figure}

  \end{itemize}
    In each subcase we have $ \widehat c^\nnu_{\llambda,(n)} - \zzz{\kkappa}{\nnu}{m}=1$ as desired.
  \een
This case analysis concludes the proof.
\end{proof}

We now arrive at our second main theorem.

\begin{theorem}\label{jp-thm}
If $\mu$ and $ \lambda$ are strict partitions then $ \jp_{\lambda/\mu}=\wtjp_{\lambda/\mu}$.
\end{theorem}

\begin{proof}
Our argument has the same structure as the proof of Theorem~\ref{jq-thm}.
Let $\prec$ be the total order on strict partitions defined in that proof.
Substituting Lemmas~\ref{jp-lem2} and \ref{jp-lem3} with  $\kkappa := (n+\lambda_1,\lambda_1,\lambda_2,\dots)$ 
into Corollary~\ref{one-row-rules-cor},
shows
that  
$ \wtjp_\lambda \wtjq_n  = \sum_\nu  \widehat c^{\nu}_{\lambda,(n)}   \beta^{|\lambda|-|\nu|+n}   \wtjp_\nu$
for all $n>0$.
Next, if $\lambda$  has $k>0$ nonzero parts,
 $\mu := (\lambda_1,\lambda_2,\dots,\lambda_{k-1})$, and $n :=\lambda_k$,
then 
$ \jp_{\mu} \jq_{n} 
= 2\jp_\lambda + \sum_{\nu \prec \lambda} \widehat c^{\nu}_{\mu,(n)}   \beta^{|\mu|-|\nu|+n}   \jp_\nu$
by Proposition~\ref{hat-b-prop},
 so    $\jp_\lambda  =  \frac{1}{2}\jp_{\mu} \jq_{n}  - \frac{1}{2} \sum_{\nu \prec \lambda} \widehat c^{\nu}_{\mu,(n)}   \beta^{|\mu|-|\nu|+n}   \jp_\nu$ and 
an analogous formula holds for $\wtjp_{\lambda} $.
Since $\jp_n=\wtjp_n$ and $\jq_n=\wtjq_n$ by Proposition~\ref{one-part-prop},
the identity $\jp_\lambda=\wtjp_\lambda$ follows by induction.
To deduce that $\jp_{\lambda/\mu}=\wtjp_{\lambda/\mu}$,
  one can repeat the last paragraph of the proof of Theorem~\ref{jq-thm}
after changing each ``$\jq$'' to ``$\jp$''.
\end{proof}

%
%

\subsection{Reverse plane partition generating functions}\label{rpp-gf-sect}

\def\finLam{\widehat{\Lambda}_{\mathrm{fin}}}
\def\mindeg{\mathrm{mindeg}}

It remains to show that $\gp_{\lambda/\mu} = \wtgp_{\lambda/\mu}$
and $\gq_{\lambda/\mu} = \wtgq_{\lambda/\mu}$.
Thankfully, we can derive this from Theorems~\ref{jq-thm} and \ref{jp-thm}
by a relatively succinct formal argument.

Let $\Lambda $ be the free $\ZZ[\beta]\llbracket x_1,x_2,\dots\rrbracket$-module of   linear combinations of strict partitions
and let $\widehat{\Lambda} $ be the $\ZZ[\beta]\llbracket x_1,x_2,\dots\rrbracket$-module of formal (possibly infinite) linear combinations of strict partitions.  
Write $\mindeg(f)$ to denote the smallest degree 
of any monomial in $0\neq f \in \ZZ[\beta]\llbracket x_1,x_2,\dots\rrbracket$, where $\deg \beta = 0$ and $\deg x_i = 1$, and set $\mindeg(0)=\infty$.
Let  $\finLam$ be the submodule of $\widehat{\Lambda} $ consisting of the elements $\sum_\lambda f_\lambda \cdot \lambda$
for which $\{ \lambda : \mindeg (f_\lambda) = n\}$ is finite for all integers $n$.
Then write
\[\langle \cdot,\cdot \rangle : \finLam \times \Lambda  \to \ZZ[\beta]\llbracket x_1,x_2,\dots\rrbracket\]
for the  bilinear form,
continuous in the first coordinate,
with $\langle \lambda,\mu \rangle = \delta_{\lambda\mu}$ for all strict partitions.

Fix an element $t \in \ZZ[\beta]\llbracket x_1,x_2,\dots\rrbracket$ with $\mindeg(t)>0$.
Define $\cP(t)$ and $\cQ(t)$
to be the continuous linear maps $\finLam \to \finLam$ 
such that for all strict partitions $\mu$ and $\lambda$ one has 
\be
\langle \cP(t) \mu, \lambda\rangle =  \GP_{\lambda\ss\mu}(t)
\quand
\langle \cQ(t) \mu, \lambda\rangle =  \GQ_{\lambda\ss\mu}(t).
\ee
Both $\cP(t)$ and $\cQ(t)$ make sense as maps 
$\widehat{\Lambda} \to \widehat{\Lambda}$ but not as maps ${\Lambda} \to {\Lambda}$, 
since they send any single $\mu$ to the infinite sums $\sum_\lambda \GP_{\lambda\ss\mu}(t) \cdot \lambda$ 
and $\sum_\lambda \GQ_{\lambda\ss\mu}(t)\cdot\lambda$ over all strict partitions $\lambda \supseteq \mu$.
These operators are well-defined maps $\finLam\to\finLam$ because $\mindeg (\GP_{\lambda\ss\mu}(t))=\mindeg (\GQ_{\lambda\ss\mu}(t)) = |\lambda/\mu| \cdot \mindeg(t)$.
Similarly, define $\cp(t)$ and $\cq(t)$
to be the continuous linear maps $\finLam \to \finLam$ with 
\be
\langle \cp(t) \lambda, \mu\rangle =  \wtgp_{\lambda/\mu}(t)
\quand
\langle \cq(t) \lambda, \mu\rangle =  \wtgq_{\lambda/\mu}(t).
\ee
Finally, let $\cpp(t)$ and $\cqq(t)$
 be the continuous linear maps $\finLam \to \finLam$ with 
\be
\langle \cpp(t) \lambda, \mu\rangle =  \wtjp_{\lambda/\mu}(t) = \jp_{\lambda/\mu}(t)
\quand
\langle \cqq(t) \lambda, \mu\rangle =  \wtjq_{\lambda/\mu}(t) = \jq_{\lambda/\mu}(t).
\ee
The second equalities in this last pair of definitions rely on Theorems~\ref{jq-thm} and \ref{jp-thm}.

The operators $ \cp(t) $, $ \cq(t)$, $ \cpp(t) $, $ \cqq(t)$
all make sense as maps ${\Lambda} \to {\Lambda}$ but not as maps $\widehat{\Lambda} \to \widehat{\Lambda}$.
 For example, the coefficient of $\emptyset$
in the image of $(1) + (2) + (3) + \dots \in \widehat{\Lambda} - \finLam$ under any of these operators would be in 
$\ZZ\llbracket \beta, t\rrbracket$ rather than $\ZZ[\beta]\llbracket t\rrbracket$.
These operators send $\finLam\to\finLam$
because they each map $\lambda$ into the $\ZZ[\beta]\llbracket t\rrbracket$-span of the finite set of strict partitions $\mu \subseteq \lambda$.

\begin{lemma}\label{gg-lem1}
It holds that $\cp(-t)\cpp(t) = \cpp(t)\cp(-t) =  \cq(-t)\cqq(t) = \cqq(t)\cq(-t) = 1$.
\end{lemma}

\begin{proof}
For each integer $j>0$ let $r_j $, $c_j$, and $b_j$ be the 
  linear maps ${\Lambda}  \to {\Lambda} $ 
with 
\[
 r_j(\lambda) = \sum (-\beta)^{|\mu|-|\lambda|-1} \mu,
\quad c_j(\lambda) = \sum  (-\beta)^{|\nu|-|\lambda|-1} \nu,
\quand
 b_j(\lambda) = \sum  (-\beta)^{|\kappa|-|\lambda|-1} \kappa,
\]
where the sums are respectively over the finite sets of strict partitions $\mu\subsetneq\lambda$, $\nu\subsetneq\lambda$,
and $\kappa\subsetneq \lambda$
such that  $\SD_{\lambda/\mu}$ is contained in row $j$,
$\SD_{\lambda/\nu}$ is contained in column $j$,
and
$\SD_{\lambda/\kappa}$ is contained in column $j$ but does not contain a box on the main diagonal.
Next, fix an integer $n>0$ and define
\[
\ba
R_n(t)& := (1+t r_1)(1+tr_2) \cdots (1+tr_n), \\
C_n(t) &:= (1+tc_1)(1+tc_2)\cdots (1+tc_n), \\
B_n(t) &:= (1+t b_1)(1+t b_2)\cdots (1+t b_n),
\ea
\qquand
\ba
\tilde R_n(t)& := \tfrac{1}{1-t r_n} \cdots\tfrac{1}{1-tr_2}\tfrac{1}{1-tr_1}, \\
\tilde C_n(t) &:=   \tfrac{1}{1-t c_n} \cdots\tfrac{1}{1-tc_2}\tfrac{1}{1-tc_1} , \\
\tilde B_n(t) &:=  \tfrac{1}{1-t b_n} \cdots\tfrac{1}{1-t b_2}\tfrac{1}{1-tb_1} .
\ea
\]
Here we interpret $\tfrac{1}{1-t r_j} $ as the geometric series $1 + t r_j + t^2 r_j^2 + \dots$ (and likewise for $\tfrac{1}{1-t c_j} $ and $\tfrac{1}{1-t b_j} $);
this makes sense as maps ${\Lambda}  \to {\Lambda} $ because $r_j$, $c_j$, and $b_j$ are locally nilpotent, i.e., for each $\lambda$ we have $r_j^m(\lambda) = 0$ for sufficiently large $m$.

Now suppose $\lambda$ is a strict partition whose shifted diagram fits inside $[n]\times [n]$.
By construction,
\[
\cq(t) \lambda = \sum_{\mu \subseteq \lambda}   \wtgq_{\lambda/\mu}(t) \mu  = R_n(t) C_n(t) \lambda
\quand
\cqq(t) \lambda = \sum_{\mu \subseteq \lambda}   \wtjq_{\lambda/\mu}(t) \mu  = \tilde C_n(t) \tilde R_n(t) \lambda.
\]
Here, one should think of $C_n(t)$ as carving out space from $\lambda$ for the unprimed
entries in a ShRPP while $R_n(t)$ carves out space for the primed entries.
On the other hand, $ \tilde R_n(t) $ carves out space from $\lambda$ for the unprimed row bars in a ShBT
while $ \tilde C_n(t) $ carves out space for the primed column bars.
Since $\SD_\mu\subseteq [n]\times [n]$ for all $\mu\subseteq \lambda$, 
it follows that
\[
\cqq(-t) \cq(t) \lambda = \sum_{\mu \subseteq \lambda}   \wtgq_{\lambda/\mu}(t) \cqq(-t)\mu
=\sum_{\mu \subseteq \lambda}   \wtgq_{\lambda/\mu}(t)  \tilde C_n(-t) \tilde R_n(-t)\mu
=  \tilde C_n(-t) \tilde R_n(-t) R_n(t) C_n(t) \lambda.
\]
Since $ \tilde R_n(-t) R_n(t)   =  \tilde C_n(-t)  C_n(t)  = 1$, we have $\cqq(-t) \cq(t) \lambda = \lambda$, and a similar calculation shows $\cq(t) \cqq(-t) \lambda = \lambda$.
Since $n$ can be arbitrary, so that $\lambda$ can also be arbitrary, we conclude that $\cq(-t)\cqq(t) = \cqq(t)\cq(-t) = 1$.
The identity  $\cp(-t)\cpp(t) = \cpp(t)\cp(-t) = 1$
follows by a similar argument since
$
\cp(t) \lambda =  R_n(t) B_n(t) \lambda
$
and
$
\cpp(t) \lambda =  \tilde B_n(t) \tilde R_n(t) \lambda
$
when $\SD_\lambda \subseteq [n]\times[n]$.
\end{proof}

Fix $u,v\in \ZZ[\beta]\llbracket x_1,x_2,\dots\rrbracket$ with 
$\mindeg(u)=\mindeg(v)=1$
  and recall that $\overline{x} := \frac{-x}{1+\beta x}$.

\begin{lemma}\label{gg-lem2}
It holds that $\cpp(v)\cQ(u) =\tfrac{1+uv}{1+\overline{u}v} \cQ(u)\cpp(v) $
and $\cqq(v)\cP(u) =\tfrac{1+uv}{1+\overline{u}v} \cP(u)\cqq(v) $.
\end{lemma}

\begin{proof}
For any strict partitions $\mu$ and $\nu$, the Cauchy identity
\eqref{cauchy-eq5} implies that
\[ \langle \cpp(u)\cQ(v) \mu, \nu\rangle =  \sum_\lambda \GQ_{\lambda\ss\mu}(u) \jp_{\lambda/\nu}(v)
=\tfrac{1+uv}{1+\overline{u}v} \sum_\kappa \GQ_{\nu\ss\kappa}(u) \jp_{\mu/\kappa}(v)
\]
(after setting $x_1 = u$, $y_1 = v$, and all other variables to $0$).
The first identity holds since we have
$\langle \cQ(v) \cpp(u)\mu, \nu\rangle = \sum_\kappa \GQ_{\nu\ss\kappa}(u) \jp_{\mu/\kappa}(v)$. 
The second identity follows similarly using \eqref{cauchy-eq4}.
\end{proof}

\begin{corollary}\label{gg-cor}
It holds that $\cp(v)\cQ(u) =\tfrac{1-\overline{u}v} {1-uv} \cQ(u)\cp(v) $
and $\cq(v)\cP(u) =\tfrac{1-\overline{u}v} {1-uv} \cP(u)\cq(v) $.
\end{corollary}

\begin{proof}
Consider the first identity in Lemma~\ref{gg-lem2}.
Multiplying both sides on the left and right by $\cp(-v)$ gives
$\cQ(u)\cp(-v) =\tfrac{1+uv}{1+\overline{u}v} \cp(-v) \cQ(u)$
by Lemma~\ref{gg-lem1}.
Thus $ \cp(-v) \cQ(u) = \tfrac{1+\overline{u}v} {1+uv} \cQ(u)\cp(-v)$,
and this becomes the first desired identity on negating $v$.
The other identity follows similarly.
\end{proof}

Putting together all of these identities lets us prove our last main theorem.

\begin{theorem}\label{gp-gq-thm}
If $\mu$ and $ \lambda$ are strict partitions then $ \gp_{\lambda/\mu}=\wtgp_{\lambda/\mu}$
and  $ \gq_{\lambda/\mu}=\wtgq_{\lambda/\mu}$.
\end{theorem}

\begin{proof}
Fix an integer $n>0$.
Then one has
$
\langle  \cQ(x_n)\cdots \cQ(x_2)\cQ(x_1) \emptyset, \lambda \rangle = \GQ_{\lambda}(x_1,x_2,\dots,x_n)
$
and
$
\langle  \cp(x_1)\cp(x_2)\cdots \cp(x_n) \lambda, \emptyset \rangle = \wtgp_{\lambda}(x_1,x_2,\dots,x_n)
$
for any strict partition $\lambda$, so
\[\ba
\sum_\lambda \GQ_\lambda(x_1,\dots,x_n) \wtgp_\lambda(y_1,\dots,y_n)
&=
\langle  \cp(y_1) \cp(y_2)\cdots  \cp(y_n) \cQ(x_n)\cdots\cQ(x_2)\cQ(x_1) \emptyset, \emptyset \rangle
\\&=
\prod_{1\leq i,j\leq n} \tfrac{1-\overline{x_i}y_j}{1-x_iy_j} 
\langle  \cQ(x_n)  \cdots \cQ(x_1) \cp(y_1)\cdots \cp(y_n) \emptyset, \emptyset \rangle
\ea
\]
by Corollary~\ref{gg-cor}.
As   $\cp(t) \emptyset = \emptyset$,  $\langle \cQ(t) \emptyset,\emptyset\rangle = 1$, and $\langle \cQ(t) \mu,\emptyset\rangle = 0$ for $\mu \neq \emptyset$,
we get
\[\sum_\lambda \GQ_\lambda({\bf x}) \wtgp_\lambda({\bf y}) = \prod_{i,j\geq 1} \tfrac{1-\overline{x_i}y_j}{1-x_iy_j}\]
by taking   $n\to \infty$.
This is the   Cauchy identity \eqref{cauchy-eq} defining $\gp_\lambda$, so  
$\gp_\lambda = \wtgp_\lambda$.
In particular, this means that $\wtgp_\lambda$ is symmetric.
It therefore follows from the definition of $\wtgp_{\lambda/\mu}$ that 
$\wtgp_\lambda({\bf x},{\bf y}) = \sum_{\mu }\wtgp_\mu({\bf x}) \wtgp_{\lambda/\mu}({\bf y})$.
On comparing this with \eqref{gp-xy-eq}
we conclude that $\gp_{\lambda/\mu} = \wtgp_{\lambda/\mu}$ for all  $\lambda$ and $\mu$.
A similar argument using $\cq(t)$ and $\cP(t)$ 
shows that $\gq_{\lambda/\mu} = \wtgq_{\lambda/\mu}$.
\end{proof}

\end{document}